\documentclass[11pt,letterpaper,reqno]{amsart}
\usepackage{color}
\usepackage[latin1]{inputenc}
\usepackage{amsmath}
\usepackage{amsfonts}
\usepackage{amsthm}
\usepackage{epsfig}
\usepackage{graphicx}
\usepackage{amssymb}
\usepackage{esint}
\usepackage{caption}

\newcommand{\F}{\mathcal F}

\renewcommand{\H}{\mathcal{H}}
\renewcommand{\L}{\mathcal{L}}
\newcommand{\M}{\mathcal{M}}
\renewcommand{\P}{\mathcal{P}}

\newcommand{\X}{\mathcal{X}}
\newcommand{\V}{\mathcal{V}}

\newcommand{\mres}{\mathbin{\vrule height 1.6ex depth 0pt width 
0.13ex\vrule height 0.13ex depth 0pt width 1.3ex}}

\newcommand{\hd}{\mathrm{hd}\,}
\newcommand{\N}{\mathbb{N}}

\newcommand{\R}{\mathbb{R}}
\renewcommand{\SS}{\mathbb{S}}

\newcommand{\Om}{\Omega}
\renewcommand{\S}{\Sigma}
\renewcommand{\a}{\alpha}
\renewcommand{\b}{\beta}
\newcommand{\g}{\gamma}
\newcommand{\de}{\delta}
\newcommand{\e}{\varepsilon}

\renewcommand{\l}{\lambda}
\newcommand{\s}{\sigma}

\newcommand{\om}{\omega}
\newcommand{\vphi}{\varphi}
\newcommand{\Lip}{{\rm Lip}}

\newcommand{\Div}{{\rm div}\,}
\newcommand{\Id}{{\rm Id}\,}
\newcommand{\dist}{{\rm dist}}

\newcommand{\diam}{{\rm diam}\,}

\newcommand{\spt}{{\rm spt}}

\newcommand{\bd}{{\rm bdry}\,}

\newcommand{\weak}{\rightharpoonup}

\newcommand{\toloc}{\stackrel{{\rm loc}}{\to}}
\newcommand{\ov}{\overline}

\newcommand{\pa}{\partial}

\newcommand{\cc}{\subset\subset}

\newcommand{\cl}{\mathrm{cl}\,}

\newcommand{\p}{\mathbf{p}}
\newcommand{\q}{\mathbf{q}}

\newcommand{\pp}{\mathbf{p}}

\newcommand{\var}{\mathbf{var}\,}

\newcommand{\Rr}{\mathcal{R}}
\newcommand{\Ss}{\mathcal{S}}
\newcommand{\CC}{\textbf{\textup{C}}}
\newcommand{\DD}{\textbf{\textup{D}}}
\newcommand{\hn}{\mathcal{H}^n}
\newcommand{\atan}{{\rm arctn}}

\theoremstyle{plain}
\newtheorem{theorem}{Theorem}[section]
\newtheorem{lemma}[theorem]{Lemma}

\newtheorem{proposition}[theorem]{Proposition}
\newtheorem*{theorem*}{Theorem}
\newtheorem*{corollary*}{Corollary}

\theoremstyle{definition}

\newtheorem{remark}[theorem]{Remark}

\newtheorem*{notation*}{Notation}

\numberwithin{equation}{section}
\numberwithin{figure}{section}
\newcommand{\id}{{\rm id}\,}

\title[Isoperimetric residues]{Isoperimetric residues and a mesoscale flatness criterion for hypersurfaces with bounded mean curvature}

\author{Francesco Maggi}
\address{Department of Mathematics, The University of Texas at Austin, 2515 Speedway, Stop C1200, Austin TX 78712-1202, United States of America}
\email{maggi@math.utexas.edu}
\author{Michael Novack}
\email{michael.novack@austin.texas.edu}

\begin{document}

\begin{abstract} We obtain a full resolution result for minimizers in the exterior isoperimetric problem with respect to a compact obstacle in the large volume regime $v\to\infty$. This is achieved by the study of a Plateau-type problem with free boundary (both on the compact obstacle and at infinity) which is used to identify the first obstacle-dependent term (called {\it isoperimetric residue}) in the energy expansion, as $v\to\infty$, of the exterior isoperimetric problem. A crucial tool in the analysis of isoperimetric residues is a new ``mesoscale flatness criterion'' for hypersurfaces with bounded mean curvature, which we obtain as a development of ideas originating in the theory of minimal surfaces with isolated singularities.
\end{abstract}


\maketitle

\tableofcontents

\section{Introduction} \subsection{Overview} Given a compact set $W\subset\R^{n+1}$ ($n\ge 1$), we consider the classical {\bf exterior isoperimetric problem} associated to $W$, namely
\begin{equation}
  \label{psiv}
  \psi_W(v)=\inf\big\{P(E;\Om):E\subset\Om=\R^{n+1}\setminus W\,,|E|=v\big\}\,,\qquad v>0\,,
\end{equation}
in the large volume regime $v\to\infty$. Here $|E|$ denotes the volume (Lebesgue measure) of $E$, and $P(E;\Om)$ the (distributional) perimeter of $E$ relative to $\Om$, so that $P(E;\Om)=\H^n(\Om\cap\pa E)$ whenever $\pa E$ is locally Lipschitz. Relative isoperimetric problems are well-known for their analytical \cite[Sections 6.4-6.6]{mazyaBOOKSobolevPDE} and geometric \cite[Chapter V]{ChavelBOOK} relevance.
They are also important in physical applications: beyond the obvious example of capillarity theory \cite{FinnBOOK}, exterior isoperimetry at large volumes provides an elegant approach to the Huisken-Yau theorem in general relativity, see \cite{eichmair_metzgerINV}.

When $v\to\infty$, we expect minimizers $E_v$ in \eqref{psiv} to closely resemble balls of volume $v$. Indeed, by minimality and isoperimetry, denoting by $B^{(v)}(x)$ the ball of center $x$ and volume $v$, and with $B^{(v)}=B^{(v)}(0)$, we find that
\begin{equation}
  \label{basic energy estimate}
\lim_{v\to\infty}\frac{\psi_W(v)}{P(B^{(v)})}=1\,.
\end{equation}
Additional information can be obtained by combining \eqref{basic energy estimate} with quantitative isoperimetry \cite{fuscomaggipratelli,FigalliMaggiPratelliINVENTIONES}: if $0<|E|<\infty$, then
\begin{equation}
  \label{quantitative euclidean isop}
  P(E)\ge P(B^{(|E|)})\Big\{1+ c(n)\,\inf_{x\in\R^{n+1}}\,\Big(\frac{|E\Delta B^{(|E|)}(x)|}{|E|}\Big)^2\Big\}\,.
\end{equation}
The combination of \eqref{basic energy estimate} and \eqref{quantitative euclidean isop} shows that minimizers $E_v$ in $\psi_W(v)$ are close in $L^1$-distance to balls. Based on that, a somehow classical argument exploiting the local regularity theory of perimeter minimizers shows the existence of $v_0>0$ and of a function $R_0(v)\to 0^+$, $R_0(v)\,v^{1/(n+1)}\to\infty$ as $v\to\infty$, both depending on $W$, such that,
\begin{figure}
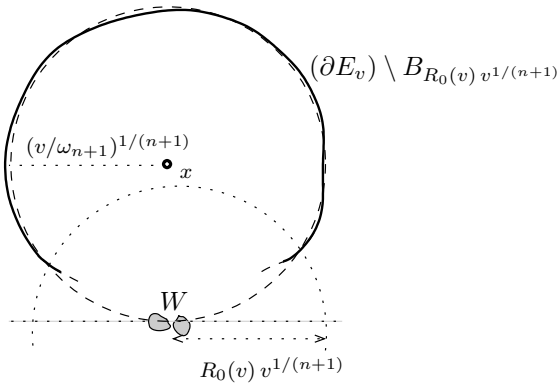
\caption{\small{Quantitative isoperimetry gives no information on how $W$ affects $\psi_W(v)$ for $v$ large.
}}
\label{fig noinfo}
\end{figure}
if $E_v$ is a minimizer of \eqref{psiv} with $v>v_0$, then (see Figure \ref{fig noinfo})
\begin{eqnarray}
\label{basic C1 estimate}
  &&\mbox{$(\pa E_v)\setminus B_{R_0\,v^{1/(n+1)}}\subset$ a $C^1$-small normal graph over $\pa B^{(v)}(x)$},
  \\\nonumber
  &&\mbox{for some $x\in\R^{n+1}$ with $|x|=(v/\om_{n+1})^{1/(n+1)}+{\rm o}(v^{1/(n+1)})$ as $v\to\infty$}\,;
\end{eqnarray}
here $\om_m$ stands for the volume of the unit ball in $\R^m$, $B_r(x)$ is the ball of center $x$ and radius $r$ in $\R^{n+1}$, and $B_r=B_r(0)$. The picture of the situation offered by \eqref{basic energy estimate} and \eqref{basic C1 estimate} is thus incomplete under one important aspect: it offers no information related to the specific ``obstacle'' $W$ under consideration -- in other words, {\it two different obstacles are completely unrecognizable from \eqref{basic energy estimate} and \eqref{basic C1 estimate} alone}.

The first step to obtain obstacle-dependent information on $\psi_W$ is studying $L^1_{\rm loc}$-subsequential limits $F$ of exterior isoperimetric sets $E_v$ as $v\to\infty$. Since the mean curvature of $\pa E_v$ has order $v^{-1/(n+1)}$ as $v\to\infty$ in $\Om$, each $\pa F$ is easily seen to be a minimal surface in $\Om$. A finer analysis leads to establish a more useful characterization of such limits $F$ as minimizers in a ``Plateau's problem with free boundary on the obstacle and at infinity'', whose negative is precisely defined in \eqref{def RW} below and denoted by $\Rr(W)$. We call $\Rr(W)$ the {\bf isoperimetric residue of $W$} because it captures the ``residual effect'' of $W$ in \eqref{basic energy estimate}, as expressed by the limit identity
\begin{equation}
  \label{main energy estimate}
\lim_{v\to\infty}\,\psi_W(v)-P(B^{(v)})=-\Rr(W)\,.
\end{equation}
The study of the geometric information about $W$ stored in $\Rr(W)$ is particularly interesting: roughly, $\Rr(W)$ is close to an $n$-dimensional sectional area of $W$, although its precise value is elusively determined by the behavior of certain ``plane-like'' minimal surfaces with free boundary on $W$. The proof of \eqref{main energy estimate} itself requires proving a blowdown result for such exterior minimal surfaces, and then extracting sharp decay information towards hyperplane blowdown limits. In particular, in the process of proving \eqref{main energy estimate}, we shall prove the existence of a positive $R_2$ (depending on $n$ and $W$ only) such that for every maximizer $F$ of $\Rr(W)$, $(\pa F)\setminus B_{R_2}$ is the graph of a smooth solution to the minimal surfaces equation. An application of Allard's regularity theorem \cite{Allard} leads then to complement \eqref{basic C1 estimate} with the following ``local'' resolution formula: for every $S>R_2$, for $v$ is large in terms of $n$, $W$ and $S$,
\begin{eqnarray}
\label{local C1 estimate}
  &&\mbox{$(\pa E_v)\cap \big(B_S\setminus B_{R_2}\big)\subset$ a $C^1$-small normal graph over $\pa F$},
  \\
  \nonumber
  &&\mbox{where $F$ is optimal for  the isoperimetric residue $\Rr(W)$ of $W$}\,.
\end{eqnarray}
Interestingly, this already fine analysis gives no information on $\pa E_v$ in the {\it mesoscale} region $B_{R_0(v)\,v^{1/(n+1)}}\setminus B_S$ between the  resolution formulas \eqref{basic C1 estimate} and \eqref{local C1 estimate}. To address this issue, we are compelled to develop what we have called a {\bf mesoscale flatness criterion} for hypersurfaces with bounded mean curvature. This kind of statement is qualitatively novel with respect to the flatness criteria typically used in the study of blowups and blowdowns of minimal surfaces -- although it is clearly related to those tools at the mere technical level -- and holds promise for applications to other geometric variational problems. In the study of the exterior isoperimetric problem, it allows us to prove the existence of positive constants $v_0$ and $R_1$, depending on $n$ and $W$ only, such that if $v>v_0$ and $E_v$ is a minimizer of $\psi_W(v)$, then
\begin{eqnarray}  \nonumber
  &&\mbox{$(\pa E_v)\cap \big(B_{R_1\,v^{1/(n+1)}}\setminus B_{R_2}\big)\subset$ a $C^1$-small normal graph over $\pa F$},
  \\\label{main C1 estimate}
  &&\mbox{where $F$ is optimal for the isoperimetric residue $\Rr(W)$ of $W$}\,.
\end{eqnarray}
The key difference between \eqref{local C1 estimate} and \eqref{main C1 estimate} is that the domain of resolution given in \eqref{main C1 estimate} {\it overlaps} with that of \eqref{basic C1 estimate}: indeed, $R_0(v)\to0^+$ as $v\to\infty$ implies that $R_0\,v^{1/(n+1)}< R_1\,v^{1/(n+1)}$ for $v>v_0$. As a by-product of this overlapping and of the graphicality of $\pa F$ outside of $B_{R_2}$, we deduce that {\it boundaries of exterior isoperimetric sets, outside of $B_{R_2}$, are diffeomorphic to $n$-dimensional disks}. Finally, when $n\le 6$, and maximizers $F$ of $\Rr(W)$ have locally smooth boundaries in $\Om$, \eqref{main C1 estimate} can be propagated up to the obstacle itself; see Remark \ref{remark up to the obstacle} below.

Concerning the rest of this introduction: In section \ref{subsection isoperimetric residues} we present our analysis of isoperimetric residues, see Theorem \ref{thm main of residue}. In section \ref{subsection resolution of isop sets} we gather all our results concerning exterior isoperimetric sets with large volumes, see Theorem \ref{thm main psi}. Finally, we present the mesoscale flatness criterion in section \ref{subsection mesoscale flatness criterion intro} and the organization of the paper in section \ref{subsection organization}.

\subsection{Isoperimetric residues}\label{subsection isoperimetric residues} To define $\Rr(W)$ we introduce the class
\[
\F
\]
of those pairs $(F,\nu)$ with $\nu\in\SS^n$ ($=$ the unit sphere of $\R^{n+1}$) and $F\subset\R^{n+1}$ a set of locally finite perimeter in $\Om$ (i.e., $P(F;\Om')<\infty$ for every $\Om'\cc\Om$), contained in slab around $\nu^\perp=\{x:x\cdot\nu=0\}$, and whose boundary (see Remark \ref{remark boundaries in RW} below) has full projection over $\nu^\perp$ itself: i.e., for some $\a,\b\in\R$,
\begin{eqnarray}\label{def Sigma nu 1}
  &&\pa F\subset \big\{x:\a< x\cdot\nu<\b\big\}\,,
  \\\label{def Sigma nu 2}
  &&\pp_{\nu^\perp}(\pa F)=\nu^\perp:=\big\{x:x\cdot\nu=0\big\}\,,
\end{eqnarray}
where $\pp_{\nu^\perp}(x)=x-(x\cdot\nu)\,\nu$, $x\in\R^{n+1}$. In correspondence to $W$ compact, we define the {\bf residual perimeter functional}, ${\rm res}_W:\F\to\R\cup\{\pm\infty\}$, by
\[
{\rm res}_W(F,\nu)=\varlimsup_{R\to\infty}\om_n\,R^n-P(F;\CC_R^\nu\setminus W)\,,\qquad (F,\nu)\in\F\,,
\]
where $\CC_R^\nu=\{x\in\R^{n+1}:|\pp_{\nu^\perp}(x)|<R\}$ denotes the (unbounded) cylinder of radius $R$ with axis along $\nu$ -- and where the limsup is actually a monotone decreasing limit thanks to \eqref{def Sigma nu 1} and \eqref{def Sigma nu 2} (see \eqref{perimeter is decreasing} below for a proof). For a reasonably ``well-behaved'' $F$, e.g. if $\pa F$ is the graph of a Lipschitz function over $\nu^\perp$, $\om_n\,R^n$ is the (obstacle-independent) leading order term of the expansion of $P(F;\CC_R^\nu\setminus W)$ as $R\to\infty$, while ${\rm res}_W(F,\nu)$ is expected to capture the first obstacle-dependent ``residual perimeter'' contribution of $P(F;\CC_R^\nu\setminus W)$ as $R\to\infty$.  The {\bf isoperimetric residue} of $W$ is then defined by maximizing ${\rm res}_W$ over $\F$, so that
\begin{equation}
  \label{def RW}
  \Rr(W)=\sup_{(F,\nu)\in\F}\,{\rm res}_W(F,\nu)\,;
\end{equation}
see
\begin{figure}
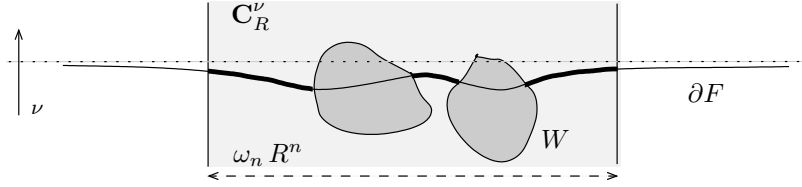\caption{\small{If $(F,\nu)\in\F$ then $F$ is contained in a slab around $\nu^\perp$ and is such that $\pa F$ has full projection over $\nu^\perp$. Only the behavior of $\pa F$ outside $W$ matters in computing ${\rm res}_W(F,\nu)$. The perimeter of $F$ in $\mathbf{C}_R^\nu\setminus W$ (depicted as a bold line) is compared to $\om_n\,R^n(=$perimeter of a half-space orthogonal to $\nu$ in $\mathbf{C}_R^\nu$); the corresponding ``residual'' perimeter as $R\to\infty$, is ${\rm res}_W(F,\nu)$.}}\label{fig fnu}
\end{figure}
Figure \ref{fig fnu}. Clearly $\Rr(\l\,W)=\l^n\,\Rr(W)$ if $\l>0$, and $\Rr(W)$ is trapped between the areas of the largest hyperplane section and directional projection of $W$, see \eqref{R larger than S} below. In the simple case when $n=1$ and $W$ is connected, $\Rr(W)=\diam(W)$ by \eqref{characterization 1} and \eqref{characterization 3} below, although, in general, $\Rr(W)$ does not seem to admit a simple characterization, and it is finely tuned to the near-to-the-obstacle behavior of ``plane-like'' minimal surfaces with free boundary on $W$. Our first main result collects these (and other) properties of isoperimetric residues and of their maximizers.

\begin{theorem}[Isoperimetric residues]\label{thm main of residue} If $W\subset\R^{n+1}$ is compact, then there are $R_2$ and $C_0$ positive and depending on $W$ with the following property.

\noindent {\bf (i):} If $\Ss(W)=\sup\{\H^n(W\cap\Pi):\mbox{$\Pi$ is a hyperplane in $\R^{n+1}$}\}$ and $\P(W)=\sup\{\H^n(\pp_{\nu^{\perp}}(W)):\nu\in\SS^n\}$, then we have
\begin{eqnarray}
\label{R larger than S}
\Ss(W)\le \Rr(W)\le \P(W)\,.
\end{eqnarray}

\noindent {\bf (ii):} The family ${\rm Max}[\Rr(W)]$ of maximizers of $\Rr(W)$ is non-empty. If $(F,\nu)\in{\rm Max}[\Rr(W)]$, then $F$ is a {\bf perimeter minimizer with free boundary in $\Om=\R^{n+1}\setminus W$}, i.e.
\begin{equation}\label{local perimeter minimizer}
P(F;\Omega \cap B) \leq P(G;\Omega \cap B)\,,\qquad\mbox{$\forall F\Delta G\cc B$, $B$ a ball}\,;
\end{equation}
and if $\Rr(W)>0$, then $\pa F$ is contained in the smallest slab $\{x:\a\le x\cdot\nu\le\b\}$ containing $W$, and there are $a,b\in\R$, $c\in\nu^\perp$ with $\max\{|a|,|b|,|c|\}\le C_0$ and $f\in C^\infty(\nu^\perp)$ such that
\begin{equation}
  \label{main residue graphicality of F}
  \,\,\,\,\,\,\,\,\,\,\,\,\,\,\,\,\,(\partial F) \setminus \CC^\nu_{R_2}=\big\{x+f(x)\,\nu:x\in\nu^\perp\,,|x|>R_2\big\}\,,
\end{equation}
\begin{eqnarray}\nonumber
  &&f(x)=a\,,\hspace{5.4cm} (n=1)
  \\\label{asymptotics of F}
  &&\Big|f(x)-\Big(a+\frac{b}{|x|^{n-2}}+\frac{c\cdot x}{|x|^n}\Big)\Big|\le\frac{C_0}{|x|^n}\,,\qquad (n\ge 2)
  \\\nonumber
  &&\max\big\{|x|^{n-1}\,|\nabla f(x)|,|x|^n\,|\nabla^2f(x)|\big\}\le C_0\,,\qquad\forall x\in\nu^\perp\,,|x|>R_2\,.
\end{eqnarray}

\noindent {\bf (iii):} At fixed diameter, isoperimetric residues are maximized by balls, i.e.
\begin{equation}\label{optimal RW}
\Rr(W) \leq \omega_n \big(\diam W/2\big)^n\,= \Rr\big({\rm cl}\big(B_{\diam W /2}\big)\big)\,,
\end{equation}
where $\cl(X)$ denotes topological closure of $X\subset\R^{n+1}$. Moreover, if equality holds in \eqref{optimal RW} and $(F,\nu)\in{\rm Max}[\Rr(W)]$, then \eqref{asymptotics of F} holds with $b=0$ and $c=0$, and setting $\Pi=\big\{y:y\cdot\nu=a\big\}$, we have
\begin{equation}
  \label{characterization 2}
  (\partial F)\setminus W=\Pi \setminus\cl\big(B_{\diam W/2}(x)\big)\,,
\end{equation}
for some $x\in\Pi$. Finally, equality holds in \eqref{optimal RW} if and only if there are a hyperplane $\Pi$ and a point $x\in\Pi$ such that
\begin{eqnarray}\label{characterization 1}
  \partial B_{\diam W/2}(x) \cap \Pi\subset W\,,
\end{eqnarray}
i.e., $W$ contains an $(n-1)$-dimensional sphere of diameter $\diam(W)$, and
\begin{eqnarray}
\label{characterization 3}
&&\mbox{$\Om\setminus \big(\Pi \setminus \cl\big(B_{\diam W/2}(x)\big)\big)$}
\\\nonumber
&&\mbox{has exactly two unbounded connected components}.
\end{eqnarray}
\end{theorem}

\begin{remark}
  {\rm The assumption $\Rr(W)>0$ is quite weak: indeed, {\bf if $\Rr(W)=0$, then $W$ is purely $\H^n$-unrectifiable}; see Proposition \ref{prop RW zero} in the appendix. For the role of the topological condition \eqref{characterization 3}, see Figure
  \begin{figure}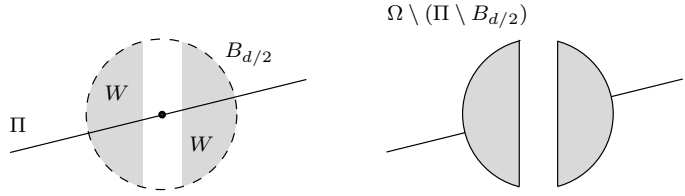\caption{\small{The obstacle $W$ (depicted in grey) is obtained by removing a cylinder $\CC_r^{e_{n+1}}$ from a ball $B_{d/2}$ with $d/2>r$. In this way $d=\diam(W)$ and $B_{d/2}$ is the only ball such that \eqref{characterization 1} can hold. Hyperplanes $\Pi$ satisfying \eqref{characterization 1} are exactly those passing through the center of $B_{d/2}$, and intersecting $W$ on a $(n-1)$-dimensional sphere of radius $d/2$. For every such $\Pi$, $\Om\setminus(\Pi\setminus B_{d/2})$ has exactly one unbounded connected component, and  \eqref{characterization 3} does not hold.}}\label{fig second}\end{figure}
  \ref{fig second}.}
\end{remark}

\begin{remark}[Regularity of isoperimetric residues]\label{remark regularity}
  {\rm In the physical dimension $n=2$, and provided $\Om$ has boundary of class $C^{1,1}$, maximizers of $\Rr(W)$ are $C^{1,1/2}$-regular up to the obstacle, and smooth away from it. More generally, condition \eqref{local perimeter minimizer} implies that $M=\cl(\Om\cap\pa F)$ is a smooth hypersurface with boundary in $\Omega\setminus\Sigma$, where $\Sigma$ is a closed set such that $\Sigma\cap\Om$ is empty if $1\le n\le 6$, is locally discrete in $\Om$ if $n=7$, and is locally $\H^{n-7}$-rectifiable in $\Om$ if $n\ge 8$; see, e.g. \cite[Part III]{maggiBOOK}, \cite{nabervaltortaJEMS}. Of course, by \eqref{main residue graphicality of F}, $\Sigma\setminus B_{R_2}=\emptyset$ in every dimension.  Moreover, justifying the initial claim concerning the case $n=2$, if we assume that $\Om$ is an open set with $C^{1,1}$-boundary, then $M$ is a $C^{1,1/2}$-hypersurface with boundary in $\R^{n+1}\setminus\Sigma$, with boundary contained in $\pa\Om$, $\Sigma\cap\pa\Om$ is $\H^{n-3+\e}$-negligible for every $\e>0$, and Young's law $\nu_F\cdot\nu_\Om=0$ holds on $(M\cap\pa\Om)\setminus\Sigma$; see, e.g. \cite{gruter,gruterjost,dephilippismaggiCAP-ARMA,dephilippismaggiCAP-CRELLE}.}
\end{remark}

\begin{remark}
  {\rm An interesting open direction is finding additional geometric information on $\Rr(W)$, e.g. in the class of convex obstacles.}
\end{remark}

\begin{remark}[Normalization of competitors]\label{remark boundaries in RW}
  {\rm We adopt the convention that any set of locally finite perimeter $F$ in $\Om$ open is tacitly modified on and by a set of zero Lebesgue measure so to entail $\Om\cap\pa F=\Om\cap\cl(\pa^*F)$, where $\pa^*F$ is the reduced boundary of $F$ in $\Om$; see \cite[Proposition 12.19]{maggiBOOK}. Under this normalization, local perimeter minimality conditions like \eqref{local perimeter minimizer} (or \eqref{uniform lambda minimality} below) imply that $F\cap\Om$ is open in $\R^{n+1}$; see, e.g. \cite[Lemma 2.16]{dephilippismaggiCAP-ARMA}.}
\end{remark}

\subsection{Resolution of exterior isoperimetric sets}\label{subsection resolution of isop sets} Denoting the family of minimizers of $\psi_W(v)$ by ${\rm Min}[\psi_W(v)]$, our second main result is as follows:

\begin{theorem}[Resolution of exterior isoperimetric sets]\label{thm main psi}
If $W\subset\R^{n+1}$ is compact, then ${\rm Min}[\psi_W(v)]\ne\emptyset\,\,\forall v>0$. Moreover, if $\Rr(W)>0$, then
\begin{equation}
  \label{main asymptotic expansion}
  \lim_{v\to\infty}\psi_W(v)-P(B^{(v)})=-\Rr(W)\,,
\end{equation}
and, depending on $n$ and $W$ only, there are $v_0$, $C_0$, $R_1$, and $R_2$ positive, and $R_0(v)$ with $R_0(v)\to 0^+$, $R_0(v)\,v^{1/(n+1)}\to\infty$ as $v\to\infty$, such that, if $E_v\in {\rm Min}[\psi_W(v)]$ and $v>v_0$, then:

\noindent {\bf (i)}: $E_v$ determines $x\in\R^{n+1}$ and $u\in C^\infty(\pa B^{(1)})$ such that
\begin{eqnarray}
  \label{nonsharp rates}
  &&\frac{|E_v\Delta B^{(v)}(x)|}v \le \frac{C_0}{v^{1/[2(n+1)]}}\,,
\\\label{x and u of Ev}
 &&(\pa E_v)\setminus B_{R_0(v)\,v^{1/(n+1)}}
 \\\nonumber
 &&=\Big\{y+v^{1/(n+1)}\,u\Big(\frac{y-x}{v^{1/(n+1)}}\Big)\,\nu_{B^{(v)}(x)}(y):y\in\pa B^{(v)}(x)\Big\}\setminus B_{R_0(v)\,v^{1/(n+1)}}\,,
\end{eqnarray}
where, for any $G\subset\R^{n+1}$ with locally finite perimeter, $\nu_G$ is the outer unit normal to $G$;

\noindent {\bf (ii):} $E_v$ determines $(F,\nu)\in{\rm Max}[\Rr(W)]$ and $f\in C^\infty((\pa F)\setminus B_{R_2})$ with
\begin{equation}
  \label{f of Ev}
  (\pa E_v)\cap A_{R_2}^{R_1\,v^{1/(n+1)}}=\big\{y+f(y)\,\nu_F(y):y\in\pa F\big\}\cap A_{R_2}^{R_1\,v^{1/(n+1)}}\,;
\end{equation}

\noindent {\bf (iii):} $(\pa E_v)\setminus B_{R_2}$ is diffeomorphic to an $n$-dimensional disk;

\noindent {\bf (iv):} Finally, with $(x,u)$ as in \eqref{x and u of Ev} and $(F,\nu,f)$ as in \eqref{f of Ev},
\begin{eqnarray*}
  &&\lim_{v\to\infty}\sup_{E_v\in
  {\rm Min}[\psi_W(v)]}\Big\{\Big|\frac{|x|}{v^{1/(n+1)}}-\frac{1}{\omega_{n+1}^{1/(n+1)}}\Big|\,,\Big|\nu-\frac{x}{|x|}\Big|\,,
  \|u\|_{C^1(\pa B^{(1)})}\Big\}=0\,,
  \\
  &&\lim_{v\to\infty}\sup_{E_v\in{\rm Min}[\psi_W(v)]}\,\|f\|_{C^1(B_M\cap\pa F)}=0\,,\hspace{2cm}\forall M>R_2\,.
\end{eqnarray*}
\end{theorem}

\begin{remark}[Resolution up to the obstacle]\label{remark up to the obstacle}
  {\rm By Remark \ref{remark regularity} and a covering argument, if $n\le 6$, $\de>0$, and $v>v_0(n,W,\de)$, then \eqref{f of Ev} holds with $B_{R_1\,v^{1/(n+1)}}\setminus I_\de(W)$ in place of $B_{R_1\,v^{1/(n+1)}}\setminus B_{R_2}$, where $I_\de(W)$ is the open $\de$-neighborhood of $W$. Similarly, when $\pa\Om\in C^{1,1}$ and $n=2$ (and thus $\Om\cap\pa F$ is regular up to the obstacle), we can find $v_0$ (depending on $n$ and $W$ only) such that \eqref{f of Ev} holds with $B_{R_1\,v^{1/(n+1)}}\cap \Om$ in place of $B_{R_1\,v^{1/(n+1)}}\setminus B_{R_2}$, that is, graphicality over $\pa F$ holds up to the obstacle itself.}
\end{remark}

\begin{remark}
  {\rm If $W$ is convex and $J$ is an half-space, then $\psi_W(v)\ge\psi_J(v)$ for every $v>0$, with equality for $v>0$ if and only if $\pa W$ contains a flat facet supporting an half-ball of volume $v$; see \cite{choeghomiritore,fuscomoriniCONVEX}. Since $\psi_J(v)=P(B^{(v)})/2^{1/(n+1)}$ and $\psi_W(v)-P(B^{(v)})\to-\Rr(W)$ as $v\to\infty$, the bound $\psi_W(v)\ge\psi_J(v)$ is far from optimal if $v$ is large. Are there stronger global bounds than $\psi_W\ge\psi_J$ on convex obstacles? Similarly, it would be interesting to quantify the convergence towards $\Rr(W)$ in \eqref{main asymptotic expansion}, or even that of $\pa E_v$ towards $\pa B^{(v)}$ and $\pa F$ (where \eqref{nonsharp rates} should not to be sharp).}
\end{remark}


\subsection{The mesoscale flatness criterion}\label{subsection mesoscale flatness criterion intro} We work with with hypersurfaces $M$ whose mean curvature is bounded by $\Lambda\ge0$ in an annulus $B_{1/\Lambda}\setminus\ov{B}_R$, $R\in(0,1/\Lambda)$. Even without information on  $M$ inside $B_R$ (where $M$ could have a non-trivial boundary, or topology, etc.) the classical proof of the monotonicity formula can be adapted to show the monotone increasing character on  $r\in(R,1/\Lambda)$ of
\begin{eqnarray}\nonumber
\Theta_{M,R,\Lambda}(r)&=&
\frac{\H^n\big(M\cap(B_r\setminus B_R)\big)}{r^n}+\frac{R}{n\,r^n}\,\int_{M\cap\pa B_R}\frac{|x^{TM}|}{|x|}\,d\H^{n-1}
\\\label{theta smooth}
&&
+\Lambda\,\int_R^r\frac{\H^n\big(M\cap(B_\rho\setminus B_R)\big)}{\rho^n}\,d\rho\,,\,\,\,\,
\end{eqnarray}
(here $x^{TM}={\rm proj}_{T_xM}(x)$);
moreover, if $\Theta_{M,R,\Lambda}$ is constant over $(a,b)\subset(R,1/\Lambda)$, then $M\cap(B_b\setminus\ov{B}_a)$ is a cone. Since the constant density value corresponding to $M=H\setminus B_R$, $H$ an hyperplane through the origin, is $\om_n$ (as a result of a double cancellation which also involves the ``boundary term'' in $\Theta_{H\setminus B_R,R,0}$), we consider the {\bf area deficit}
\begin{equation}
  \label{def of delta}
  \de_{M,R,\Lambda}(r)=\om_n-\Theta_{M,R,\Lambda}(r)\,,\qquad r\in(R,1/\Lambda)\,,
\end{equation}
which defines a decreasing quantity on $(R,1/\Lambda)$. Here we use the term ``deficit'', rather than the more usual term ``excess'', since $\de_{M,R,\Lambda}$ does not necessarily have non-negative sign (which is one of the crucial property of ``excess quantities'' typically used in $\e$-regularity theorems, see, e.g., \cite[Lemma 22.11]{maggiBOOK}). Recalling that $A_r^s=B_s\setminus\cl(B_r)$ if $s>r>0$, we are now ready to state the following ``smooth version'' of our mesoscale flatness criterion (see Theorem \ref{theorem mesoscale criterion} below for the varifold version).

\begin{theorem}[Mesoscale flatness criterion (smooth version)]\label{theorem mesoscale smooth}
If $n\ge 2$, $\Gamma\ge 0$, and $\s>0$, then there are $M_0$ and $\e_0$ positive and depending on $n$, $\Gamma$ and $\s$ only, with the following property. Let $\Lambda\ge0$, $R\in(0,1/\Lambda)$, and $M$ be a smooth hypersurface with mean curvature bounded by $\Lambda$ in $A^{1/\Lambda}_R$, and with
  \begin{equation}
    \label{meso intro bounds Gamma}
      \H^{n-1}\big(M\cap\pa B_{R}\big)\le\Gamma\,R^{n-1}\,,
  \quad
  \sup_{\rho\in(R,1/\Lambda)}\frac{\H^n\big(M\cap(B_\rho\setminus B_R)\big)}{\rho^n}\le\Gamma\,.
  \end{equation}
  If there is $s>0$ such that
  \begin{eqnarray}\label{meso intro range of s}
  &&\max\{M_0,64\}\,R<s<\frac{\e_0}{4\,\Lambda}\,,
  \end{eqnarray}
  and
  \begin{equation}
    \label{meso intro flat hp}
      |\de_{M,R,\Lambda}(s/8)|\le \e_0\,,
  \end{equation}
  and if, setting,
  \begin{equation}
    \label{meso intro propagation}
      R_*=\sup\Big\{\rho\ge\frac{s}8: \de_{M,R,\Lambda}(\rho)\ge -\e_0\Big\}\,,\qquad S_*=\min\Big\{R_*,\frac{\e_0}{\Lambda}\Big\}\,,
  \end{equation}
  we have $R_*>4\,s$ (and thus $S_*>4\,s$), then
  \begin{eqnarray}
    \label{meso intro conclusion}
      &&M\cap A_{s/32}^{S_*/16}=\big\{x+f(x)\,\nu_K: x\in K\big\}\cap A_{s/32}^{S_*/16}\,,
      \\
      \nonumber
      &&\sup\big\{|x|^{-1}\,|f(x)|+|\nabla f(x)|:x\in K\big\}\le C(n)\,\s\,, 
  \end{eqnarray}
  for a hyperplane $K$ with $0\in K$ and unit normal $\nu_K$, and for $f\in C^1(K)$.
\end{theorem}

\begin{remark}[Structure of the statement]\label{remark structure of the statement}
  {\rm The first condition in \eqref{meso intro range of s} implicitly requires $R$ to be sufficiently small in terms of $1/\Lambda$, as it introduces a mesoscale $s$ which is both small with respect to $1/\Lambda$ and large with respect to $R$. The condition in \eqref{meso intro flat hp} express the flatness of $M$ at the mesoscale $s$ in terms of its area deficit. The final key assumption, $R_*>4\,s$, express the requirement that the area deficit does not decrease too abruptly, and stays above $-\e_0$ at least up to the scale $4\,s$. Under these assumptions, graphicality with respect to a hyperplane $K$ is inferred on an annulus whose lower radius $s/32$ has the order of the mesoscale $s$, and whose upper radius $S_*/16$ can be as large as the decay of the area deficit allows (potentially up to $\e_0/16\,\Lambda$ if $R_*=\infty$), but in any case not too large with respect to $1/\Lambda$.}
\end{remark}

\begin{remark}[Relationship to other flatness criteria]\label{remark relationship to other criteria}
If $M$ is a hypersurface containing the origin, so that, formally speaking, $R=0$, and the tangent cone of $M$ there is a plane, Theorem 1 reduces to Allard's theorem \cite{Allard}. Similarly, if $\Lambda=0$ and the exterior minimal hypersurface $M$ has a planar tangent cone at infinity, we recover the exterior blow-down results stated in  \cite{SimonMontecatini,SimonAIHP}. In particular, although the motivation for Theorem 1 comes from scenarios where both $R$ and $\Lambda$ are positive, it can also be viewed as a general framework containing as special cases the blow-up and blow-down flatness criteria for hypersurfaces with planar tangent cones.
\end{remark}

\begin{remark}[Sharpness of the statement]\label{remark unduloids}
  {\rm The statement is sharp in the sense that for a surface ``with bounded mean curvature and non-trivial topology inside a hole'', flatness can only be established on a mesoscale which is both large with respect to the size of the hole and small with respect to the size of the inverse mean curvature. An example is provided by unduloids $M_\e$ with waist size $\e$ and mean curvature $n$ in $\R^{n+1}$; see
  \begin{figure}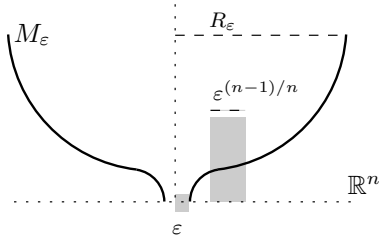
  \caption{\small{A half-period of an unduloid with mean curvature $n$ and waist size $\e$ in $\R^{n+1}$. By \eqref{def of veps}, the flatness of $M_\e$ is no smaller than ${\rm O}(\e^{2(n-1)/n})$, and is exactly ${\rm O}(\e^{2(n-1)/n})$ on an annulus sitting in the mesoscale ${\rm O}(\e^{(n-1)/n})$. This mesoscale is both very large with respect to waist size $\e$, and very small with respect to the size of the inverse mean curvature, which is order one.}}
  \label{fig undo}
  \end{figure}
  Figure \ref{fig undo}. A ``half-period'' of $M_\e$ is the graph $\{x+f_\e(x)\,e_{n+1}:x\in\R^n\,,\e<|x|<R_\e\}$ of
  \begin{equation}
    \label{def of veps}
      f_\e(x)=\int_\e^{|x|}\Big\{\Big(\frac{r^{n-1}}{r^n-\e^n+\e^{n-1}}\Big)^2-1\Big\}^{-1/2}\,dr\,,\qquad\e<|x|<R_\e\,,
  \end{equation}
  where $\e$ and $R_\e$ are the only solutions of $r^{n-1}=r^n-\e^n+\e^{n-1}$. Clearly $f_\e$ solves $-\Div(\nabla f_\e/\sqrt{1+|\nabla f_\e|^2})=n$ with $f_\e=0$, $|\nabla f_\e|=+\infty$ on $\{|x|=\e\}$, and $|\nabla f_\e|=+\infty$ on $\{|x|=R_\e\}$, where $R_\e=1-{\rm O}(\e^{n-1})$; moreover, $\min|\nabla f_\e|$ is achieved at $r={\rm O}(\e^{(n-1)/n})$, and if $r\in(a\,\e^{(n-1)/n},b\,\e^{(n-1)/n})$ for some $b>a>0$, then $|\nabla f_\e|={\rm O}_{a,b}(\e^{2(n-1)/n})$. Thus, the horizontal flatness of $M_\e$ is no smaller than ${\rm O}(\e^{2(n-1)/n})$, and has that exact order on a scale which is both very large with respect to the hole ($\e^{(n-1)/n}>\!\!\!>\e$) and very small with respect to the inverse mean curvature ($\e^{(n-1)/n}<\!\!\!<1$).}
\end{remark}

\begin{remark}[On the application to $\psi_W(v)$]
  {\rm Exterior isoperimetric sets $E_v$ with large volume $v$ have small constant mean curvature of order $\Lambda=\Lambda_0(n,W)/v^{1/(n+1)}$. We will work with ``holes'' of size $R=R_3(n,W)$, for some $R_3$ sufficiently large with respect to the radius $R_2$ appearing in Theorem \ref{thm main of residue}-(ii), and determined through the sharp decay rates \eqref{asymptotics of F}. 
  The decay properties of $F$ towards $\{x:x\cdot\nu=a\}$ when $(F,\nu)$ is a maximizer of $\Rr(W)$, the $C^1$-proximity of $\pa E$ to $\pa B^{(v)}(x)$ for $|x|\approx(\om_{n+1}/v)^{1/(n+1)}$, and the $C^1$-proximity of $\pa E$ to $\pa F$ for some optimal $(F,\nu)$ on bounded annuli of the form $A^{2\,R_3}_{R_2}$ are used in checking that \eqref{meso intro bounds Gamma} holds with $\Gamma=\Gamma(n,W)$, that $E_v$ is flat in the sense of \eqref{meso intro flat hp}, and, most importantly, that the area deficit $\de_{M,R,\Lambda}$ of $M=(\pa E_v)\setminus B_{R_3}$ lies above $-\e_0$ up to scale $r={\rm O}(v^{1/(n+1)})$ (which is the key information to deduce $R_*\approx1/\Lambda$), and thus obtain overlapping domains of resolutions in terms of $\pa B^{(v)}(x)$ and $\pa F$.}
\end{remark}

\begin{remark}
  {\rm While Theorem \ref{theorem mesoscale smooth} seems clearly applicable to other problems, there are situations where one may need to develop considerably finer ``mesoscale flatness criteria''. For example, consider the problem of ``resolving'' almost CMC boundaries undergoing bubbling \cite{ciraolomaggi2017,delgadinomaggimihailaneumayer,delgadinomaggi}. When the oscillation of the mean curvature around a constant $\Lambda$ is small, such boundaries are close to finite unions of mutually tangent spheres of radius $n/\Lambda$, and can be covered by $C^1$-small normal graphs over such spheres away from their tangency points up to distance $\e/\Lambda$, with $\e=\e(n)$, and provided the mean curvature oscillation is small in terms of $\e$. For propagating flatness up to a distance directly related to the oscillation of the mean curvature, one would need a version of Theorem \ref{theorem mesoscale smooth} for ``double'' spherical graphs; in the setting of blowup/blowdown theorems, this would be similar to passing to the harder case of multiplicity larger than one.}
\end{remark}

\begin{remark}[Comparison with blowup/blowdown results]
  {\rm From the technical viewpoint, Theorem \ref{theorem mesoscale smooth} fits into the framework set up by Allard and Almgren in \cite{allardalmgrenRADIAL} for the study of blowups and blowdowns of minimal surfaces with tangent integrable cones. At the same time, as exemplified by Remark \ref{remark unduloids}, Theorem \ref{theorem mesoscale smooth} really points in a different direction,  since it pertains to situations where neither blowup or blowdown limits make sense.
  Another interesting point is that, in \cite{allardalmgrenRADIAL}, the area deficit $\de_{M,R,\Lambda}$ is considered with a sign, non-positive for blowups, and non-negative for blowdowns, see \cite[Theorem 5.9(4), Theorem 9.6(4)]{allardalmgrenRADIAL}. A key insight here is that for hypersurfaces where the deficit changes sign, graphicality obtained through small negative (or positive) deficit nevertheless persists {\it past} the scale where $\delta_{M,R,\Lambda}$ vanishes, and possibly much farther depending on the surface in question; this is actually {\it crucial} for obtaining overlapping domains of resolutions in statements like \eqref{basic C1 estimate} and \eqref{main C1 estimate}.}
\end{remark}

\begin{remark}[Extension to general minimal cones]\label{remark extensions}
  {\rm Proving Theorem \ref{theorem mesoscale smooth} in higher codimension and with arbitrary {\it integrable} minimal cones should be possible with essentially the same proof presented here. We do not pursue this extension because, first, only the case of hypersurfaces and hyperplanes is needed in studying $\psi_W(v)$; and, second, in going for generality, one should work in the framework set up by Simon in \cite{Simon83,SimonMontecatini,simonETH}, which, at variance with the simpler Allard--Almgren's framework used here, allows one to dispense with the integrability assumption. In this direction, we notice that Theorem \ref{theorem mesoscale smooth} with $\Lambda=0$ and $R_*=+\infty$ is a blowdown result for exterior minimal surfaces (see also Theorem \ref{theorem mesoscale criterion}-(ii, iii)). A blowdown result for exterior minimal surfaces is outside the scope of \cite[Theorem 9.6]{allardalmgrenRADIAL} which pertains to {\it entire} minimal surfaces, but it is claimed, with a sketch of proof, on \cite[Page 269]{SimonMontecatini} as a modification of \cite[Theorem 5.5, $m<0$]{SimonMontecatini}. It should be mentioned that, to cover the case of exterior minimal surfaces, an additional term of the form $C\int_{\Sigma}(\dot u(t))^{2}$ should be added on the right side of  assumption \cite[5.3, $m<0$]{SimonMontecatini}. This additional term seems not to cause difficulties with the rest of the arguments leading to \cite[Theorem 5.5, $m<0$]{SimonMontecatini}. Thus Simon's approach, in addition to giving the blowdown analysis of exterior minimal surfaces, should also be viable for generalizing our mesoscale flatness criterion.}
\end{remark}

\subsection{Organization of the paper}\label{subsection organization} In section \ref{section mesoscale} we prove Theorem \ref{theorem mesoscale smooth} (actually, its generalization to varifolds, i.e. Theorem \ref{theorem mesoscale criterion}). In section \ref{section existence and quantitative isoperimetry} we prove those parts of Theorem \ref{thm main psi} which follow simply by quantitative isoperimetry (i.e., they do not require isoperimetric residues nor our mesoscale flatness analysis); see Theorem \ref{thm existence and uniform min}. Section \ref{section isoperimetric residues} is devoted to the study of isoperimetric residues and of their maximizers, and contains the proof Theorem \ref{thm main of residue}. We also present there a statement, repeatedly used in our analysis, which summarizes some results from \cite{Scho83}; see Proposition \ref{prop schoen}. Finally, in section \ref{section resolution for exterior}, we prove the energy expansion \eqref{main asymptotic expansion} and those parts of Theorem \ref{thm main psi} left out in section \ref{section existence and quantitative isoperimetry} (i.e., statements (ii, iii, iv)). This final section is, from a certain viewpoint, the most interesting part of the paper: indeed, it is only the detailed examination of those arguments that clearly illustrates the degree of fine tuning of the preliminary analysis of exterior isoperimetric sets and of maximizers of isoperimetric residues which is needed in order to allow for the application of the mesoscale flatness criterion.

\noindent {\bf Acknowledgements:} Supported by NSF-DMS RTG 1840314, NSF-DMS FRG 1854344, and NSF-DMS 2000034. We thank William Allard and Leon Simon for clarifications on \cite{allardalmgrenRADIAL} and \cite{SimonMontecatini} respectively, and Luca Spolaor for his comments on some preliminary drafts of this work.

\section{A mesoscale flatness criterion for varifolds}\label{section mesoscale} In section \ref{subsection statement meso} we introduce the class $\V_n(\Lambda,R,S)$ of varifolds used to reformulate Theorem \ref{theorem mesoscale smooth}, see Theorem \ref{theorem mesoscale criterion}. In sections \ref{subsection spherical graphs}-\ref{subsection energy estimates on annuli} we present two reparametrization lemmas (Lemma \ref{lemma step one sigma} and Lemma \ref{lemma step one}) and some ``energy estimates'' (Theorem \ref{theorem 7.15AA main estimate lambda}) for spherical graphs; in section \ref{subsection monotonicity of exterior varifolds} we state the monotonicity formula in $\V_n(\Lambda,R,S)$ and some energy estimates involving the monotonicity gap;
in section \ref{subsection mesoscale flatness criterion}, we prove Theorem \ref{theorem mesoscale criterion}.

\subsection{Statement of the criterion}\label{subsection statement meso} Given an $n$-dimensional integer rectifiable varifold $V=\var(M,\theta)$ in $\R^{n+1}$, defined by a locally $\H^n$-rectifiable set $M$, and by a multiplicity function $\theta:M\to\N$ (see \cite{SimonLN}), we denote by $\|V\|=\theta\,\H^n\llcorner M$ the weight of $V$, and by $\de V$ the first variation of $V$, so that
$\de V(X)=\int\,\Div^T\,X(x)\,dV(x,T)=\int_M\,\Div^M\,X(x)\,\theta\,d\H^n_x$ for every $X\in C^1_c(\R^{n+1};\R^{n+1})$. Given $S>R>0$ and $\Lambda\ge0$, we consider the family
\[
\V_n(\Lambda,R,S)\,,
\]
of those $n$-dimensional integral varifolds $V$ with $\spt\,V\subset\R^{n+1}\setminus B_R$ and
\begin{eqnarray*}
\de V(X)=\int\,X\cdot\vec{H}\,d\|V\|+\int X\cdot\nu_V^{\rm co}\,d \,\,{\rm bd}_V\,,\qquad \forall X\in C^1_c(B_S;\R^{n+1})\,,
\end{eqnarray*}
holds for a Radon measure ${\rm bd}_V$ in $\R^{n+1}$ supported in $\pa B_R$, and Borel vector fields $\vec{H}:\R^{n+1}\to\R^{n+1}$ with $|\vec{H}|\le \Lambda$ and $\nu_V^{\rm co}:\pa B_R\to\R^{n+1}$ with $|\nu_V^{\rm co}|=1$. We let $\M_n(\Lambda,R,S)=\{V\in\V_n(\Lambda,R,S):\mbox{$V=\var(M,1)$ for $M$ smooth}\}$, that is, $M\subset \R^{n+1}\setminus B_R$ is a smooth hypersurface with boundary in $A_R^S$, $\bd(M)\subset \pa B_R$, and $|H_M|\le\Lambda$. If $V\in\M_n(\Lambda,R,S)$, then $\vec{H}$ is the mean curvature vector of $M$, ${\rm bd}_V=\H^{n-1}\llcorner\bd(M)$, and $\nu_V^{\rm co}$ is the outer unit conormal to $M$ along $\pa B_R$. Given $V\in\V_n(\Lambda,R,S)$, we define $\Theta_{V,R,\Lambda}(r)$ as
\begin{eqnarray}\label{def theta}
\frac{\|V\|(B_r\setminus B_R)}{r^n}
  -\frac1{n\,r^n}\,\int x\cdot\nu^{{\rm co}}_V\,d\,{\rm bd}_V+\Lambda\,\int_R^r\,\frac{\|V\|(B_\rho\setminus B_R)}{\rho^n}\,d\rho\,;
\end{eqnarray}
$\Theta_{V,R,\Lambda}(r)$ is increasing for $r\in(R,S)$ (Theorem \ref{theorem 7.17AA exterior lambda}-(i) below), and equal to \eqref{theta smooth} when $V\in\M_n(\Lambda,R,S)$. The {\bf area deficit} of $V$ is then defined as in \eqref{def of delta}, while given a hyperplane $H$ in $\R^{n+1}$ with $0\in H$ we call the quantity
\begin{equation}
  \label{def of omega H}
\int_{A_r^s}\,\om_H(y)^2\,d\|V\|_y\,,\qquad \om_H(y)=\atan\Big(\frac{|y\cdot\nu_H|}{|\p_H y|}\Big)\,,
\end{equation}
the {\bf angular flatness of $V$ on the annulus $A_r^s=B_s\setminus\cl(B_r)$ with respect to $H$}. (See \eqref{seeee} for the notation concerning $H$.)

\begin{theorem}[Mesoscale flatness criterion]\label{theorem mesoscale criterion}
If $n\ge 2$, $\Gamma\ge 0$, and $\s>0$ then there are positive constants $M_0$ and $\e_0$, depending on $n$, $\Gamma$ and $\s$ only, with the following property. If $\Lambda\ge0$, $R\in(0,1/\Lambda)$, $V\in\V_n(\Lambda,R,1/\Lambda)$,
  \begin{equation}
    \label{Gamma bounds}
      \|{\rm bd}_V\|(\pa B_R)\le\Gamma\,R^{n-1}\,,\qquad \sup_{\rho\in(R,1/\Lambda)}\frac{\|V\|(B_\rho\setminus B_R)}{\rho^n}\le\Gamma\,.
  \end{equation}
  and for some $s>0$ we have
  \begin{eqnarray}
    \label{mesoscale bounds}
  &&\frac{\e_0}{4\,\Lambda}>s>\max\{M_0,64\}\,R\,,
  \\
  \label{mesoscale delta small s8}
  &&|\de_{V,R,\Lambda}(s/8)|\le \e_0\,,
  \\
  \label{mesoscale Rstar larger s4}
  &&
  R_*:=\sup\Big\{\rho\ge\frac{s}8: \de_{V,R,\Lambda}(\rho)\ge -\e_0\Big\}\ge 4\,s\,,
  \end{eqnarray}
  then {\bf (i):} if $S_*=\min\{R_*,\e_0/\Lambda\}<\infty$,
  then there is an hyperplane $K\subset\R^{n+1}$ with $0\in K$ and $u\in C^1((K\cap\SS^n)\times(s/32,S_*/16))$ with
  \begin{eqnarray}\nonumber
  &&(\spt\,V)\cap A_{s/32}^{S_*/16}=\Big\{r\,\frac{\om+u(r,\om)\,\nu_K}{\sqrt{1+u(r,\om)^2}}:\om\in K\cap\SS^n\,,r\in\big(s/32,S_*/16\big)\Big\}
  \\\label{mesoscale thesis graphicality}
  &&\sup_{(K\cap\SS^n)\times\big(s/32,S_*/16\big)}\Big\{|u|+|\nabla^{K\cap\SS^n}u|+|r\,\pa_ru|\Big\}\le C(n)\,\s\,;
  \end{eqnarray}


  \noindent{\bf (ii):} if $\Lambda=0$ and $\de_{V,R,0}\ge -\varepsilon_0$ on $(s/8,\infty)$, then $\de_{V,R,0}\ge 0$ on $(s/8,\infty)$, \eqref{mesoscale thesis graphicality} holds with $S_*=\infty$, and one has decay estimates, continuous in the radius, of the form
  \begin{eqnarray}
  \label{decay deficit for exterior minimal surfaces}
  \!\!\!\!\!\!\de_{V,R,0}(r)\!\!&\le& \!\!C(n)\,\Big(\frac{s}{r}\Big)^\a\,\de_{V,R,0}\Big(\frac{s}8\Big)\,,\qquad\forall r>\frac{s}4\,,
  \\
  \label{decay flatness for exterior minimal surfaces}
  \!\!\!\!\!\!\frac1{r^n}\,\int_{A_r^{2\,r}}\,\om_K^2\,d\|V\|\!\!&\le&\!\! C(n)\,(1+\Gamma)\,\Big(\frac{s}{r}\Big)^\a\,\de_{V,R,0}\Big(\frac{s}8\Big)\,,\qquad\forall r>\frac{s}4\,,
  \end{eqnarray}
  for some $\a(n)\in(0,1)$.
\end{theorem}

\begin{remark}
  {\rm In Theorem \ref{theorem mesoscale criterion}, graphicality is formulated in terms of the notion of {\it spherical graph} (see section \ref{subsection spherical graphs}) which is more natural than the usual notion of ``cylindrical graph'' in setting up the iteration procedure behind Theorem \ref{theorem mesoscale criterion}. Spherical graphicality in terms a $C^1$-small $u$ as in \eqref{mesoscale thesis graphicality} translates into cylindrical graphicality in terms of $f$ as in \eqref{meso intro conclusion} with
  $f(x)/|x|\approx u(|x|,\hat{x})$ and $\nabla_{\hat x} f(x)-(f(x)/|x|)\approx |x|\,\pa_r\,u(|x|,\hat{x})$  for $x\ne 0$ and $\hat{x}=x/|x|$; see, in particular, Lemma \ref{lemma D1} in appendix \ref{appendix spherical cylindrical}.}
\end{remark}

\subsection{Spherical graphs}\label{subsection spherical graphs} We start setting up some notation. We denote by
\[
\H
\]
the family of the oriented hyperplanes $H\subset\R^{n+1}$ with $0\in\H$, so that choice of $H\in\H$ implies that a unit normal vector $\nu_H$ to $H$. Given $H\in\H$, we set
\begin{equation}
  \label{seeee}
  \Sigma_H=H\cap\SS^n\,,\qquad \p_H:\R^{n+1}\to H\,,\qquad \q_H:\R^{n+1}\to H^\perp\,,
\end{equation}
for the equatorial sphere defined by $H$ on $\SS^n$ and for the orthogonal projections of $\R^{n+1}$ onto $H$ and onto $H^\perp=\{t\,\nu_H:t\in\R\}$. We set
\[
\X_\s(\Sigma_H)=\big\{u\in C^1(\S_H):\|u\|_{C^1(\S_H)}<\s\big\}\,,\qquad\s>0\,.
\]
Clearly there is $\s_0=\s_0(n)>0$ such that if $H\in\H$ and $u\in\X_{\s_0}(\S_H)$, then
\[
f_u(\om)=\frac{\om+u(\om)\,\nu_H}{\sqrt{1+u(\om)^2}}\,,\qquad \om\in\S_H\,,
\]
defines a diffeomorphism of $\S_H$ into an hypersurface $\S_H(u)\subset\SS^n$, namely
\begin{equation}
  \label{def of spherical graph}
  \S_H(u)=f_u(\Sigma_H)=\Big\{\frac{\om+u(\om)\,\nu_H}{\sqrt{1+u(\om)^2}}:\om\in\S_H\Big\}\,.
\end{equation}
We call $\S_H(u)$ a {\bf spherical graph} over $\S_H$. Exploiting the fact that $\S_H$ is a minimal hypersurface in $\SS^n$ and that if $\{\tau_i\}_i$ is a local orthonormal frame on $\S_H$ then $\nu_H\cdot\nabla_{\tau_i}\tau_j=0$, a second variation computation (see, e.g., \cite[Lemma 2.1]{EngelSpolaVelichGandT}) gives, for $u\in\X_\s(\S_H)$,
\[\Big|\H^{n-1}(\S_H(u))-n\,\om_n-\frac12\,\int_{\S_H}\!\!\!|\nabla^{\S_H} u|^2-(n-1)\,u^2\Big|\le\! C(n)\,\s\!\int_{\S_H}\!\!\!u^2+|\nabla^{\Sigma_H} u|^2\,,
\](where $n\,\om_n=\H^{n-1}(\S_H)=\H^{n-1}(\S_H(0))$). We recall that $u\in L^2(\S_H)$ is a unit norm Jacobi field of $\S_H$ (i.e., a zero eigenvector of $\Delta^{\S_H}+(n-1)\,\Id$ with unit $L^2(\S_H)$-norm) if and only if there is $\tau\in\SS^n$ with $\tau\cdot\nu_H=0$ and   $u(\om)=c_0(n)\,(\om\cdot\tau)$ ($\om\in\S_H$) for $c_0(n)=(n/\H^{n-1}(\S_H))^{1/2}$. We denote by $E^0_{\S_H}$ the orthogonal projection operator of $L^2(\Om)$ onto the span of the Jacobi fields of $\S_H$. The following lemma provides a way to reparameterize spherical graphs over equatorial spheres so that the projection over Jacobi fields is annihilated.

  \begin{lemma}\label{lemma step one sigma}
  There exist constants $C_0$, $\e_0$ and $\s_0$, depending on the dimension $n$ only, with the following properties:

  \noindent {\bf (i):} if $H,K\in\H$, $|\nu_H-\nu_K|\le\e<\e_0$, and $u\in\X_\s(\S_H)$ for $\s<\s_0$, then the map $T_u^K:\S_H\to\S_K$ defined by
  \[
  T_u^K(\om)=\frac{\p_K(f_u(\om))}{|\p_K(f_u(\om))|}=\frac{\p_K\om+u(\om)\,\p_K\nu_H}{|\p_K\om+u(\om)\,\p_K\nu_H|}\,,\qquad\om\in\S_H\,,
  \]
  is a diffeomorphism between $\S_H$ and $\S_K$, and $v_u^K:\S_K\to\R$ defined by
  \begin{equation}
    \label{vuK def}
      v_u^K(T_u^K(\om))=\frac{\q_K(f_u(\om))}{|\p_K(f_u(\om))|}
      =\frac{\nu_K\cdot(\om+u(\om)\,\nu_H)}{|\p_K\om+u(\om)\,\p_K\nu_H|}\,,\qquad\om\in\S_H\,,
  \end{equation}
  is such that
  \begin{eqnarray}
  \label{for later}
  &&v_u^K\in\X_{C(n)\,(\s+\e)}(\S_K)\,,\quad \S_H(u)=\S_K(v_u^K)\,,
  \\
    \label{auto 4 pre}
    &&\Big|\int_{\S_K}(v_u^K)^2-\int_{\S_H}u^2\Big|\le C(n)\,\Big\{|\nu_H-\nu_K|^2+\int_{\S_H}u^2\Big\}\,.
  \end{eqnarray}

  \noindent {\bf (ii):} if $H\in\H$ and $u\in\X_{\s_0}(\Sigma_H)$, then there exist $K\in\H$ with $|\nu_H-\nu_K|<\e_0$ and $v\in\X_{C_0\,\s_0}(\Sigma_K)$ such that
   \begin{eqnarray}\label{auto 1}
    &&\Sigma_H(u)=\Sigma_K(v)\,,
    \\\label{auto 2}
    &&E_{\Sigma_K}^0[v]=0\,,
    \\\label{auto 3}
    &&|\nu_K-\nu_H|^2\le C_0(n)\,\int_{\Sigma_H}\,\big(E_{\Sigma_H}^0[u]\big)^2\,,
    \\\label{auto 4}
    &&\Big|\int_{\Sigma_K}v^2-\int_{\Sigma_H}u^2\Big|\le C_0(n)\,\int_{\S_H}u^2\,.
  \end{eqnarray}
\end{lemma}

\begin{remark}
  {\rm It may seem unnecessary to present a detailed proof of Lemma \ref{lemma step one sigma}, as we are about to do,  given that, when $\Sigma_H$ is replaced by a generic integrable minimal surface $\Sigma$ in $\SS^n$, similar statements are found in the first four sections of \cite[Chapter 5]{allardalmgrenRADIAL}. However, two of those statements, namely \cite[5.3(4), 5.3(5)]{allardalmgrenRADIAL}, seem not to be correct; and the issue requires clarification, since those statements are used in the iteration arguments for the blowup/blowdown theorems \cite[Theorem 5.9/Theorem 9.6]{allardalmgrenRADIAL}; see, e.g., the second displayed chain of inequalities on \cite[Page 254]{allardalmgrenRADIAL}. To explain this issue {\bf we momentarily adopt the notation of \cite{allardalmgrenRADIAL}}. In \cite[Chapter 5]{allardalmgrenRADIAL} they consider a family of minimal surfaces $\{M_t\}_{t\in U}$ in $\SS^n$ obtained as diffeomorphic images of a minimal surface $M=M_0$. The parameter $t$ ranges in an open ball $U\subset\R^j$, where $j$ is the dimension of the space of Jacobi fields of $M$. Given a vector field $Z$ in $\SS^n$, defined on and normal to $M_t$, they denote by $F_t(Z)$ the diffeomorphism of $M_t$ into $\SS^n$ obtained by combining $Z$ with the exponential map of $\SS^n$ (up to lower than second order corrections in $Z$, this is equivalent to taking the graph of $Z$ over $M_t$, and then projecting it back on $\SS^n$, which is what we do, following \cite{Simon83}, in \eqref{def of spherical graph}). Then, in \cite[5.2(2)]{allardalmgrenRADIAL}, they define $\Lambda_t$ as the family of those $Z$ such that ${\rm Image}(F_t(Z))={\rm Image}(F_0(W))$ for some vector field $W$ normal to $M$, and, given $t,u\in U$ and $Z\in\Lambda_t$, they define $F_t^u:\Lambda_t\to\Lambda_u$ as the map between such classes of normal vector fields with the property that ${\rm Image}(F_t(Z))={\rm Image}(F_u(F_t^u(Z)))$: in particular, $F_t^u(Z)$ is the vector field that takes $M_u$ to the same surface to which $Z$ takes $M_t$. With this premise, in \cite[5.3(5)]{allardalmgrenRADIAL} they say that if $t,u\in U$, and $Z\in\Lambda_t$, then
  \begin{equation}
    \label{from AA}
    \Big|\int_{M_u}|F_t^u(Z)|^2-\int_{M_t}|Z|^2\Big|\le C\,|t-u|\,\int_{M_t}|Z|^2\,,
  \end{equation}
  for a constant $C$ depending on $M$ only. Testing this with $Z=0$ (notice that $0\in\Lambda_t$ by \cite[5.3(1)]{allardalmgrenRADIAL}) one finds $F_t^u(0)=0$, and thus $M_t={\rm Image}(F_t(0))={\rm Image}(F_u(F_t^u(0)))={\rm Image}(F_u(0))=M_u$. In particular, $M_u=M_t$ for every $t,u\in U$, that is, $\{M_t\}_{t\in U}$ consists of a single surface, $M$ itself. But this is never the case since $\{M_t\}_{t\in U}$ always contains, to the least, every sufficiently small rotation of $M$ in $\SS^n$. An analogous problem is contained in \cite[5.3(4)]{allardalmgrenRADIAL}. Coming back to our notation, the analogous estimate to \eqref{from AA} in our setting would mean that, for every $H,K\in\H$ with $|\nu_K-\nu_H|<\e_0$ and $u\in\X_{\s_0}(\S_H)$, $v_u^K$ defined in \eqref{vuK def} satisfies
  \begin{equation}
    \label{equivalent to from AA}
  \Big|\int_{\S_K}(v_u^K)^2-\int_{\Sigma_H}u^2\Big|\le C(n)\,|\nu_H-\nu_K|\,\int_{\S_H}u^2\,,
  \end{equation}
  which again gives a contradiction if $u=0$. A correct estimate, analogous in spirit to \eqref{equivalent to from AA} and still sufficiently precise to be used in iterations, is \eqref{auto 4 pre} in Lemma \ref{lemma step one sigma}. There should be no difficulty in adapting our proof to the more general context of integrable cones, and then in using the resulting generalization of \eqref{auto 4 pre} to implement the iterations needed in \cite[Theorem 5.9, Theorem 9.6]{allardalmgrenRADIAL}.}
\end{remark}

\begin{proof}
  [Proof of Lemma \ref{lemma step one sigma}] The constants $\e_0$ and $\s_0$ in the statement will be such that $\s_0=\e_0/C_*$ for a sufficiently large dimension dependent constant $C_*$.

\noindent {\bf Step one:} To prove statement (i), let $H,K\in\H$, $|\nu_H-\nu_K|\le\e<\e_0$ and $u\in\X_{\s}(\S_H)$ with $\s<\s_0$. Setting (for $\om\in\S_H$ and $x\in\R^{n+1}\setminus\{0\}$)
  \[
  g_u^K(\om)=\p_K\om+u(\om)\,\p_K\nu_H\,,\qquad \Phi(x)=x/|x|\,,
  \]
  we have $T_u^K=\Phi\circ g_u^K$, and, if $u$ is identically $0$,
  \[
  g_0^K(\om)=\p_K\om\,,\qquad T_0^K(\om)=\frac{\p_K\om}{|\p_K\om|}\,,\qquad\forall\om\in\S_H\,.
  \]
  By $|\p_K\nu_H|^2=1-(\nu_H\cdot\nu_K)^2\le 2\,(1-(\nu_H\cdot\nu_K))=|\nu_H-\nu_K|^2$,
  \begin{eqnarray*}
    &&|g_u^K-g_0^K|=|u|\,|\p_K\nu_H|\le|u|\,|\nu_H-\nu_K|\,,
   \\
   &&|\nabla^{\S_H}g_u^K-\nabla^{\S_H}g_0^K|\le|\nabla^{\S_H}u|\,|\nu_H-\nu_K|\,.
  \end{eqnarray*}
  In particular, $|g_u^K|\ge 1-\s_0\,\e_0\ge 1/2$, and since $\Phi$ and $\nabla\Phi$ are Lipschitz continuous on $\{|x|\ge 1/2\}$, we find
  \begin{equation}
    \label{guK and TuK near g0K and T0K}
      \max\big\{\|g_u^K-g_0^K\|_{C^1(\S_H)},\|T_u^K-T_0^K\|_{C^1(\S_H)}\big\}\le C(n)\, \|u\|_{C^1(\S_H)}\,|\nu_H-\nu_K|\,.
  \end{equation}
  Similarly, since $\om\cdot\nu_K=\om\cdot(\nu_K-\nu_H)$ for $\om\in\S_H$, we find that
  \begin{equation}
    \label{g0K and T0K close to id}
      \|g_0^K-{\rm id}\|_{C^1(\S_H)}\le C(n)\,|\nu_H-\nu_K|\,,\qquad \|T_0^K-{\rm id}\|_{C^1(\S_H)}\le C(n)\,|\nu_H-\nu_K|\,,
  \end{equation}
  and we thus conclude that $T_u^K$ is a diffeomorphism between $\S_H$ and $\S_K$. As a consequence, the definition \eqref{vuK def} of  $v_u^K$ is well-posed, and \eqref{for later} immediately follows (in particular, $\S_H(u)=\S_K(v_u^K)$ is deduced easily from \eqref{vuK def} and \eqref{def of spherical graph}). Finally, if we set $F_u^K(\om)=v_u^K(T_u^K(\om))^2\,J^{\S_H}\,T_u^K(\om)$ ($\om\in\S_H$), then
  \begin{eqnarray*}
  \int_{\S_K}(v_u^K)^2-\int_{\S_H}u^2=\int_{\S_H}\Big(\frac{\nu_K\cdot(\om+u\,\nu_H)}{|g_u^K(\om)|}\Big)^2\,J^{\S_H}\,T_u^K(\om)-u^2\,,
  \end{eqnarray*}
  where, using again $|\om\cdot\nu_K|\le|\nu_H-\nu_K|$ for every $\om\in\S_H$, we find
  \begin{eqnarray*}
    &&|J^{\S_H}T_u^K(\om)-1|\le C(n)\,\|T_u^K-{\rm id}\|_{C^1(\S_H)}\le C(n)\,|\nu_H-\nu_K|\,,
    \\
    &&\hspace{0.4cm}\big|1-|g_u^K(\om)|^2\big|\le\big|1-|\p_K\om|^2\big|+|\p_K\nu_H|\,u^2+2\,|u|\,|\p_K\nu_H|\,|\p_K\om|
    \\
    &&\hspace{2.9cm}\le C\,\big(|\nu_H-\nu_K|^2+u^2\big)\,,
    \\
    &&\big|(\nu_K\cdot(\om+u\,\nu_H))^2-u^2\big|
    \\
    &&\le|\nu_K\cdot\om|^2+u^2\,(1-(\nu_H\cdot\nu_K)^2)+2\,|u|\,|\nu_H\cdot\nu_K|\,|\om\cdot\nu_K|
    \\
    &&
    \le |\nu_K-\nu_K|^2+2\,u^2\,|\nu_H-\nu_K|+2\,|u|\,|\nu_H-\nu_K|    \le C\,\big(|\nu_H-\nu_K|^2+u^2\big)\;
  \end{eqnarray*}
  and thus, \eqref{auto 4 pre}, thanks to
  \begin{eqnarray*}\hspace{-1cm}
  &&\!\!\!\!\!\!\Big|\int_{\S_K}(v_u^K)^2-\int_{\S_H}u^2\Big|\le
    \int_{\S_H}\!\!|J^{\S_H}\,T_u^K-1|\,u^2
    +2\,\frac{|(\nu_K\cdot(\om+u\,\nu_H))^2-u^2|}{|g_u^K|^2}
    \\
    &&\hspace{2cm}+2\,\int_{\S_H}\,\Big|1-\frac1{|g_u^K|^2}\Big|\,u^2\le C(n)\,\Big(|\nu_H-\nu_K|^2+\int_{\S_H}u^2\Big)\,.
  \end{eqnarray*}

  \noindent {\bf Step two:} We prove (ii). If $E_{\Sigma_H}^0[u]=0$, then we conclude with $K=H$, $v=u$. We thus assume $\g^2=\int_{\Sigma_H}\,(E_{\Sigma_H}^0[u])^2>0$, and pick an orthonormal basis $\{\phi_H^i\}_{i=1}^n$ of $L^2(\S_H)\cap\{E_{\Sigma_H}^0=0\}$ with $E_{\Sigma_H}^0[u]=\g\,\phi_H^1$ and $\g=\int_{\Sigma_H}u\,\phi_H^1\ne 0$. This corresponds to choosing an orthonormal basis $\{\tau_H^i\}_{i=1}^n$ of $H$ such that
  \begin{equation}
    \label{what are the jacobi fields}
      \phi_H^i(\om)=c_0(n)\,\om\cdot\tau_H^i\,,\qquad\om\in\S_H\,,
  \end{equation}
  for $c_0(n)=(n/\H^{n-1}(\S_H))^{1/2}$. For each $K\in\H$ with $\dist_{\SS^n}(\nu_H,\nu_K)<\e_0$ we define an orthonormal basis $\{\tau_K^i\}_{i=1}^n$ of $K$ by parallel transport of $\{\tau_H^i\}_{i=1}^n\subset H\equiv T_{\nu_H}\SS^n$ to $K\equiv T_{\nu_K}\SS^n$. The maps $\nu\mapsto\tau^i(\nu):=\tau_{K(\nu)}^i$ define an orthonormal frame $\{\tau^i\}_{i=1}^n$ of $\SS^n$ on the open set $A=B_{\e_0}^{\SS^n}(\nu_H)=\{\nu\in\SS^n:\dist_{\SS^n}(\nu,\nu_H)<\e_0\}$. We denote by $\rho_H^K$ the rotation of $\R^{n+1}$ which takes $H$ into $K$ by setting $\rho_H^K(\tau_H^i)=\tau_K^i$ and $\rho_H^K(\nu_H)=\nu_K$. By the properties of parallel transport we have that
  \begin{equation}
    \label{small rotation}
    \|\rho_H^K-{\rm Id}\|_{C^0(\Sigma_K)}\le C(n)\,\dist_{\SS^n}(\nu_H,\nu_K)\le C(n)\,\e_0\,.
  \end{equation}
  Finally, we define an $L^2(\Sigma_K)$-orthonormal basis $\{\phi_K^i\}_{i=1}^n$ of $L^2(\S_K)\cap\{E_{\Sigma_K}^0=0\}$ by setting
  $\phi_K^i(\om)=c_0(n)\,\om\cdot\tau_K^i$ ($\om\in\S_K$), and correspondingly consider the map $\Psi_u:A\to\R^n$ defined by setting
  \[
  \Psi_u(\nu)=\Big(\int_{\S_{K(\nu)}}v_u^{K(\nu)}\,\phi_{K(\nu)}^1,\dots,\int_{\S_{K(\nu)}}v_u^{K(\nu)}\,\phi_{K(\nu)}^n\Big)\,,\qquad\nu\in A\,,
  \]
  where $v_u^{K(\nu)}$ is well-defined for every $\nu\in A$ thanks to step one. {\bf We now claim} the existence of $\nu_*\in A$ such that
  $\Psi_u(\nu_*)=0$.  By the area formula, \eqref{vuK def}, and $\q_{K(\nu)}[e]=\nu\cdot e$, we find
  \begin{eqnarray*}
    &&\!\!\!\!\!\!\!\!\!(e_j\cdot\Psi_u)(\nu)
    :=\int_{\S_{K(\nu)}}\!\!\!\!\!\!\!\!v_u^{K(\nu)}\,\phi_{K(\nu)}^j
    =\!\!
    \int_{\S_H}\!\!\!v_u^{K(\nu)}(T_u^{K(\nu)})\,\phi_{K(\nu)}^j(T_u^{K(\nu)})\,J^{\S_H}T_u^{K(\nu)}
    \\
    &=&
    \!\!\!\!c_0(n)\,\int_{\S_H}\!\!\nu\cdot(\om+u\,\nu_H)\Big(\rho_H^{K(\nu)}[\tau_H^j]\cdot\frac{\p_K(\om+u\,\nu_H)}{|\p_K(\om+u\,\nu_H)|^2}\Big)
    J^{\S_H}T_u^{K(\nu)}d\H^{n-1}_\om\,,
  \end{eqnarray*}
  so that \eqref{guK and TuK near g0K and T0K} gives
  \begin{eqnarray}
    \label{psiu vicina psi0}
    &&\!\!\!\!\!\!\!\!\|\Psi_u-\Psi_0\|_{C^1(A)} \le C(n)\,\s_0\,,\qquad\mbox{where}
    \\\nonumber
    &&\!\!\!\!\!\!\!\!e_j\cdot\Psi_0(\nu)=c_0(n)\,\int_{\S_H}(\nu\cdot\om)\,\,\Big(\rho_H^{K(\nu)}[\tau_H^j]\cdot\frac{\p_K\om}{|\p_K\om|^2}\Big)\,\,
    J^{\S_H}\Big[\frac{\p_K\om}{|\p_K\om|}\Big]\,d\H^{n-1}_\om\,.
  \end{eqnarray}
  By definition of $A$ and by \eqref{g0K and T0K close to id} and \eqref{small rotation},
  \begin{eqnarray}\nonumber
  &&\sup_{\nu\in A}\sup_{\om\in\S_H}\,\Big|\tau_H^j\cdot\om- \Big(\rho_H^{K(\nu)}[\tau_H^j]\cdot\frac{\p_K\om}{|\p_K\om|^2}\Big)\,\,
    J^{\S_H}\Big[\frac{\p_K\om}{|\p_K\om|}\Big]\Big|\le C(n)\,\e_0\,,
  \\\label{psi0 vicina psistar}
  &&\mbox{and thus}\,\,\|\Psi_0-\Psi_*\|_{C^1(A)}\le C(n)\,(\s_0+\e_0)\,,
  \end{eqnarray}
  where $\Psi_*:A\to\R^n$ is defined by $e_j\cdot\Psi_*(\nu)=c_0(n)\,\int_{\S_H}
  (\nu\cdot\om)\,(\tau_H^j\cdot\om)\,d\H^{n-1}_\om$ ($\nu\in A$).
  Recalling that $\{\tau^i\}_{i=1}^n$ is an orthonormal frame of $\SS^n$ on $A$, with $\nabla_{\tau^i}\nu=\tau^i(\nu)=\tau_{K(\nu)}^i=\rho_H^{K(\nu)}[\tau^i_H]$, we find
  \begin{eqnarray*}
  &&e_j\cdot\nabla_{\tau^i}\Psi_*(\nu)=
  c_0(n)\,\int_{\S_H}
  (\rho_H^{K(\nu)}[\tau^i_H]\cdot\om)\,(\tau_H^j\cdot\om)\,d\H^{n-1}_\om\,,
  \\
  &&e_j\cdot\nabla_{\tau^i}\Psi_*(\nu_H)=c_0(n)\,\int_{\S_H}(\tau^i_{H}\cdot\om)\,(\tau_H^j\cdot\om)\,d\H^{n-1}_\om=\de_{ij}/c_0(n)\,.
  \end{eqnarray*}
  By \eqref{small rotation}, \eqref{psiu vicina psi0} and \eqref{psi0 vicina psistar} we conclude that
  \begin{eqnarray}\label{pino 1}
    &&\|\Psi_u-\Psi_*\|_{C^1(A)}\le C(n)\,(\s_0+\e_0)\,,
    \\\label{pino 2}
    &&\big\|\nabla^{\SS^n}\Psi_u-c_0(n)^{-1}\,\sum_{j=1}^n\,e_j\otimes\tau^j\big\|_{C^0(A)}\le C(n)\,(\s_0+\e_0)\,.
  \end{eqnarray}
  Let us finally consider the map $h:A\times[0,1]\to\R^n$,
  \[
  h(\nu,t)=h_t(\nu)=t\,\Psi_*(\nu)+(1-t)\,\Psi_u(\nu)\,,\qquad(\nu,t)\in A\times[0,1]\,,
  \]
  which defines an homotopy between $\Psi_*$ and $\Psi_u$. By \eqref{pino 1} and \eqref{pino 2} we see that if $\nu\in\pa A$, that is, if $\dist_{\SS^n}(\nu,\nu_H)=\e_0$, then, denoting by $[\nu_H,\nu]_s$ the unit-speed length minimizing geodesic from $\nu_H$ to $\nu$, considering that $[\nu_H,\nu]_s\in A$ for every $s\in(0,\e_0$), and that $\SS^n$ is close to be flat in $A$, we find
  \begin{eqnarray*}
  |h_t(\nu)|&\ge&\Big|\int_0^{\e_0}\frac{d}{ds}\,h_t([\nu_H,\nu]_s)\,ds\Big|-|h_t(\nu_H)|
    \\
    &\ge&\Big(\frac1{c_0(n)}-C(n)\,(\e_0+\s_0)\Big)\,\e_0-C(n)\,\s_0\ge\frac{\e_0}{2\,c_0(n)}\,,
  \end{eqnarray*}
  provided $\s_0=\e_0/C_*$ is small enough with respect to $\e_0$ (i.e., provided $C_*$ is large), $\e_0$ is small in terms of $c_0$, and where we have used $\Psi_*(\nu_H)=0$ and
  \begin{equation}
    \label{pino 3}
  |\Psi_u(\nu_H)|=|\g|=\Big|\int_{\S_H}u\,\phi^1_H\Big|\le C(n)\,\s_0\,,
  \end{equation}
  to deduce $|h_t(\nu_H)|\le C(n)\,\s_0$. This proves that
  $0\not\in \pa\,h_t(\pa A)$ for every $t\in[0,1]$,
  so that $\deg(h_t,A,0)$ is independent of $t\in[0,1]$. In particular, $h_0=\Psi_u$ and $h_1=\Psi_*$ give
  $\deg(\Psi_u,A,0)=\deg(\Psi_*,A,0)=1$,
  where we have used $\Psi_*(\nu_H)=0$ and the fact that, up to decreasing the value of $\e_0$, $\Psi_*$ is injective on $A$. By $\deg(\Psi_u,A,0)=1$, there is $\nu_*\in A$ such that $\Psi_u(\nu_*)=0$, {\bf as claimed}. With $K=K(\nu_*)$ and $v=v_u^K$ we deduce \eqref{auto 1} from \eqref{for later} and \eqref{auto 2} from $\Psi_u(\nu_*)=0$. By \eqref{pino 2} and \eqref{pino 3}, if $\eta=\dist_{\SS^n}(\nu_*,\nu_H)$, then
  \begin{eqnarray*}
    &&\Big(\int_{\Sigma_H}\!\!\!\big(E_{\Sigma_H}^0[u]\big)^2\Big)^{1/2}=|\g|=|\Psi_u(\nu_H)|=|\Psi_u(\nu_H)-\Psi_u(\nu_*)|
    \\
    &=&\Big|\int_0^{\eta}\!\!\!\frac{d}{ds}\,\Psi_u([\nu_H,\nu_*]_s)\,ds\Big|
    \ge\Big(\frac1{c_0(n)}-C(n)\,(\e_0+\s_0)\Big)\,\eta
    \ge \frac{|\nu_*-\nu_H|}{2\,c_0(n)}\,\,,
  \end{eqnarray*}
  that is \eqref{auto 3}. Finally, \eqref{auto 4} follows from \eqref{auto 3} and \eqref{auto 4 pre}.
  \end{proof}

\subsection{Energy estimates for spherical graphs over annuli}\label{subsection energy estimates on annuli} Given $H\in\H$ and $0<r_1<r_2$ we let $\X_\s(\Sigma_H,r_1,r_2)$ be the class of those $u\in C^1(\Sigma_H\times(r_1,r_2))$ such that, setting $u_r=u(\cdot,r)$, one has
$u_r\in\X_\s(\Sigma_H)$ for every $r\in(r_1,r_2)$ and $|r\,\pa_r u|\le\s$ on $\Sigma_H\times(r_1,r_2)$. If $u\in\X_\s(\Sigma_H,r_1,r_2)$, then the spherical graph of $u$ over $\Sigma_H\times(r_1,r_2)$, given by
\[
\Sigma_H(u,r_1,r_2)=\Big\{r\,\frac{\om+u_r(\om)\,\nu_H}{\sqrt{1+u_r(\om)^2}}:\om\in\Sigma_H\,,r\in(r_1,r_2)\Big\}\,,
\]
is an hypersurface in $A_{r_1}^{r_2}$. It is useful to keep in mind that
$\Sigma_H(0,r_1,r_2)=\{r\,\om:\om\in\Sigma\,,r\in(r_1,r_2)\}=H\cap A_{r_1}^{r_2}$ is a flat annular region of area $\om_n\,(r_2^n-r_1^n)$, and that if $\s<\s_1=\s_1(n)$, then
\begin{equation}
  \label{equuivalence between u square and omega square}
  \frac1{C(n)}\,\int_{\S_H(u,r_1,r_2)}\!\!\!\!\om_H^2\,d\H^n\le \int_{\S_H\times(r_1,r_2)}\!\!r^{n-1}\,u^2\le C(n)\,\int_{\S_H(u,r_1,r_2)}\!\!\!\!\om_H^2\,d\H^n\,.
\end{equation}

\begin{lemma}\label{lemma step one}
  There are $\e_0$, $\s_0$, $C_0$ positive, depending on $n$ only, such that:

  \noindent {\bf (i):} if $H,K\in\H$, $\nu_H\cdot\nu_K>0$, $|\nu_H-\nu_K|=\e<\e_0$, $u\in\X_{\s}(\S_H,r_1,r_2)$, and $\s<\s_0$, then there is $v\in \X_{C_0(\s+\e)}(\S_H,r_1,r_2)$ such that $\S_K(v,r_1,r_2)=\S_H(u,r_1,r_2)$.

  \noindent {\bf (ii):} if $H\in\H$, $u\in\X_{\s_0}(\Sigma_H,r_1,r_2)$, and $(a,b)\cc(r_1,r_2)$, then there exist $K\in\H$, $v\in\X_{C_0\,\s_0}(\Sigma_K,r_1,r_2)$, and $r_*\in[a,b]$ such that
  \begin{eqnarray}\label{biro 1}
    &&\Sigma_H(u,r_1,r_2)=\Sigma_K(v,r_1,r_2)\,,
    \\\label{biro 2}
    &&E_{\Sigma_K}^0\big(v_{r_*}\big)=0\,,
    \\\label{biro 3}
    &&|\nu_H-\nu_K|^2\le C_0(n)\,\min_{\rho\in[a,b]}\int_{\Sigma_H}\,\big(E_{\Sigma_H}^0[u_\rho]\big)^2\,.
 \end{eqnarray}
 Moreover,  for every $r\in(r_1,r_2)$,
 \begin{equation}
   \label{biro 4}
    \Big|\int_{\Sigma_K}(v_r)^2-\int_{\Sigma_H}(u_r)^2\Big|\le C_0(n)\,\Big\{
    \min_{\rho\in[a,b]}\int_{\Sigma_H}\,(u_\rho)^2+\int_{\Sigma_H}(u_r)^2\Big\}\,.
 \end{equation}
\end{lemma}

\begin{proof}
   We prove statement (i). If $|\nu_H-\nu_K|=\e<\e_0$, since $u_r\in\X_{\s}(\S_H)$ for every $r\in(r_1,r_2)$, by Lemma \ref{lemma step one sigma}-(i) we see that $T_r:\S_H\to\S_K$,
  \begin{equation}
    \label{tr}
      T_r(\om)=|\p_K[\om+u_r(\om)\,\nu_H]|^{-1}\,\p_K[\om+u_r(\om)\,\nu_H]\qquad\om\in\S_H\,,
  \end{equation}
  is a diffeomorphism between $\S_H$ and $\S_K$, and $v_r:\S_K\to\R$,
  \begin{equation}
    \label{vrtr}
  v_r(T_r(\om))=\frac{\nu_K\cdot(\om+u_r(\om)\,\nu_H)}{|\p_K[\om+u_r(\om)\,\nu_H]|}\,,\qquad\om\in\S_H\,,
  \end{equation}
  satisfies $v_r\in\X_{C_0\,(\s+\e)}(\S_K)$, $\S_H(u_r)=\S_K(v_r)$ for every $r\in(r_1,r_2)$, and
  \begin{equation}
    \label{biro 5}
    \Big|\int_{\S_K}\!\!\!(v_r)^2-\int_{\S_H}\!\!\!(u_r)^2\Big|\le C(n)\,\Big\{|\nu_H-\nu_K|^2+\int_{\S_H}\!\!\!(u_r)^2\Big\}\,.
  \end{equation}
  Since $u\in\X_{\s}(\Sigma_H,r_1,r_2)$, and $T_r$ and $v_r$ depend smoothly on $u_r$,  setting $v(\om,r):=v_r(\om)$ we have $\S_H(u,r_1,r_2)=\S_K(v,r_1,r_2)$ (by $\S_H(u_r)=\S_K(v_r)$ for every $r\in(r_1,r_2)$), and $v\in\X_{C_0\,(\s+\e)}(\Sigma_H,r_1,r_2)$ ($|r\,\pa_rv_r|\le C_0(\s+\e)$ is deduced by differentiation in \eqref{tr} and \eqref{vrtr}, and by $|u_r|,|r\,\pa_ru_r|<\s$).

   \noindent {\bf Step two:} We prove (ii). Let
   $\g=\min_{\rho\in[a,b]}\int_{\Sigma_H}\,\big(E_{\Sigma_H}^0[u_\rho]\big)^2$,
   and let $r_*\in[a,b]$ be such that the minimum $\g$ is achieved at $r=r_*$. If $\g=0$, then we set $K=H$ and $v=u$. If $\g>0$, then we apply Lemma \ref{lemma step one sigma}-(ii) to $u_{r_*}\in\X_{\s_0}(\S_H)$, and find $K\in\H$ with $|\nu_K-\nu_H|<\e_0$ and $v_{r_*}\in\X_{C_0\,s_0}(\S_K)$ such that $\Sigma_H(u_{r_*})=\Sigma_K(v_{r_*})$ and
  \begin{eqnarray}
  \label{biro 6}
    &&E_{\Sigma_K}^0[v_{r_*}]=0\,,
    \\\label{biro 7}
    &&|\nu_K-\nu_H|^2\le C_0(n)\,\int_{\Sigma_H}\,\big(E_{\Sigma_H}^0[u_{r_*}]\big)^2=C_0(n)\,\g\,,
    \\\label{biro 8}
    &&\Big|\int_{\Sigma_K}(v_{r_*})^2-\int_{\Sigma_H}(u_{r_*})^2\Big|\le C_0(n)\,\int_{\S_H}(u_{r_*})^2\,.
  \end{eqnarray}
  Since $v_{r_*}=v(\cdot,r_*)$ for $v$ constructed in step one starting from $u$, $H$ and $K$, we deduce \eqref{biro 4} by \eqref{biro 5} and \eqref{biro 7}, while \eqref{biro 6} is \eqref{biro 2}.
\end{proof}

We will use two basic ``energy estimates'' for spherical graphs over annuli. To streamline the application of these estimates to diadic families of annuli we consider intervals $(r_1,r_2)$ and $(r_3,r_4)$ are {\bf $(\eta,\eta_0)$-related}, meaning that
\begin{equation}
  \label{r1r2r3r4}
  r_2=r_0(1+\eta_0)\,,\quad r_1=r_0(1-\eta_0)\,,\quad   r_4=r_0(1+\eta)\,,\quad r_3=r_0(1-\eta)\,,
\end{equation}
for some $\eta_0>\eta>0$, and with $r_0=(r_1+r_2)/2=(r_3+r_4)/2$; in particular, $(r_3,r_4)$ is contained in, and concentric to, $(r_1,r_2)$. The case $\Lambda=0$ of the following statement is the codimension one, equatorial spheres case of \cite[Lemma 7.14, Theorem 7.15]{allardalmgrenRADIAL}.

\begin{theorem}[Energy estimates for spherical graphs]\label{theorem 7.15AA main estimate lambda}
If $n\ge 2$ and $\eta_0>\eta>0$, then there are $\s_0=\s_0(n,\eta_0,\eta)$ and $C_0=C_0(n,\eta_0,\eta)$ positive, with the following property. If $H\in\H$, $\Lambda\ge0$, and $u\in\X_{\s}(\Sigma_H,r_1,r_2)$ is such that $\max\{1,\Lambda\,r_2\}\,\s\le\s_0$ and $\Sigma_H(u,r_1,r_2)$ has mean curvature bounded by $\Lambda$ in $A_{r_1}^{r_2}$, then, whenever $(r_1,r_2)$ and $(r_3,r_4)$ are $(\eta,\eta_0)$-related as in \eqref{r1r2r3r4},
\[
  \Big|\H^n(\Sigma_H(u,r_3,r_4))-\H^n(\Sigma_H(0,r_3,r_4))\Big|\le C_0\,\int_{\Sigma_H\times(r_1,r_2)}\!\!\!\!\!\!\!\!\!\!r^{n-1}\,\big(u^2+\Lambda\,r\,|u|\big)\,.
\]
Moreover, if there is $r\in(r_1,r_2)$ s.t. $E_{\Sigma_H}^0 u_r=0$ on $\Sigma_H$, then we also have
\[
\int_{\Sigma_H\times(r_3,r_4)} r^{n-1}\,u^2
\le C(n)\,\Lambda\,r_2\,(r_2^n-r_1^n)+ C_0\,\int_{\Sigma_H\times(r_1,r_2)}\,r^{n-1}\,(r\,\pa_ru)^2\,.
\]
\end{theorem}

\begin{proof}
  See appendix \ref{appendix proof of expansion mean curvature}.
\end{proof}

\subsection{Monotonicity for exterior varifolds with bounded mean curvature}\label{subsection monotonicity of exterior varifolds}
The following theorem states the monotonicity of $\Theta_{V,R,\Lambda}$ for $V\in\V_n(\Lambda,R,S)$, and provides, when $V$ corresponds to a spherical graph, a quantitative lower bound for the gap in the associated monotonicity formula; the case $\Lambda=0$, $R=0$ is contained in \cite[Lemma 7.16, Theorem 7.17]{allardalmgrenRADIAL}.

\begin{theorem}\label{theorem 7.17AA exterior lambda} {\bf (i):} If $V\in\V_n(\Lambda,R,S)$, then
\begin{equation}
  \label{lambdamonotonicity of theta ext}
\mbox{$\Theta_{V,R,\Lambda}$ is increasing on $(R,S)$}\,.
\end{equation}

\noindent {\bf (ii):} There is $\s_0(n)$ such that, if $V\in\V_n(\Lambda,R,S)$ and, for some $H\in\H$, $u\in\X_\s(\Sigma,r_1,r_2)$ with $\s\le\s_0(n)$, and $(r_1,r_2)\subset(R,S)$, we have
\begin{eqnarray}
  \label{V corresponds to u monoto lambda}
  \mbox{$V$ corresponds to $\Sigma_H(u,r_1,r_2)$ in $A_{r_1}^{r_2}$}\,,
\end{eqnarray}
then
\begin{eqnarray}
\label{lemma 7.16AA monoto lambda}
  \int_{\S_H\times(r_1,r_2)}\!\!\!\!\!\!\!\!r^{n-1} (r\,\pa u_r)^2\le C(n)\,r_2^n\,
  \Big\{\Theta_{V,R,\Lambda}(r_2)-\Theta_{V,R,\Lambda}(r_1)\Big\}\,.
\end{eqnarray}
{\bf (iii):} Finally, given $\eta_0>\eta>0$, there exist $\s_0$ and $C_0$ depending on $n$, $\eta_0$, and $\eta$ only, such that if the assumptions of part (i) and part (ii) hold and, in addition to that, we also have $\max\{1,\Lambda\,r_2\}\,\s\le\s_0$ and
\begin{equation}
  \label{lambdaAA 7.15(4) for 7.17 hole}
  \mbox{$\exists\,r\in(r_1,r_2)$ s.t. $E_{\Sigma_H}^0 u_r=0$ on $\Sigma_H$}\,,
\end{equation}
then, whenever $(r_1,r_2)$ and $(r_3,r_4)$ are $(\eta,\eta_0)$-related as in \eqref{r1r2r3r4}, we have
\begin{eqnarray}\label{lambdaAA tesi 7.17 hole}
  &&\Big|\H^n(\Sigma_H(u,r_3,r_4))-\H^n(\Sigma_H(0,r_3,r_4))\Big|
  \\\nonumber
  &&\hspace{2cm}\le C_0\,r_2^n\,\Big\{\Theta_{V,R,\Lambda}(r_2)-\Theta_{V,R,\Lambda}(r_1)+(\Lambda\,r_2)^2\Big\}\,.
\end{eqnarray}
\end{theorem}

\begin{proof}
  We give details of the proof of (i) when $V\in\M_n(\Lambda,R,S)$ (whereas the general case is addressed as in \cite[Section 17]{SimonLN}). By the coarea formula, the divergence theorem and $|\vec{H}|\le\Lambda$, for a.e. $\rho>R$,
  \begin{eqnarray}\nonumber
    \frac{d}{d\rho}\,\frac{\|V\|(B_\rho\setminus B_R)}{\rho^n}&=&
    \frac1{\rho^n}\,\int_{M\cap\pa B_\rho}\frac{|x|\,d\H^{n-1}}{|x^{TM}|}
  -\frac{n\,\H^n(M\cap (B_\rho\setminus B_R))}{\rho^{n+1}}
  \\\nonumber
  &=&
  \frac1{\rho^n}\,\int_{M\cap\pa B_\rho}\frac{|x|\,d\H^{n-1}}{|x^{TM}|}-\frac1{\rho^n}\int_{M\cap(B_\rho\setminus B_R)}\,\frac{x}{\rho}\cdot\vec{H}\,d\H^n
  \\\nonumber
  &&
  \!\!\!\!-\frac{1}{\rho^{n+1}}\Big\{\int_{M\cap\pa B_\rho}\!\!\!\!\!\nu^{{\rm co}}_M\cdot x\,d\H^{n-1}+\int_{M\cap\pa B_{R}}\!\!\!\!\!\nu^{{\rm co}}_M\cdot x\,d\H^{n-1}\Big\}
  \\\nonumber
  &\ge& \frac1{\rho^n}\,\int_{M\cap\pa B_\rho}\Big(\frac{|x|}{|x^{TM}|}-\frac{|x^{TM}|}{|x|}\Big)\,d\H^{n-1}
   \\\nonumber
    && \!\!\!\!-\frac{1}{\rho^{n+1}}\,\int_{M\cap\pa B_R}\!\!\!\!\!\nu^{{\rm co}}_M\cdot x\,d\H^{n-1}
    -\Lambda\,\frac{\H^n(M\cap(B_\rho\setminus B_R))}{\rho^n}
  \\\label{conelimit}
  &&\!\!\!\!\!\!\!\!\!\!\!\!\!\!\!\!\!\!\!\!\!\!\!\!\!\!\!\!\!\!\!\!\!\!\!\!\!\!\!\!\!\!\!\!\!\!\!
  ={\rm Mon}(V,\rho)+\frac{d}{d\rho}\,\frac1{n\,\rho^n}\,\int x\cdot\nu^{{\rm co}}_V\,d\,{\rm bd}_V-\Lambda\,\frac{\|V\|(B_\rho\setminus B_R)}{\rho^n}\,\,\,\,\,\,\,\,\,\,\,
  \end{eqnarray}
  where ${\rm Mon}(V,\rho)=(d/d\rho)\int_{B_\rho\setminus B_R}\,|x^\perp|^2\,|x|^{-n-2}\,d\|V\|$. Since ${\rm Mon}(V,\rho)\ge0$, this proves \eqref{lambdamonotonicity of theta ext}. Assuming now \eqref{V corresponds to u monoto lambda}, by using \cite[Lemma 3.5(6)]{allardalmgrenRADIAL} as done in the proof of \cite[Lemma 7.16]{allardalmgrenRADIAL}, we see that, under \eqref{V corresponds to u monoto lambda},
  \[
  C(n)\,r_2^n\,\int_{r_1}^{r_2}{\rm Mon}(V,\rho)\,d\rho\ge \int_{\Sigma_H\times(r_1,r_2)}r^{n-1}\,(r\,\pa_ru)^2\,,
  \]
  thus proving (ii). To prove (iii), we set $a=r_0\,(1-(\eta+\eta_0)/2)$ and $b=r_0\,(1+(\eta+\eta_0)/2)$, so that $(a,b)$ and $(r_3,r_4)$ are $(\eta,(\eta+\eta_0)/2)$-related, and $(r_1,r_2)$ and $(a,b)$ are $((\eta+\eta_0)/2,\eta_0)$-related (in particular, $(r_3,r_4)\subset (a,b)\subset (r_1,r_2)$). By suitably choosing $\s_0$ in terms of $n$, $\eta$ and $\eta_0$, we can apply Theorem \ref{theorem 7.15AA main estimate lambda} with $(r_3,r_4)$ and $(a,b)$, so to find (with $C=C(n,\eta_0,\eta)$)
\begin{eqnarray*}
&&\Big|\H^n(\S(u,r_3,r_4))-\H^n(\S(0,r_3,r_4))\Big|\le C\,\int_{\S_H\times(a,b)} r^{n-1}\big(u^2+\Lambda\,r\,|u|\big)
\\
&&\hspace{3cm}\le C\,\Big\{(\Lambda\,b)^2\,(b^n-a^n)+\int_{\S_H\times(a,b)} r^{n-1}\,u^2\Big\}\,.
\end{eqnarray*}
Thanks to \eqref{lambdaAA 7.15(4) for 7.17 hole} we can apply Theorem \ref{theorem 7.15AA main estimate lambda} with
$(a,b)$ and $(r_1,r_2)$ to find
\[
\int_{\S_H\times(a,b)} r^{n-1}\,u^2\le
C\,\Big\{(\Lambda\,r_2)^2\,(r_2^n-r_1^n)+\int_{\S_H\times(r_1,r_2)} r^{n-1}\,(r\,\pa_ru)^2\Big\}\,.
\]
We find \eqref{lambdaAA tesi 7.17 hole} by \eqref{lemma 7.16AA monoto lambda} and $(\Lambda\,b)^2\,(b^n-a^n)\le (\Lambda\,r_2)^2\,r_2^n$.
\end{proof}

\subsection{Proof of the mesoscale flatness criterion}\label{subsection mesoscale flatness criterion} As a final preliminary result to the proof of Theorem \ref{theorem mesoscale criterion}, we prove the following lemma, where Allard's regularity theorem is combined with a compactness argument to provide the basic graphicality criterion used throughout the iteration. The statement should be compared to \cite[Lemma 5.7]{allardalmgrenRADIAL}.

\begin{lemma}[Graphicality lemma]\label{lemma graphicality lambda}
    Let $n\ge 2$. For every $\s>0$, $\Gamma\ge0$, $(\l_3,\l_4)\cc (\l_1,\l_2)\cc (0,1)$, and $(\eta_1,\eta_2)\cc(0,1)$, there are positive constants $\e_1$ and $M_1$, depending only on $n$, $\sigma$, $\Gamma$, $(\l_1,\l_2)$, $(\l_3,\l_4)$, and $(\eta_1,\eta_2)$, and $\e_2$ and $M_2$, depending only on $n$, $\sigma$, $\Gamma$, $\l_1$, and $(\eta_1,\eta_2)$, with the following properties.\par
  \noindent {\bf (i):} If $\Lambda\ge0$, $R\in(0,1/\Lambda)$, $V\in\V_n(\Lambda,R,1/\Lambda)$,
  \begin{equation}\label{gamma bound lemma statement}
  \|{\rm bd}_V\|(\pa B_{R})\le\Gamma\,R^{n-1}\,,\qquad \sup_{\rho\in(R,1/\Lambda)}\frac{\|V\|(B_\rho\setminus B_R)}{\rho^n}\le\Gamma\,,
  \end{equation}
  there exists $r>0$ such that
  \begin{eqnarray}
  \label{conditions on r lambda}
  &&\hspace{-0.2cm}\max\{M_1,64\}\,R\le r\le \frac{\e_1}\Lambda\,,
  \\
  \label{ext hp step two 1 lambda}
  &&\hspace{1cm}|\de_{V,R,\Lambda}(r)|\le\e_1\,,
  \\\label{ext hp step two 3 lambda}
  &&\hspace{0.9cm}\|V\|(A_{\l_3\,r}^{\l_4\,r})>0\,,
  \end{eqnarray}
  and if, for some $K\in\H$, we have
  \begin{equation}
  \label{ext hp step two 2 lambda}
  \frac1{r^n}\,\int_{A_{\l_1\,r}^{\l_2\,r}}\,\om_K^2\,d\|V\|\le\e_1\,,
  \end{equation}
  then there exists $u\in\X_\s(\Sigma_K,\eta_1 \, r,\eta_2 \, r)$ such that
  \[
  \mbox{$V$ corresponds to $\Sigma_K(u,\eta_1 \, r,\eta_2 \, r)$ on $A_{\eta_1 \, r}^{\eta_2 \, r}$}\,.
  \]
 \noindent {\bf (ii):} If $\Lambda$, $R$, and $V$ are as in (i), \eqref{gamma bound lemma statement} holds, and there exists $r$ such that \begin{eqnarray}
  \label{alternative conditions on r lambda}
  &&\hspace{0.5cm}\max\{M_2,64\}\,R\le r\le \frac{\e_2}\Lambda\,,
  \\
  \label{alternative hypo to angular flatness}
  &&\max\{|\de_{V,R,\Lambda}(\lambda_1\, r)|,|\de_{V,R,\Lambda}(r)|\}\le\e_2\,,
  \end{eqnarray}
 then there exists $K\in \mathcal{H}$ and $u\in\X_\s(\Sigma_K, \eta_1\, r, \eta_2\, r)$ such that
  \[
  \mbox{$V$ corresponds to $\Sigma_K(u, \eta_1 \, r , \eta_2 \, r)$ on $A_{\eta_1 \, r}^{\eta_2 \, r}$}\,.
  \]
\end{lemma}

\begin{proof} \textit{Step one}: As a preliminary, we first show that if $V$ is a stationary, $n$-dimensional, integer rectifiable varifold in $B_1$ such that
\begin{equation}\label{plane mass}
\|V\|(B_1)\leq \omega_n\,,\quad \mathrm{spt}\, V \cap A_{\beta_1}^{\beta_2}\subset K\,,\quad\textup{and}\quad \mathrm{spt}\, V \cap A_{\beta_1}^{\beta_2} \neq \emptyset\,,
\end{equation}
for some $K\in\mathcal{H}$ and $0<\beta_1<\beta_2\leq 1$,
then $V = \var (K\cap B_1,\left.1\right|_{K\cap B_1})$.

Let $\beta'\in (\beta_1,\beta_2)$ and $\varphi_1$, $\varphi_2\in C^\infty(\mathbb{R}^{n+1};[0,1])$ be such that $\spt\, \varphi_1 \subset B_{\beta_2}$, $\left.\varphi_1\right|_{B_{\beta'}}\equiv 1$, and $\varphi_1+\varphi_2\equiv 1$. As a consequence of \eqref{plane mass} and the stationarity of $V$ in $B_{\beta_2}$, for $X\in  C_c^1(\mathbb{R}^{n+1}\setminus (K \cap (\overline{B}_{\beta_2} \setminus B_{\beta'}))$, we have
\begin{eqnarray}\notag
   \delta (V\mres B_{\beta'})(X)&=& \int_{B_{\beta'}}\Div^{M}(\varphi_1 X)+\Div^{M}(\varphi_2 X)\,d\|V\| \\ \notag
   &=&\int_{B_{\beta_2}}\Div^{M}(\varphi_1 X)\,d\|V\|
   =0\,.
\end{eqnarray}
Then by the convex hull property \cite[Theorem 19.2]{SimonLN}, $\mathrm{spt}\,(V \mres B_{\beta'}) \subset K$.
By the constancy theorem \cite[Theorem 41.1]{SimonLN}, $V \mres B_{\beta_2}=\var ( K \cap B_{\beta_2},\theta)$ for some constant $\theta$. Furthermore, since $V$ assigns non-trivial mass to $B_{\beta_2}$ by \eqref{plane mass} and is integer rectifiable, $\theta\geq 1$. Therefore $0\in \spt \|V\|$, and the monotonicity formula gives $\omega_n \leq \lim_{r\to 0^+}\|V\|(B_r)r^{-n} \leq \|V\|(B_1)\leq \omega_n$. Thus $V$ is a stationary, $n$-dimensional, integer rectifiable varifold in $B_1$ with constant area ratios $\omega_n$ and $\spt V \cap A_{\beta_1}^{\beta_2}\subset K$, so $V = \var (K\cap B_1,\left.1\right|_{K\cap B_1})$.
\par
\noindent\textit{Step two}: We prove item (i) by contradiction. If it were false, we could find $\s>0$, $\Gamma\ge0$, $(\l_3,\l_4)\cc(\l_1,\l_2)\cc(0,1)$, $(\eta_1,\eta_2)\subset (0,1)$, with $K_j\in H$, positive numbers $R_j$, $\Lambda_j<1/R_j$, $r_j$, and $W_j\in\V_n(\Lambda_j,R_j,1/\Lambda_j)$  such that
$\|W_j\|\big(A_{\l_3\,r_j}^{\l_4\,r_j}\big)>0$, $\|{\rm bd}_{W_j}\|(\pa B_{R_j})\le\Gamma\,R_j^{n-1}$, $\|W_j\|(B_\rho\setminus B_{R_j})\le\Gamma\,\rho^n$ for every $\rho\in(R_j,1/\Lambda_j)$, and $\rho_j=R_j/r_j\to 0$, $r_j\,\Lambda_j\to 0$, $\de_{W_j,R_j,\Lambda_j}(r_j)\to 0$, and $r_j^{-n}\int_{B_{\l_2\,r_j}\setminus B_{\l_1\,r_j}} \om_{K_j}^2\,d\|W_j\|\to 0$, but there is no $u\in\X_{\s}(\Sigma_{K_j},\eta_1\,r_j,\eta_2\,r_j)$ with the property that $W_j$ corresponds to $\Sigma_{K_j}(u,\eta_1\,r_j,\eta_2\,r_j)$ on $A_{\eta_1\,r_j}^{\eta_2\,r_j}$.
Hence, setting $V_j=W_j/r_j$, no $u\in\X_{\s}(\Sigma_{K_j},\eta_1,\eta_2)$ can exist such that $V_j$ corresponds to $\Sigma_{K_j}(u,\eta_1,\eta_2)$ on $A_{\eta_1}^{\eta_2}$, despite the fact that each $V_j$ belongs to $\V_n(r_j\,\Lambda_j,\rho_j,1/(r_j\,\Lambda_j))$ and satisfies
\begin{eqnarray}\nonumber
&&\|V_j\|(A_{\l_3}^{\l_4})>0\,,\,\,
\frac{\|{\rm bd}_{V_j}\|(\pa B_{\rho_j})}{\rho_j^{n-1}}\le\Gamma\,,\,\,
\sup_{\rho\in(\rho_j,1/(\Lambda_j\,r_j))}\!\!\frac{\|V_j\|(B_\rho\setminus B_{\rho_j})}{\rho^n}\le\Gamma\,,
\\ \label{wigged out}
&&\hspace{2cm}\lim_{j\to\infty}\max\big\{\delta_{V_j,\rho_j,r_j\,\Lambda_j}(1)\,,\,\,\, \int_{A_{\l_1}^{\l_2}}\,\om_{K_j}^2\,d\|V_j\|\big\}=0\,.
\end{eqnarray}
Clearly we can find $K\in\H$ such that, up to extracting subsequences, $K_j\cap B_1\to K\cap B_1$ in $L^1(\R^{n+1})$. Similarly, by \eqref{wigged out}, we can find an $n$-dimensional integer rectifiable varifold $V$ such that $V_j\weak V$ as varifolds in $B_1\setminus\{0\}$. Since the bound on the distributional mean curvature of $V_j$ on $B_{1/(\Lambda_j\,r_j)}\setminus \ov{B}_{\rho_j}$ is $r_j\,\Lambda_j$, and since $\rho_j\to 0^+$ and $r_j\,\Lambda_j\to 0^+$, it also follows that $V$ is stationary in $B_1\setminus \{0\}$, and thus, by a standard argument and since $n\ge 2$, on $B_1$. By $\|V_j\|(A_{\l_3}^{\l_4})>0$, for every $j$ there is $x_j\in A_{\l_3}^{\l_4}\cap\spt\,V_j$,  so that, up to extracting subsequences, $x_j\to x_0$ for some $x_0\in \ov{A}_{\l_3}^{\l_4}\cap\spt\,V$.
By $(\l_3,\l_4)\cc  (\l_1,\l_2)$, there is $\rho>0$ such that $B_\rho(x_0)\subset A_{\l_1}^{\l_2}$, hence
\begin{equation}\label{nontrivial V step 2}
\|V\|(A_{\l_1}^{\l_2})\ge\|V\|(B_\rho(x_0))\ge \om_n\,\rho^n>0\,,
\end{equation}
thus proving $V\,\llcorner\, A_{\l_1}^{\l_2}\ne\emptyset$. By this last fact, by $\om_K=0$ on $(\spt\,V)\cap A_{\l_1}^{\l_2}$, and by the constancy theorem \cite[Theorem 41.1]{SimonLN}, we have
\begin{equation}\label{support in plane step two}
  A_{\l_1}^{\l_2}\cap\spt\,V=A_{\l_1}^{\l_2}\cap K\,.
  \end{equation}
  At the same time, since $\|{\rm bd}_{V_j}\|(\pa B_{\rho_j})\le\Gamma\,\rho_j^{n-1}$ and $\|V_j\|(B_\rho\setminus B_{\rho_j})\le\Gamma\,\rho^n$ for every $\rho\in(\rho_j,1/(\Lambda_j\,r_j))\supset (\rho_j\,1)$, by \eqref{wigged out},
  \begin{eqnarray}\notag
  \om_n&=&\!\!\!\!\lim_{j\to\infty}\|V_j\|(B_1\setminus B_{\rho_j})-
  \frac{\rho_j}n\,\|\de V_j\|(\pa B_{\rho_j})+\Lambda_j\,r_j\,\int_{\rho_j}^1\,\frac{\|V_j\|(B_\rho\setminus B_{\rho_j})}{\rho^n}\,d\rho
  \\ \label{less mass than plane step two}
  &\ge&\|V\|(B_1)-\Gamma\,\varlimsup_{j\to\infty}\big(\rho_j^n+\Lambda_j\,r_j\big)=\|V\|(B_1)\,.
  \end{eqnarray}
  Since $V$ is stationary in $B_1$ and integer rectifiable, and since \eqref{nontrivial V step 2} and \eqref{less mass than plane step two} imply \eqref{plane mass} with $\lambda_1=\beta_1$ and $\lambda_2=\beta_2$, the first step yields $V = \mathbf{var}\,(K\cap B_1,\left.1\right|_{K\cap B_1})$. By Allard's regularity theorem and by $V_j\weak V$ as $j\to\infty$ we deduce the existence of a sequence $\{u_j\}_j$, with $u_j\in\X_{\s_j}(\Sigma_K,\eta_1,\eta_2)$ for some $\s_j\to 0$ as $j\to\infty$, such that $V_j$ corresponds to $\Sigma_K(u_j,\eta_1,\eta_2)$ in $A_{\eta_1}^{\eta_2}$ for $j$ large enough. As soon as $j$ is large enough to give $\s_j<\s$, we have reached a contradiction.\par
  \noindent\textit{Step three}: For item (ii), we again argue by contradiction. Should the lemma be false, then we could find $\s>0$, $\Gamma\ge0$, $\l_1\in (0,1)$, $(\eta_1,\eta_2)\subset (0,1)$, 
positive numbers $R_j$, $\Lambda_j<1/R_j$, $r_j$, and, by the same rescaling as in step two, $V_j\in\V_n(r_j\,\Lambda_j,\rho_j,1/(r_j\,\Lambda_j))$ with
\begin{align}\label{gamma bounds}
&\frac{\|{\rm bd}_{V_j}\|(\pa B_{\rho_j})}{\rho_j^{n-1}}\le\Gamma\,,\qquad
\sup_{\rho\in(\rho_j,1/(\Lambda_j\,r_j))}\frac{\|V_j\|(B_\rho\setminus B_{\rho_j})}{\rho^n}\le\Gamma\,,
\\  \label{lambdaext contra 1}
&\lim_{j\to\infty}\max\Big\{\rho_j=\frac{R_j}{r_j}\,,\,\, r_j\,\Lambda_j\,,\,\,
|\delta_{V_j,\rho_j,r_j\,\Lambda_j}(1)|\,,\,\, |\delta_{V_j,\rho_j,r_j\,\Lambda_j}(\lambda_1)|\Big\}=0\,,
\end{align}
such that there exists no $u\in\X_{\s}(\Sigma_{K_j},\eta_1 ,\eta_2)$ with the property that $V_j$ corresponds to $\Sigma_{K_j}(u,\eta_1,\eta_2 )$ on $A_{\eta_1}^{\eta_2 }$.
As in step two, we can find an $n$-dimensional integer rectifiable varifold $V{\color{black}=\mathbf{var}\,(M,\theta)}$ such that $V_j\weak V$ as varifolds in $B_1\setminus\{0\}$ and $V$ is stationary on $B_1$. If for some $K\in \mathcal{H}$, $V  =\mathbf{var}\,(K\cap B_1,\left.1\right|_{K\cap B_1})$, then using Allard's theorem as in the proof of (i), we have a contradiction. So we prove $V  =\mathbf{var}\,(K\cap B_1,\left.1\right|_{K\cap B_1})$.
\par
{\color{black}For every $r\in [\lambda_1,1]$, using $\rho_j\to 0^+$ and $r_j \Lambda_j \to 0^+$ in conjunction with \eqref{gamma bounds}, and then the monotonicity of $\delta_{V_j, \rho_j, r_j\,\Lambda_j}$ and \eqref{lambdaext contra 1}, we have
\begin{eqnarray}\notag
\varlimsup_{j\to \infty}\bigg|\omega_n - \frac{\|V_j\|(B_{r}\setminus B_{\rho_j})}{r^n} \bigg| &=& \varlimsup_{j\to \infty}\big|\delta_{V_j, \rho_j, r_j\,\Lambda_j}(r)\big|\\ \notag
&\leq& \lim_{j\to\infty}\max_{r\in\{\lambda_1,1\}}\Big\{|\delta_{V_j,\rho_j,r_j\,\Lambda_j}(r)|\Big\}=0\,.
\end{eqnarray}
Thus the convergence $V_j \rightharpoonup V$ and the monotonicity of $\|V \|(B_r)/r^n$ yield
\begin{equation}\label{density for V}
\|V\|(B_{r})=\omega_n r^n\quad\forall r\in(\lambda_1,1)
\quad
\textup{and}\quad
    \| V \|(B_1) 
    = \omega_n\,.
\end{equation}
By \eqref{density for V}, $V\mres( B_1\setminus \overline{B}_{\lambda_1})=\var(C,\theta_C) \mres ( B_1\setminus \overline{B}_{\lambda_1})$ for some locally $\mathcal{H}^n$-rectifiable cone $C\subset\mathbb{R}^{n+1}$ and zero homogeneous $\theta_C:C\to \mathbb{N}$. 
Now since the integer rectifiable varifold cone $\var(C,\theta_C)$ is stationary in $B_1\setminus \overline{B}_{\lambda_1}$, it is stationary in $\mathbb{R}^{n+1}$ by $n\geq 2$, and due to \eqref{density for V}, satisfies $\int_{C\cap B_1} \theta_C\, d\mathcal{H}^n = \omega_n$. Therefore $C =K$ for some $K\in  \mathcal{H}$, and $\theta_C\equiv 1$. From the definition of $C$, it follows that
\begin{equation}\label{annuluar containment}
\mathrm{spt}\,V \cap (B_1\setminus \overline{B}_{\lambda_1}) \subset K\,.
\end{equation}
Finally, \eqref{density for V} and \eqref{annuluar containment} give \eqref{plane mass} with $\beta_1=\lambda_1$, $\beta_2=1$. The result of step one then completes the proof that $V  =\mathbf{var}\,(K\cap B_1,\left.1\right|_{K\cap B_1})$.
}
\end{proof}

\begin{proof}[Proof of Theorem \ref{theorem mesoscale criterion}]
{\color{black}The proof proceeds in four steps, which we outline here. Precise statements can be found at the beginning of each step. First, we assume that $\delta_{V,R,\Lambda}(s/8)\geq 0$, and prove that $C^1$-graphicality can be propagated from $s/32$ to an upper radius $S_+/16\leq S_*/16$ as long as $\delta_{V,R,\Lambda}(S_+)$ remains non-negative and $S_+\leq \varepsilon_0/\Lambda$. This is then enough to prove the exterior blow-down result in part (ii) of Theorem \ref{theorem mesoscale criterion} in step two. In the third step, we argue that if $\delta_{V,R,\Lambda}(s/8)\leq 0$, then $C^1$-graphicality can be propagated inwards from $S_*/2$ down to $s/32$. The details in this step are quite similar to the first, so we summarize them. Finally, the first and third steps are combined in step four to conclude the proof Theorem \ref{theorem mesoscale criterion}-(i), in which there are no sign restrictions on the deficit.}\par
\noindent\textit{Step one}: In this step, given $n\geq 2$, $\Gamma\geq 0$, and $\s>0$, we prove the existence of $\e_0$ and $M_0$ (specified below in \eqref{choice of e0 2} and \eqref{choice of M0 2}, and depending on $n$, $\Gamma$, and $\s$) such that if \eqref{Gamma bounds}, \eqref{mesoscale bounds}, \eqref{mesoscale delta small s8} and \eqref{mesoscale Rstar larger s4} hold with $\e_0$ and $M_0$, and in addition,
\begin{equation}\label{step one positive deficit assumption}
    0\leq \delta_{V,R,\Lambda}(s/8)\leq \varepsilon_0\,,
\end{equation}
then there exists $K_+\in \mathcal{H}$ and $u_+\in \X_\sigma(\Sigma_{K_+},s/32,S_{+}/16
)$ such that
\begin{equation}\label{step one conclusion}
\textup{$V$ corresponds to $\Sigma_{K_+}(u_+,s/32,S_{+}/16)$ on $A_{s/32}^{S_{+}/16}$,}
\end{equation}
where
\begin{equation}\label{S+ definition}
R_+= \max \Big\{\sup\Big\{\rho\ge\frac{s}8: \de_{V,R,\Lambda}(\rho)\ge 0\Big\}, 4\,s\Big\} \,,\quad
S_+= \min \Big\{R_+,\frac{\varepsilon_0}{\Lambda} \Big\}\geq 4\,s \,.
\end{equation}\par
We start by imposing some constraints on the constants $\e_0$ and $M_0$. For the finite set
  \begin{equation}
    \label{def of J lambda}
      J=\Big\{\Big(\frac13,\frac16\Big),\Big(\frac23,\frac13\Big)\Big\}\subset\big\{(\eta_0,\eta):\eta_0>\eta>0\big\}\,,
  \end{equation}
  we let $\s_0=\s_0(n)$ be such that Lemma \ref{lemma step one}-(ii), Theorem \ref{theorem 7.15AA main estimate lambda}, and Theorem \ref{theorem 7.17AA exterior lambda}-(ii,iii) hold for every $(\eta_0,\eta)\in J$, Lemma \ref{lemma step one}-(i) holds for $\s<\s_0$, and
  \begin{equation}
    \label{choice of sigma0}
  \s_0\le\frac{\s_1}{C_0}\qquad\mbox{for $\s_1(n)$ as in \eqref{equuivalence between u square and omega square}, and
  $C_0(n)$ as in Lemma \ref{lemma step one}-(ii)}\,;
  \end{equation}
  we shall henceforth assume, without loss of generality, that
  \[
  \s<\s_0\,.
  \]
  Moreover, for $\e_1$ and $M_1$ as in Lemma \ref{lemma graphicality lambda}-(i) and $C_0$ as in Lemma \ref{lemma step one}, we let
  \begin{eqnarray}
   \notag
      M_0'&\ge&\max\Big\{M_1\Big(n,\frac{\s}{2\,C_0},\Gamma,\Big(\frac1{8},\frac12\Big),\Big(\frac16,\frac14\Big),\Big(\frac{1}{32},\frac{1}{2}\Big)\Big),\\  \label{choice of M0}
      &&\qquad M_1\Big(n,\frac{\s}{2\,C_0},\Gamma,\Big(\frac1{16},\frac18\Big),\Big(\frac{3}{32},\frac{7}{64}\Big),\Big(\frac{1}{32},\frac{1}{2}\Big)\Big)\Big\}\,,
  \end{eqnarray}
  \begin{eqnarray}
    \notag
  \e_0'&\le&
  \min\Big\{\e_1\Big(n,\frac{\s}{2\,C_0},\Gamma,\Big(\frac1{8},\frac12\Big),\Big(\frac16,\frac14\Big),\Big(\frac{1}{32},\frac{1}{2}\Big)\Big),\\ \label{choice of eps0}
 && \qquad \e_1\Big(n,\frac{\s}{2\,C_0},\Gamma,\Big(\frac1{16},\frac18\Big),\Big(\frac{3}{32},\frac{7}{64}\Big),\Big(\frac{1}{32},\frac{1}{2}\Big)\Big)\Big\}\,.
  \end{eqnarray}
  We also assume that
  \begin{equation}\label{ability to rotate}
      C(n,\Gamma)(\e_0')^{1/2} \leq \min \Big\{ \e_0 , \frac{\s}{2\,C_0}\Big\}\,,
  \end{equation}
  where $C(n,\Gamma)$ is specified in \eqref{root bound 2}, $C_0$ is as in Lemma \ref{lemma step one}, and $\e_0$ is smaller than both of the $n$-dependent $\e_0$'s appearing in Lemma \ref{lemma step one sigma} and Lemma \ref{lemma step one}. Lastly, we choose $\overline{\s}>0$ such that
  \begin{equation}\label{sprime le s and e0}
      \overline{\s}\leq \min \Big\{ \frac{\s}{2\,C_0}, \sqrt{\e_0'/\omega_n} \Big\}\,,
  \end{equation}
  and then, for $\e_2$, $M_2$ as in Lemma \ref{lemma graphicality lambda}-(ii), we choose $\e_0$ and $M_0$ so that
  \begin{equation}\label{choice of e0 2}
      \e_0 \leq \min \Big\{
      \e_0'\,, \e_2\Big(n,\overline{\s},\Gamma,\frac{1}{8} ,\Big(\frac{1}{32},\frac{1}{2}\Big)\Big) \Big\}
  \end{equation}
  \begin{equation}\label{choice of M0 2}
      M_0 \geq \max \Big\{M_0'\,, M_2\Big(n,\overline{\s},\Gamma,\frac{1}{8},\Big(\frac{1}{32},\frac{1}{2}\Big) \Big) \Big\}\,.
  \end{equation}
Let us now recall that, by assumption, $V\in\V_n(\Lambda,R,1/\Lambda)$ is such that
  \begin{equation}
    \label{bounds needed for proof meso}
    \|{\rm bd}_V\|(\pa B_{R})\le\Gamma\,R^{n-1}\,,\qquad \sup_{\rho\in(R,1/\Lambda)}\frac{\|V\|(B_\rho\setminus B_R)}{\rho^n}\le\Gamma\,;
  \end{equation}
  in particular, by Theorem \ref{theorem 7.17AA exterior lambda}-(i),
  \begin{equation}
    \label{is decreasing}
    \mbox{$\de_{V,R,\Lambda}$ is decreasing on $(R,1/\Lambda)$}\,.
  \end{equation}
  Moreover, we are assuming the existence of $s$ with $\max\{64,M_0\}\,R<s<\e_0/4\,\Lambda$ such that
  \begin{eqnarray}\label{ext bd easy 1 proof lambda}
  &&|\de_{V,R,\Lambda}(s/8)|\le \e_0\,,
  \\\label{ext bd easy 3 proof lambda}
  &&R_*=\sup\Big\{\rho\ge\frac{s}8: \de_{V,R,\Lambda}(\rho)\ge -\e_0\Big\}\ge 4\,s\,,
  \end{eqnarray}
  so that the latter inequality, together with \eqref{S+ definition}, implies
  \begin{equation}\label{R+ inequality}
    R_* \geq R_+\,.
  \end{equation}
  By \eqref{is decreasing}, \eqref{ext bd easy 1 proof lambda} and \eqref{R+ inequality} we have
  \begin{equation}
    \label{delta below eps0 in absolute value}
  |\de_{V,R,\Lambda}(r)|\le\e_0\,,\qquad\forall r\in[s/8,R_+]\,.
  \end{equation}
\par
By \eqref{bounds needed for proof meso}, the specification of $s$ satisfying \eqref{mesoscale bounds}, and \eqref{delta below eps0 in absolute value}, the assumptions \eqref{gamma bound lemma statement}, \eqref{alternative conditions on r lambda}, and \eqref{alternative hypo to angular flatness}, respectively, of Lemma \ref{lemma graphicality lambda}-(ii) with $r=s$, $\lambda_1=1/8$, and $(\eta_1,\eta_2)=(1/32,1/2)$ are satisfied due to our choices \eqref{choice of e0 2} and \eqref{choice of M0 2}. Setting $H_0=H$, {\color{black}where $H\in\mathcal{H}$ is from the application of Lemma \ref{lemma graphicality lambda}-(ii),} we thus find $u_0\in\X_{\overline{\s}}(\S_{H_0},s/32,s/2)$ such that 
\begin{equation}\label{ext bd easy corresponds 0 lambda}
    \mbox{\textup{$V$ corresponds to $\Sigma_{H_0}(u_0,s/32, s/2)$ on $A_{s/32}^{s/2}$.}}
\end{equation}
If it is the case that $S_+=4\,s$, we are in fact done with the proof of \eqref{step one conclusion}, since then $s/2\geq S_+ / 16$. We may for the rest of this step assume then that $S_+ > 4\,s$, so that
\begin{equation}\label{remaining definition of R+}
   R_+=\sup\Big\{\rho\ge\frac{s}8: \de_{V,R,\Lambda}(\rho)\ge 0\Big\}\geq S_+ > 4\,s\,.
\end{equation}

\par

First, we observe that
  thanks to \eqref{ext bd easy corresponds 0 lambda} and then \eqref{sprime le s and e0},
  \begin{eqnarray}\label{ext bd decay Tj 0 lambda}
  T_0:=\frac{1}{(s/4)^n}\,\int_{s/8}^{s/4}\,r^{n-1}\,dr\,\int_{\Sigma_{H_0}}[u_0]_r^2
  \le \omega_n \,\overline{\s}^2 \le  \e_0'\,.
  \end{eqnarray}
  We let $s_j=2^{j-3}\,s$ for  $j\in\mathbb{Z}_{\geq -1}$. By \eqref{remaining definition of R+} and by $s<\e_0/4\,\Lambda\leq \e_0'/4\,\Lambda$ there exists $N\in\{j\in\N:j\ge 2\}\cup\{+\infty\}$ such that
  \begin{equation}
    \label{def of N}
    \{0,1,..,N\}=\Big\{j\in\N: 8\,s_j\le S_+=\min\big\{R_+,\frac{\e_0'}\Lambda\big\}\Big\}\,.
  \end{equation}
  Notice that if $\Lambda>0$ then it must be $N<\infty$. We are now in the position to make the following:

\noindent {\bf Claim:} There exist $\tau=\tau(n)\in(0,1)$ and $\{(H_j,u_j)\}_{j=0}^{N-2}$ with $H_j\in\H$ such that, setting
  \[
  T_j=\frac1{s_{j+1}^n}\,\int_{s_j}^{s_{j+1}}\,r^{n-1}\,dr\,\int_{\Sigma_{H_j}}\,[u_j]_r^2\,,
  \]
  for every $j=0,...,N-2$,
  \begin{eqnarray}
  \label{claim is spherical graph 1}
  &&\mbox{$u_j\in\X_\s(\Sigma_{H_j},s/32,4\,s_{j-1})\cap \X_{\s/2\,C_0}(\Sigma_{H_j},s_j/4,4\,s_j)$}\,,
  \\ \label{claim is spherical graph 2}
  &&\mbox{$V$ corresponds to $\Sigma_{H_j}(u_j,s/32,4\,s_j)$ on $A_{s/32}^{4\,s_j}$}\,,
  \end{eqnarray}
  where $C_0$ is from Lemma \ref{lemma step one}, and
  \begin{eqnarray}
  \label{claim delta smaller eps0}
  |\de_{V,R,\Lambda}(s_j)|&\le&\e_0'\,,
  \\
  \label{claim Tj smaller eps0}
  T_j&\le& C(n)\,\e_0'\,;
  \end{eqnarray}
  additionally, for every $j=1,...,N-2$,
  \begin{eqnarray}
  \label{claim decay normal vectors}
  |\nu_{H_j}-\nu_{H_{j-1}}|^2&\le&C(n)\,T_{j-1}\,,
  \\
  \label{claim deltaj smaller deltaj-1}
  \de_{V,R,\Lambda}(s_j)&\le& \tau\,\big\{\de_{V,R,\Lambda}(s_{j-1})+(1+\Gamma)\,\Lambda\,s_{j-1}\big\}\,,
  \\
  \label{claim Tj smaller deltaj-1}
  T_j&\le& C(n)\,\Big\{\de_{V,R,\Lambda}(s_{j-1})-\de_{V,R,\Lambda}(s_{j+2})+\Lambda\,s_{j-1}\Big\}\,.
  \end{eqnarray}
  \noindent {\bf Proof of the claim:} We argue by induction. Clearly \eqref{claim is spherical graph 1}$_{j=0}$, \eqref{claim is spherical graph 2}$_{j=0}$, \eqref{claim delta smaller eps0}$_{j=0}$ and \eqref{claim Tj smaller eps0}$_{j=0}$ are, respectively, \eqref{ext bd easy corresponds 0 lambda}, \eqref{ext bd easy 1 proof lambda} and \eqref{ext bd decay Tj 0 lambda}. This concludes the proof of the claim if $N=2$, therefore we shall assume $N\ge 3$ for the rest of the argument. To set up the inductive argument, we consider $\ell\in\N$ such that: either $\ell=0$; or $1\le\ell\le N-3$ and \eqref{claim is spherical graph 1}, \eqref{claim is spherical graph 2}, \eqref{claim delta smaller eps0}, and
  \eqref{claim Tj smaller eps0} hold for $j=0,...,\ell$, and \eqref{claim decay normal vectors}, \eqref{claim deltaj smaller deltaj-1} and
  \eqref{claim Tj smaller deltaj-1} hold for $j=1,...,\ell$; and prove that all the conclusions of the claim hold with $j=\ell+1$.

   The validity of \eqref{claim delta smaller eps0}$_{j=\ell+1}$ is of course immediate from \eqref{delta below eps0 in absolute value} and \eqref{def of N}. Also, after proving \eqref{claim Tj smaller deltaj-1}$_{j=\ell+1}$, we will be able to combine with  \eqref{claim delta smaller eps0}$_{j=\ell+1}$ and \eqref{def of N} to deduce \eqref{claim Tj smaller eps0}$_{j=\ell+1}$. We now prove, in order,  \eqref{claim decay normal vectors}, \eqref{claim is spherical graph 1}, \eqref{claim is spherical graph 2}, \eqref{claim deltaj smaller deltaj-1}, and \eqref{claim Tj smaller deltaj-1} with $j=\ell+1$.

   \noindent{\it To prove \eqref{claim decay normal vectors}$_{j=\ell+1}$}: Let $[a,b]\cc (s_\ell,s_{\ell+1})$ with $(b-a)=(s_{\ell+1}-s_\ell)/2$, so that
  \begin{equation}
    \label{ext sell star lambda}
  \frac1{C(n)}\,\min_{r\in[a,b]}\int_{\Sigma_{H_\ell}}[u_\ell]_r^2\le
  \frac1{s_{\ell+1}^n}\,\int_{s_{\ell}}^{s_{\ell+1}}\,r^{n-1}\,dr\,\int_{\Sigma_{H_\ell}}[u_\ell]_r^2= T_\ell\,.
  \end{equation}
  Keeping in mind \eqref{claim is spherical graph 1}$_{j=\ell}$, \eqref{claim is spherical graph 2}$_{j=\ell}$, we can apply Lemma \ref{lemma step one}-(ii) with $(r_1,r_2)=(s/32,4\,s_\ell)$ and $[a,b]\subset(s_\ell,s_{\ell+1})$ to find $H_{\ell+1}\in\H$,
  \begin{equation}\label{first u ell+1 bound}
  u_{\ell+1}\in\X_{C_0\,\s_0}(\Sigma_{H_{\ell+1}},s/32,4\,s_{\ell})
  \end{equation}
  (with $C_0$ as in Lemma \ref{lemma step one}-(ii)) and
  \begin{equation}
    \label{where is sell star}
      s_\ell^*\in[a,b]\subset(s_\ell,s_{\ell+1})\,,
  \end{equation}
  such that, thanks also to \eqref{ext sell star lambda},
  \begin{eqnarray}\label{ext u ell plus one 1 lambda}
  &&\Sigma_{H_\ell}(u_\ell,s/32,4\,s_{\ell})=\Sigma_{H_{\ell+1}}(u_{\ell+1},s/32,4\,s_{\ell})\,,
  \\\label{ext u ell plus one 2 lambda}
  &&E_{\Sigma_{H_\ell+1}}^0\big([u_{\ell+1}]_{s_{\ell}^*}\big)=0\,,
  \\\label{ext u ell plus one 3 lambda}
  &&|\nu_{H_{\ell}}-\nu_{H_{\ell+1}}|^2\le C(n)\,T_\ell\,,
  \\\label{ext u ell plus one 4 lambda}
  &&\int_{\Sigma_{H_{\ell+1}}}[u_{\ell+1}]_r^2\le C(n)\,\Big(T_\ell+\int_{\Sigma_{H_\ell}}[u_\ell]_r^2\Big)\,\,,\qquad\forall r\in(s/32,4\,s_{\ell})\,.\hspace{1cm}
  \end{eqnarray}
  In particular, \eqref{ext u ell plus one 3 lambda} is \eqref{claim decay normal vectors}$_{j=\ell+1}$.

  \noindent {\it To prove \eqref{claim is spherical graph 1}$_{j=\ell+1}$ and \eqref{claim is spherical graph 2}$_{j=\ell+1}$}: Notice that \eqref{first u ell+1 bound}, \eqref{ext u ell plus one 1 lambda} do not imply \eqref{claim is spherical graph 1}$_{j=\ell+1}$ and \eqref{claim is spherical graph 2}$_{j=\ell+1}$, since, in \eqref{claim is spherical graph 2}$_{j=\ell+1}$, we are claiming the graphicality of $V$ inside $A_{s/32}^{4\,s_{\ell+1}}$ (which is strictly larger than $A_{s/32}^{4\,s_\ell}$), and in \eqref{claim is spherical graph 1}$_{j=\ell+1}$ we are claiming that $u_{\ell+1}$ has $C^1$-norm bounded by $\s$ or $\s/2\,C_0$ (depending on the radius), and not just by $C_0\,\s_0$ (with $C_0$ as in Lemma \ref{lemma step one}-(ii)).

   We want to apply Lemma \ref{lemma graphicality lambda}-(i) with $K=H_{\ell+1}$ and
  \begin{equation}
    \label{ext parameters for step two lambda}
      r=8\,s_{\ell+1}\,,\,\, (\l_1,\l_2)=\Big(\frac1{16},\frac18\Big)\,,\,\,(\l_3,\l_4)=\Big(\frac{3}{32},\frac{7}{64}\Big)\,,\,\,  (\eta_1,\eta_2)=\Big(\frac{1}{32},\frac{1}{2}\Big)\,.
  \end{equation}
  We check the validity of \eqref{conditions on r lambda}, \eqref{ext hp step two 1 lambda}, \eqref{ext hp step two 3 lambda}, and \eqref{ext hp step two 2 lambda} with $\e_1=\e_0'$ and $M_1=M_0'$ for these choices of $r$, $\l_1$, $\l_2$, $\l_3$, $\l_4$, $\eta_1$, $\eta_2$, and $K$. Since $r=8\,s_{\ell+1}\ge s\ge\max\{M_0,64\,R\}\geq \max\{M_0',64\,R\}$, and since \eqref{def of N} and $\ell+1\le N$ give $r=8\,s_{\ell+1}\le \e_0/\Lambda\leq \e_0'/\Lambda$, we deduce the validity of \eqref{conditions on r lambda} with $r=8\,s_{\ell+1}$. The validity of \eqref{ext hp step two 1 lambda} with $r=8\,s_{\ell+1}$ is immediate from \eqref{delta below eps0 in absolute value} by our choice \eqref{choice of eps0} of $\e_0'$. Next we notice that
  \[
  \|V\|(A_{\l_3\,r}^{\l_4\,r})=\|V\|(A_{3\,[8\,s_{\ell+1}]/32}^{7\,[8\,s_{\ell+1}]/64})=\|V\|(A_{3\,s_{\ell}/2}^{7\,s_\ell/4})>0
  \]
  thanks to \eqref{claim is spherical graph 2}$_{j=\ell}$, so that \eqref{ext hp step two 3 lambda} holds for $r$, $\l_3$ and $\l_4$ as in \eqref{ext parameters for step two lambda}. Finally, by \eqref{equuivalence between u square and omega square} (which can be applied to $u_{\ell+1}$ thanks to \eqref{choice of sigma0}), \eqref{ext u ell plus one 1 lambda} and \eqref{claim is spherical graph 1}$_{j=\ell}$, and, then by \eqref{ext u ell plus one 4 lambda}, we have
  \begin{eqnarray*}
      \frac1{r^n}\,\int_{A_{\l_1\,r}^{\l_2\,r}}\,\om_{H_{\ell+1}}^2\,d\|V\|&\le&
      \frac{C(n)}{s_{\ell+1}^n}\,\int_{s_\ell}^{s_{\ell+1}}\,r^{n-1}\,dr\,\int_{\Sigma_{H_{\ell+1}}}[u_{\ell+1}]_r^2
      \\
      &\le& C(n)\,T_\ell+
      \frac{C(n)}{s_{\ell+1}^n}\,\int_{s_\ell}^{s_{\ell+1}}\,r^{n-1}\,dr\,\int_{\Sigma_{H_{\ell}}}[u_{\ell}]_r^2
      \\
      &\le& C(n)\,T_\ell\le  C(n)\,\e_0'\,,
  \end{eqnarray*}
  where in the last inequality we have used \eqref{claim Tj smaller eps0}$_{j=\ell}$. Again by our choice \eqref{choice of eps0} of $\e_0'$, we deduce that \eqref{ext hp step two 2 lambda} holds with $r$, $\l_1$ and $\l_2$ as in \eqref{ext parameters for step two lambda}.
  We can thus apply Lemma \ref{lemma graphicality lambda}-(i), and find $v\in\X_{\s/2\,C_0}(\Sigma_{H_{\ell+1}},s_{\ell+1}/4,4\,s_{\ell+1})$ such that
  \begin{equation}
    \label{ext jelly lambda}
      \mbox{$V$ corresponds to $\Sigma_{H_{\ell+1}}(v,s_{\ell+1}/4,4\,s_{\ell+1})$ on $A_{s_{\ell+1}/4}^{4\,s_{\ell+1}}$}\,.
  \end{equation}
  By \eqref{ext u ell plus one 1 lambda}, \eqref{claim is spherical graph 2}$_{j=\ell}$, and \eqref{ext jelly lambda}, $v=u_{\ell+1}$ on $\S_{H_{\ell+1}}\times(s_{\ell+1}/4, 4\,s_\ell)$. We can thus use $v$ to extend $u_{\ell+1}$ from $\S_{H_{\ell+1}}\times(s/32, 4\,s_\ell)$ to $\S_{H_{\ell+1}}\times(s/32,4\,s_{\ell+1})$, and, thanks to \eqref{ext u ell plus one 1 lambda}, \eqref{claim is spherical graph 2}$_{j=\ell}$ and \eqref{ext jelly lambda}, the resulting extension is such that
  \begin{eqnarray}
      \label{initial ul+1 bound}
  &&\mbox{$u_{\ell+1}\in\X_{\s/2\,C_0}(\Sigma_{H_{\ell+1}},s_{\ell+1}/4,4\,s_{\ell+1})$}\quad\textup{and}\\ \label{V correspondence}
  &&\textup{$V$ corresponds to $\Sigma_{H_{\ell+1}}(u_{\ell+1},s/32,4\,s_{\ell+1})$ on $A_{s/32}^{4\,s_{\ell+1}}$}\,.
  \end{eqnarray}
  The bound \eqref{initial ul+1 bound} is part of \eqref{claim is spherical graph 1}$_{j=\ell+1}$, and \eqref{V correspondence} is \eqref{claim is spherical graph 2}$_{j=\ell+1}$, so in order to complete the proof of \eqref{claim is spherical graph 1}$_{j=\ell+1}$ and \eqref{claim is spherical graph 2}$_{j=\ell+1}$, it remains to show that the $C^1$-norm of $u$ is bounded by $\sigma$ in between $s/32$ and $4\,s_{\ell}$.

\par

  Towards this end, we record the following consequence of taking square roots in \eqref{claim deltaj smaller deltaj-1}$_{j=m}$ (using $\delta_{V,R,\Lambda}\geq 0$ from \eqref{def of N}) and summing over $m=1,...,i$ for any $1\leq i \leq\ell$: for $\alpha=\sum_{k=0}^\infty 2^{-k/2}$ and $\tilde{C}(n,\Gamma)=\tau^{1/2}(1+\Gamma)$,
\begin{eqnarray}\notag
    S_{i}:=\sum_{m=0}^{i}\delta_{V,R,\Lambda}(s_m)^{1/2} &\leq& \tau^{1/2} \sum_{m=0}^{i-1}\delta_{V,R,\Lambda}(s_{m})^{1/2} + (1+\Gamma)(\Lambda \, s_m)^{1/2} \\ \notag &&\qquad+\delta_{V,R,\Lambda}(s_0)^{1/2}\\ \notag
    &\leq& \tau^{1/2} S_{i-1} +\alpha\,\tilde{C}(n,\Gamma)(\Lambda \, s_{i-1})^{1/2} +\delta_{V,R,\Lambda}(s_0)^{1/2}\\ \label{iterated delta bound}
    &\leq&\tau^{1/2} S_{i-1}+ (1+\alpha\,\tilde{C}(n,\Gamma))(\e_0')^{1/2}\,,
\end{eqnarray}
where in the last line we have used \eqref{def of N} and \eqref{delta below eps0 in absolute value}. By induction, utilizing \eqref{step one positive deficit assumption}, \eqref{choice of e0 2} for the base case and \eqref{iterated delta bound} for the induction step we have
\begin{equation}\label{sum bound}
    S_{i}\leq  \frac{(1+\alpha\,\tilde{C}(n,\Gamma))(\e_0')^{1/2}}{1-\tau^{1/2}}\quad\forall \,0\leq i \leq \ell\,.
\end{equation}
Now by the positivity of $\delta_{V,R,\Lambda}$ and \eqref{claim Tj smaller deltaj-1}$_{j=\ell}$, for all $m=1,...,\ell$,
\begin{align}\label{iterated T bound}
    T_m^{1/2} \leq C(n)\delta_{V,R,\Lambda}(s_{m-1})^{1/2}  + C(n)(\Lambda\, s_{m-1})^{1/2}\,,
\end{align}
In turn, by \eqref{claim decay normal vectors}$_{j=\ell+1}$, \eqref{ext bd decay Tj 0 lambda} and \eqref{iterated T bound}, then \eqref{def of N} and \eqref{sum bound}$_{i=\ell-1}$,
\begin{eqnarray}\notag
    \frac{1}{C(n)} \sum_{m=1}^{\ell+1} |\nu_{H_m}-\nu_{H_{m-1}}|
    &\leq& \sum_{m=0}^{\ell} T_{m}^{1/2} \\ \notag
    &\leq&   (\e_0')^{1/2}+C(n)S_{\ell-1} + \alpha\,C(n)(\Lambda\,{s_{\ell-1}})^{1/2}  \\ \label{root epsilon bound}
    &\leq& C(n,\Gamma)(\e_0')^{1/2}
\end{eqnarray}
for suitable $C(n,\Gamma)$. Combining constants, we use \eqref{root epsilon bound} to see
\begin{eqnarray}\label{root bound 2}
    |\nu_{H_i}-\nu_{H_{\ell+1}}| \leq C(n,\Gamma)(\e_0')^{1/2}\quad\forall i=0,...,\ell \,.
\end{eqnarray}
Now $u_i \in \X_{\s/2\,C_0}(\Sigma_{H_j},s_i/4, 4\,s_i)$ by \eqref{claim is spherical graph 1}$_{j=i}$, and $\s/2\,C_0$ and $|\nu_{H_i}-\nu_{H_{\ell+1}}|$ are small enough to apply Lemma \ref{lemma step one}-(i) by our choice of $\s$ above \eqref{choice of sigma0} and \eqref{root bound 2} with \eqref{ability to rotate}, respectively. Then we obtain $w_i$ 
corresponding to $V$ on $A_{s_i/4}^{4\,s_i}$ and in $ \X_{\s/2 + C_0|\nu_{H_i}-\nu_{H_{\ell+1}}|}(\Sigma_{H_{\ell+1}}, s_i/4,4\,s_i)$, and by \eqref{root bound 2}, \eqref{ability to rotate},
$$
\frac{\s}{2}+C_0\,|\nu_{H_i}-\nu_{H_{\ell+1}}| \leq \frac{\s}{2}+C_0\,\frac{\s}{2\,C_0}=\s\,,
$$
so $w_i\in \mathcal{X}_\sigma(\Sigma_{H_{\ell+1}},s_i/4,4\,s_i)$. Finally, since they represent the same surface over $\S_{H_{\ell+1}}$, $w_i=u_{\ell+1}$ on $A^{4\,s_i}_{s_i/4}$. Gathering these estimates for $i=0,...,\ell$, we have $u_{\ell+1}\in \X_\s(\S_{H_{\ell+1}},s/32, 4\,s_\ell)$, which finishes the proof of \eqref{claim is spherical graph 1}$_{j=\ell+1}$.

 \par

  \noindent {\it To prove \eqref{claim deltaj smaller deltaj-1}$_{j=\ell+1}$}: We set $r_0=(s_\ell+s_{\ell+1})/2$ and notice that for $\eta_0=1/3$,
  \begin{equation}
    \label{ext choices 1 lambda}
      r_1=r_0\,(1-\eta_0)=s_\ell\,,\qquad r_2=r_0\,(1+\eta_0)=s_{\ell+1}\,.
  \end{equation}
  For $\eta=1/6$ we correspondingly set
  \begin{equation}
    \label{ext choices 2 lambda}
    r_3=r_0\,(1-\eta)=:s_\ell^-\,,\qquad r_4=r_0\,(1+\eta)=:s_\ell^+\,,
  \end{equation}
  and notice that $(\eta_0,\eta)\in J$, see \eqref{def of J lambda}. With the aim of applying Theorem \ref{theorem 7.17AA exterior lambda}-(iii) to these radii, we notice that \eqref{claim is spherical graph 2}$_{j=\ell+1}$ implies that assumption \eqref{V corresponds to u monoto lambda} holds with $H=H_{\ell+1}$ and $u=u_{\ell+1}$, while, by \eqref{ext u ell plus one 2 lambda}, $r=s_\ell^*\in(s_\ell,s_{\ell+1})$ is such that \eqref{lambdaAA 7.15(4) for 7.17 hole} holds. By $\Lambda\,s_{\ell+1}\le\e_0\le 1$, \eqref{def of N}, and \eqref{lambdaAA tesi 7.17 hole}, with $C(n)=C_0(n,1/6,1/3)$ for $C_0$ as in Theorem \ref{theorem 7.17AA exterior lambda}-(iii), we have
  \begin{eqnarray}
    \nonumber
    &&s_{\ell+1}^{-n}\,\,\big|\|V\|\big(B_{s_\ell^+}\setminus B_{s_\ell^-}\big)-\om_n\,\big((s_\ell^+)^n-(s_\ell^-)^n\big)\big|
    \\\nonumber
    &&= s_{\ell+1}^{-n}\big|\H^n(\Sigma_{H_{\ell+1}}(u_{\ell+1},s_\ell^-,s_\ell^+))-\H^n(\Sigma_{H_{\ell+1}}(0,s_\ell^-,s_\ell^+))\big|
   \\\label{compare to lambda}
&&\le C(n)\,\big\{(\Lambda\,s_{\ell+1})^2\,\,
+\Theta_{V,R,\Lambda}(s_{\ell+1})-\Theta_{V,R,\Lambda}(s_{\ell})\big\}\,.
  \end{eqnarray}
Setting for brevity $\de=\de_{V,R,\Lambda}$ and $\Theta=\Theta_{V,R,\Lambda}$, and recalling that
\begin{eqnarray*}
&&r^n\,\de(r)=\om_n\,r^n-\Theta(r)\,r^n
\\
&&=\om_n\,r^n-\|V\|(B_r\setminus B_R)
-\Lambda\,r^n\,\int_R^r\,\frac{\|V\|(B_\rho\setminus B_R)}{\rho^n}\,d\rho+\frac{R\,\|\de V\|(\pa B_R)}{n}
\end{eqnarray*}
we have
\begin{eqnarray*}
&&\!\!\!\!\!\!\!\!s_\ell^{-n}\,\big|(s_{\ell}^-)^n\,\de(s_{\ell}^-)
  -(s_{\ell}^+)^n\,\de(s_{\ell}^+)\big|
  \le C(n)\,\big\{(\Lambda\,s_\ell)^2+\Theta(s_{\ell+1})-\Theta(s_{\ell})\big\}
\\
&&\!\!\!\!\!\!\!\!\!\!+C(n)\,\Lambda\,s_\ell^{-n}\Big\{(s_\ell^+)^n\int_{R}^{s_{\ell}^+}\frac{\|V\|(B_\rho\setminus B_R)}{\rho^n}\,d\rho
-(s_\ell^-)^n\int_{R}^{s_{\ell}^-}\frac{\|V\|(B_\rho\setminus B_R)}{\rho^n}\,d\rho\Big\}
\\
&&\!\!\le C(n)\,\Big\{(\Lambda\,s_\ell)^2 +\Theta(s_{\ell+1})-\Theta(s_{\ell})\Big\}
+C(n)\,\Lambda\, \int_{R}^{s_{\ell}^+}\frac{\|V\|(B_\rho\setminus B_R)}{\rho^n}\,d\rho\,.
\end{eqnarray*}
By $\Lambda\,s_\ell\le 1$ and since $s_{\ell}^+\le s_\ell\le \e_0/8\,\Lambda$ thanks to $\ell<N$, we can use the upper bound  $\|V\|(B_\rho\setminus B_R)\le \Gamma\,\rho^n$ with $\rho\in(R,s_\ell^+)\subset(R,1/\Lambda)$, to find that
\begin{eqnarray*}
\Big|\frac{(s_{\ell}^-)^n}{s_\ell^n}\,\de(s_{\ell}^-)
  -\frac{(s_{\ell}^+)^n}{s_\ell^n}\,\de(s_{\ell}^+)\Big|
\le\! C_*(n)\,\big\{\de(s_{\ell})-\de(s_{\ell+1})\big\}+C_*(n)\,(\Gamma+1)\,\Lambda\,s_\ell\,,
\end{eqnarray*}
for a constant $C_*(n)$. By rearranging terms and using the monotonicity of $\de$ on $(R,\infty)$ and $(s_\ell^-,s_\ell^+)\subset(s_\ell,s_{\ell+1})$ we find that
  \begin{eqnarray}\nonumber
  &&\big(C_*(n)\,+(s_{\ell}^+)^n/(s_\ell^n)\big)\,\de(s_{\ell+1})\le C_*(n)\,\de(s_{\ell+1})+\big((s_{\ell}^+)^n/(s_\ell^n)\big)\,\de(s_{\ell}^+)
  \\\nonumber
  &\le&
  C_*(n)\,\de(s_\ell)+\big((s_{\ell}^-)^n/(s_\ell^n)\big)\,\de(s_\ell^-)+C_*(n)\,(1+\Gamma)\,\Lambda \,s_\ell
  \\\label{ext grandi aa lambda}
  &\le&
  \big(C_*(n)\,+(s_{\ell}^-)^n/(s_\ell^n)\big)\,\de(s_{\ell})+C_*(n)\,(1+\Gamma)\,\Lambda \,s_\ell\,.
  \end{eqnarray}
  We finally notice that by \eqref{ext choices 1 lambda}, \eqref{ext choices 2 lambda}, $\eta_0=1/3$, and $\eta=1/6$, we have
  \[
  \frac{s_\ell^-}{s_\ell}=\frac{r_0\,(1-\eta)}{r_0\,(1-\eta_0)}=\frac{5}4\,,\qquad
  \frac{s_\ell^+}{s_\ell}=2\,\frac{s_\ell^+}{s_{\ell+1}}=2\,\frac{1+\eta}{1+\eta_0}=\frac{7}4\,,
  \]
  so that, we find $\de(s_{\ell+1})\le\tau\{\de(s_\ell)+(1+\Gamma)\,\Lambda\,s_\ell\}$ (i.e. \eqref{claim deltaj smaller deltaj-1}$_{j=\ell+1}$) with
  \[
  \tau=\tau(n)=\frac{C_*(n)+(5/4)^n}{C_*(n)+(7/4)^n}\,,\qquad \tau_*=\tau_*(n)=\frac{C_*(n)}{C_*(n)+(7/4)^n}<\tau\,.
  \]


  \noindent {\it To prove  \eqref{claim Tj smaller deltaj-1}$_{j=\ell+1}$}: We finally prove  \eqref{claim Tj smaller deltaj-1}$_{j=\ell+1}$, i.e.
  \begin{equation}
    \label{ext bd decay Tj j ell plus one lambda}  \frac1{s_{\ell+1}^n}\int_{s_{\ell+1}}^{2\,s_{\ell+1}}\!\!\!r^{n-1}\!\!\int_{\Sigma_{H_{\ell+1}}}\!\!\!\!\![u_{\ell+1}]_r^2\le
    C(n)\big\{\de_{V,R,\Lambda}(s_{\ell})-\de_{V,R,\Lambda}(s_{\ell+3})+\Lambda\,s_{\ell}\big\}.
  \end{equation}
  By \eqref{claim is spherical graph 2}$_{j=\ell+1}$ we know that
  \begin{equation}
    \label{ext yes we know lambda}
    \mbox{$V$ corresponds to $\S_{H_{\ell+1}}(u_{\ell+1},s/32,4\,s_{\ell+1})$ on $A_{s/32}^{4\,s_{\ell+1}}$}\,.
  \end{equation}
  Now, \eqref{r1r2r3r4} holds with $r_0=3\,s_\ell$ and $(\eta_0,\eta)=(2/3,1/3)\in J$, see \eqref{def of J lambda}, if
  \begin{eqnarray*}
  &&r_1=s_\ell=3\,s_\ell-2\,s_\ell\,,\qquad r_2=5\,s_{\ell}=3\,s_\ell+2\,s_\ell\,,
  \\
  &&r_3=s_{\ell+1}=3\,s_\ell-s_\ell\,,\qquad\,\,\,\, \,r_4=2\,s_{\ell+1}=3\,s_\ell+s_\ell\,.
  \end{eqnarray*}
  Since $s_\ell^*\in(s_\ell,s_{\ell+1})\subset (r_1,r_2)$,
  by \eqref{ext yes we know lambda}, \eqref{ext u ell plus one 2 lambda} and $(r_1,r_2)\subset (s/32,4\,s_{\ell+1})$ we can apply  Theorem \ref{theorem 7.15AA main estimate lambda}  to deduce that
  \[
  \int_{s_{\ell+1}}^{2\,s_{\ell+1}}\!\!\!r^{n-1}\!\int_{\Sigma_{H_{\ell+1}}}\,[u_{\ell+1}]_r^2\le C(n)\,\int_{s_\ell}^{5\,s_\ell}\!\!\!\!r^{n+1}\!\int_{\Sigma_{H_{\ell+1}}}\!\!\!\!\!\!(\pa_ru_{\ell+1})_r^2
  +C(n)\,\Lambda\,(s_\ell)^{n+1}\,.
  \]
  Again by \eqref{ext yes we know lambda}, Theorem \ref{theorem 7.17AA exterior lambda}-(ii) with $(r_1,r_2)=(s_\ell,8\,s_\ell)$ gives
  \begin{eqnarray*}
  &&s_\ell^{-n}\,\int_{s_\ell}^{5\,s_\ell}\,r^{n+1}\,\int_{\Sigma_{H_{\ell+1}}}\,(\pa_r[u_{\ell+1}])_r^2
  \le
  s_\ell^{-n}\,\int_{s_\ell}^{8\,s_\ell}\,r^{n+1}\,\int_{\Sigma_{H_{\ell+1}}}\,(\pa_r[u_{\ell+1}])_r^2
  \\
  &&\le C(n)\,\big\{\Theta_{V,R,\Lambda}(8\,s_\ell)-\Theta_{V,R,\Lambda}(s_{\ell})\big\}
  \le C(n)\,\big\{\de_{V,R,\Lambda}(s_{\ell})-\de_{V,R,\Lambda}(s_{\ell+3})\big\}\,.
  \end{eqnarray*}
  The last two estimates combined give \eqref{ext bd decay Tj j ell plus one lambda}, which finishes the {\bf claim}.

  \noindent {\it Proof of \eqref{step one conclusion}}: We assume $S_+<\infty$ (that is either $\Lambda>0$ or $R_+<\infty$), and recall that we have already proved \eqref{step one conclusion} if $S_+=4\,s$. Otherwise, $N$ (as defined in \eqref{def of N}) is finite, with $2^N\le \frac{S_+}{s} <2^{N+1}$. By \eqref{claim is spherical graph 1}$_{j=N-2}$ and \eqref{claim is spherical graph 2}$_{j=N-2}$, we have that $u_{N-2}\in \mathcal{X}_{\s}(\S_{H_{N-2}},s/32,4\,s_{N-2})$ and $V$ corresponds to $\Sigma_{H_{N-2}}(u_{N-2},s/32,4\,s_{N-2})$ on $A_{s/32}^{4\,s_{N-2}}$. Since $4\,s_{N-2}
  =2^{N+1}\,s/16> S_+/16$, we deduce \eqref{step one conclusion} with $K_+=H_{N-2}$ and $u_+=u_{N-2}$.



  \noindent \noindent {\it Step two}: In this step we prove statement (ii) in Theorem \ref{theorem mesoscale criterion}. We assume that $\Lambda=0$ and that
  \begin{equation}
    \label{positive}
    \de(r)\ge - \e_0\qquad\forall r\ge \frac{s}8\,,
  \end{equation}
  where we have set for brevity $\de=\de_{V,R,0}$. We must first show that \begin{equation}\label{actually positive}
  \delta(r) \geq 0 \quad \forall r\ge \frac{s}{8}\,.
  \end{equation}
  Since $\delta$ is decreasing in $r$, it has a limit $\lim_{r\to \infty}\delta(r)=:\delta_\infty\geq -\e_0$. Next, we know that for any sequence $R_i\to \infty$, $V/R_i$ converges locally in the varifold sense to a limiting integer rectifiable varifold cone $W$. By the local varifold convergence and $n\geq 2$, $W$ is stationary in $\mathbb{R}^{n+1}$, and it is the case that $$
  \delta_{W,0,0}(r)=\delta_\infty\geq -\e_0 \quad \forall r>0\,.
  $$
  Up to decreasing $\e_0$ if necessary (and recalling that $\delta_{W,0,0}$ is the usual area excess multiplied by $-1$), Allard's theorem and the fact that $W/r = W$ imply that $W$ corresponds to a multiplicity one plane. In particular, it must be that $\delta_\infty=0$, which together with the monotonicity of $\delta$ yields \eqref{actually positive}.
  \par
  By \eqref{actually positive}, $S_+=S_*=\infty$, and so by \eqref{claim is spherical graph 1} and \eqref{claim is spherical graph 2}, there is a sequence  $\{(H_j,u_j)\}_{j=0}^N$ but with $N=\infty$ now, satisfying
  \begin{eqnarray}
  \label{claim is spherical graph fine}
  &&\mbox{$V$ corresponds to $\Sigma_{H_j}(u_j,s/32,4\,s_j)$ on $A_{s/32}^{4\,s_j}$}\quad\forall j\geq 0\,,
  \end{eqnarray}
  \begin{eqnarray}
  \label{claim decay normal vectors fine}
  |\nu_{H_j}-\nu_{H_{j-1}}|^2&\le& C(n)\,T_{j-1}\,,\qquad\hspace{0.8cm}\mbox{if $j\ge 1$}\,,
  \\
  \label{claim delta fine}
  \de(s_j)&\le&
  \left\{
  \begin{split}
  &\e_0\,,\hspace{1.75cm}\qquad\mbox{if $j=0$}\,,
  \\
  &\tau\,\de(s_{j-1})\,,\hspace{0.63cm}\qquad\mbox{if $j\ge1$}\,,
  \end{split}
  \right .
  \\
  \label{claim Tj fine}
  T_j&\le&\left\{
  \begin{split}
  &C(n)\,\e_0\,,\hspace{0.85cm}\qquad\mbox{if $j=0$}\,,
  \\
  &C(n)\,\de(s_{j-1})\,,\qquad\mbox{if $j\ge1$}\,.
  \end{split}
  \right .
  \end{eqnarray}
  Notice that, in asserting the validity of \eqref{claim Tj fine} with $j\ge 1$, we have used \eqref{actually positive} to estimate $-\de(s_{j+2})\le0$ in \eqref{claim Tj smaller deltaj-1}$_{j}$. By iterating \eqref{claim delta fine} we find
  \begin{equation}
    \label{fine deltaj}
      \de(s_j)\le \tau^j\,\de(s/8)\le \tau^j\,\e_0\,,\qquad\forall j\ge 1\,,
  \end{equation}
  which, combined with \eqref{claim Tj fine} and \eqref{claim decay normal vectors fine}, gives, for every $j\ge 1$,
  \begin{eqnarray}\label{fine Tj}
  T_j\le C(n)\,\min\{1,\tau^{j-1}\}\,\de(s/8)\le C(n)\,\tau^j\,\de(s/8)\,,
  \\\label{fine nuj}
  |\nu_{H_j}-\nu_{H_{j-1}}|^2\le C(n)\,\min\{1,\tau^{j-2}\}\,\de(s/8)\le C(n)\,\tau^j\,\de(s/8)\,,
  \end{eqnarray}
  thanks also to $\tau=\tau(n)$ and, again, to \eqref{actually positive}. By \eqref{fine nuj}, for every $j\ge 0$, $k\ge 1$, we have $|\nu_{H_{j+k}}-\nu_{H_j}|
  \le C(n)\,\sqrt{\de(s/8)}\,\sum_{h=1}^{k+1}\big(\sqrt\tau\big)^{j-1+h}$, so that there exists $K\in\H$ such that
  \begin{equation}
    \label{nuK convergence}
    \e_j^2:=|\nu_K-\nu_{H_j}|^2\le C(n)\,\tau^j\,\de(s/8)\,,\qquad\forall j\ge 1\,,
  \end{equation}
  In particular, for $j$ large enough, we have $\e_j<\e_0$, and thus, by Lemma \ref{lemma step one}-(i) and by \eqref{claim is spherical graph fine} we can find $v_j\in\X_{C(n)\,(\s+\e_j)}(\S_{K};s/32,4\,s_{j})$ such that
   \begin{equation}
       \label{claim is spherical graph jjj}
   \mbox{$V$ corresponds to $\Sigma_{K}(v_{j},s/32,4\,s_{j})$ on $A_{s/32}^{4\,s_{j}}$}\,.
  \end{equation}
   By \eqref{claim is spherical graph jjj}, $v_{j+1}=v_{j}$ on $\Sigma_{K}\times(s/32,4\,s_{j})$. Since $s_{j}\to\infty$ we have thus found $u\in\X_{C(n)\,\s}(\S_{K};s/32,\infty)$ such that
   \begin{equation}\label{lambdaext bd easy corresponds final}
   \mbox{$V$ corresponds to $\Sigma_{K}(u,s/32,\infty)$ on $A_{s/32}^\infty$}\,,
   \end{equation}
  which is \eqref{mesoscale thesis graphicality} with $S_*=\infty$.

 %
%
To prove \eqref{decay deficit for exterior minimal surfaces}, we notice that if $r\in(s_j,s_{j+1})$ for some $j\ge1$, then, setting $\tau=(1/2)^\a$ (i.e., $\a=\log_{1/2}(\tau)\in(0,1)$) and noticing that $r/s\le 2^{j+1-3}$, by \eqref{is decreasing} and \eqref{fine deltaj} we have
   \begin{eqnarray}\nonumber
  \de(r)\!\!\!&\le&\!\!\!\de(s_j)\le \tau^j\,\de(s/8)=2^{-j\,a}\,\de(s/8)=
  4^{-\a}\,2^{-(j-2)\a}\,\de(s/8)
  \\\label{ex cont}
  \!\!\!&\le&\!\!\! C(n)\,(s/r)^\a\,\de(s/8)\,,\,\,\,\,\,
  \end{eqnarray}
  where in the last inequality \eqref{positive} was used again; this proves \eqref{decay deficit for exterior minimal surfaces}. To prove \eqref{decay flatness for exterior minimal surfaces}, we recall that $\om_K(y)=\arctan(|\nu_K\cdot\hat{y}|/|\p_K\,\hat{y}|)$, provided $\arctan$ is defined on $\R\cup\{\pm\infty\}$, and where $\hat{y}=y/|y|$, $y\ne 0$. Now, by \eqref{lambdaext bd easy corresponds final},
  \[
  y=|y|\,\frac{\p_K\,\hat{y}+u(\p_K\,\hat{y},|y|)\,\nu_K}{\sqrt{1+u(\p_K\,\hat{y},|y|)^2}}\,,\qquad\forall y\in(\spt\,V)\setminus B_{s/32}\,,
  \]
  so that $|\p_K\,\hat{y}|\ge 1/2$ for $y\in(\spt V)\setminus B_{s/32}$; therefore, by \eqref{nuK convergence}, up to further decreasing the value of $\e_0$, and recalling $\de(s/8)\le\e_0$, we conclude
  \begin{equation}
    \label{nuK uniform projections}
      |\p_{H_j}\,\hat{y}|\ge \frac13\,,\qquad\forall y\in(\spt V)\setminus B_{s/32}\,,
  \end{equation}
  for every $j\in\N\cup\{+\infty\}$ (if we set $H_\infty=K$). By \eqref{nuK uniform projections} we easily find
  \[
  |\om_K(y)-\om_{H_j}(y)|\le C\,|\nu_{H_j}-\nu_K|\,,\qquad\forall y\in(\spt V)\setminus B_{s/32}\,,\forall j\ge 1\,,
  \]
  from which we deduce that, if $j\ge 1$ and $r\in(s_j,s_{j+1})$, then
  \begin{eqnarray*}
    \frac1{r^n}\,\int_{A_r^{2\,r}}\om_K^2\,d\|V\|\!\!\!&\le&\!\!\!C(n)\,\Big\{\frac1{s_j^n}\,\int_{A_{s_j}^{s_{j+1}}}\om_K^2\,d\|V\|+
    \frac1{s_{j+1}^n}\,\int_{A_{s_{j+1}}^{s_{j+2}}}\om_K^2\,d\|V\|\Big\}
    \\
    \!\!\!&\le&\!\!\!C(n)\,\Big\{\frac1{s_j^n}\,\int_{A_{s_j}^{s_{j+1}}}\om_{H_j}^2\,d\|V\|+
    \frac1{s_{j+1}^n}\,\int_{A_{s_{j+1}}^{s_{j+2}}}\om_{H_{j+1}}^2\,d\|V\|\Big\}
    \\
    \!\!\!&&\!\!\!+C(n)\,\Gamma\,\big\{|\nu_K-\nu_{H_j}|^2+|\nu_K-\nu_{H_{j+1}}|^2\big\}\,,
  \end{eqnarray*}
  where \eqref{bounds needed for proof meso} was used to bound $\|V\|(A_\rho^{2\,\rho})\le \Gamma\,(2\,\rho)^n$ with $\rho=s_j,s_{j+1}\in(R,1/\Lambda)$.
  By \eqref{claim is spherical graph fine} we can exploit \eqref{equuivalence between u square and omega square} on the first two integrals, so that taking \eqref{nuK convergence} into account we find that, if $j\ge 1$ and $r\in(s_j,s_{j+1})$, then $r^{-n}\,\int_{A_r^{2\,r}}\om_K^2\,d\|V\|\le C(n)\{T_j+T_{j+1}\big\}+C(n)\,\Gamma\,\tau^j\,\de(s/8)
        \le C(n)\,(1+\Gamma)\,\tau^j\,\de(s/8)$, where in the last inequality we have used \eqref{fine Tj}. Since $\tau^j\le C(n)\,(s/r)^\a$, we conclude the proof of \eqref{decay flatness for exterior minimal surfaces}, and thus, of Theorem \ref{theorem mesoscale criterion}-(ii).

  \par

\noindent\textit{Step three}: In this step, given $n\geq 2$, $\Gamma\geq 0$, and $\s>0$, we claim the existence of $\e_0$ and $M_0$, depending only on $n$, $\Gamma$, and $\s$, such that if \eqref{Gamma bounds}, \eqref{mesoscale bounds}, \eqref{mesoscale delta small s8} and \eqref{mesoscale Rstar larger s4} hold with $\e_0$ and $M_0$, and in addition,
\begin{equation}\label{step three sign assumption}
    -\varepsilon_0\leq \delta_{V,R,\Lambda}(s/8)\leq 0\,,
\end{equation}
then there exists $K_-\in\mathcal{H}$ and $u_-\in \X_\sigma(\Sigma_{K_-},s/32,S_{*}/2
)$ such that
\begin{equation}\label{step three conclusion}
\textup{$V$ corresponds to $\Sigma_{K_-}(u_+,s/32,S_{*}/2)$ on $A_{s/32}^{S_{*}/2}$,}
\end{equation}
where $S_*$ and $R_*$ are as in Theorem \ref{theorem mesoscale criterion}.
The argument is quite similar to that of the first step, with minor differences due to the opposite sign of the deficit. The first is that the iteration instead begins at the outer radius $S_*$ and proceeds inwards via intermediate radii $s_j=2^{-j}S_*$, and the second is that, in the analogue of the graphicality propagation claims \eqref{claim is spherical graph 1}$_{j=\ell+1}$ and \eqref{claim is spherical graph 2}$_{j=\ell+1}$, the negative sign on $\delta_{V,R,\Lambda}$ is used to sum the ``tilting" between successive planes $H_j$ and $H_{j+1}$.

\par

\noindent\textit{Step four}: Finally, we combine steps one and three to prove statement (i) in Theorem \ref{theorem mesoscale criterion}. Before choosing the  parameters $\e_0$ and $M_0$, we need a preliminary result. We claim that for any $\e'>0$, there exists $\s'(\e')>0$ such that if $r_1<r_2$, $K_1$, $K_2\in \mathcal{H}$ with $\nu_{K_1}\cdot \nu_{K_2}\geq 0$ and accompanying $u_i\in \X_{\s'}(\S_{K_i}, r_1 , r_2)$, and $M$ is a smooth hypersurface such that $M\cap A_{r_1}^{r_2}$ corresponds to $\S_{K_i}(u_i,r_1,r_2)$ for $k=1,2$, then
\begin{align}\label{prove planes are close}
    |\nu_{{K_1}}-\nu_{{K_2}}|< \e'\,.
\end{align}
This is immediate from $\nu_{K_1}\cdot \nu_{K_2}\geq 0$ the fact that the $L^\infty$-bounds on $u_i$ imply that $M$ is contained in the intersection of two cones containing $K_1$ and $K_2$, whose openings become arbitrarily narrow as $\s'\to 0$\,.

\par

Fix $n\geq 2$, $\Gamma\geq 0$, and $\s>0$; we assume without loss of generality that $\s<\s_0$, where $\s_0$ is the dimension-dependent constant from Lemma \ref{lemma step one}. We choose $\e'$ with corresponding $\s'$ according to \eqref{prove planes are close} such that, up to decreasing $\s'$ if necessary,
\begin{align}\label{were gonna tilt}
   \e'<\e_0\,,\quad  C_0(\s' + \e') \leq \s\,,
\end{align}
where $\e_0$, $C_0$ are as in Lemma \ref{lemma step one}. Next, we choose $\e_0=\e_0(n,\Gamma,\s)$ and $M_0=M_0(n,\Gamma,\s)$ to satisfy several restrictions: first, $\e_0$ is smaller than the $\e_0$ from Lemma \ref{lemma step one} and each $\e_0(n,\Gamma,\s')$ from steps one and three, and $M_0$ is larger than $M_0(n,\Gamma,\s')$ from those steps; second, with $\e_2$ and $M_2$ as in Lemma \ref{lemma graphicality lambda}-(ii), we also assume that
\begin{equation}\label{step four eps0 M0 lemma}
    \e_0 \leq \min\Big\{\e',\e_2\Big(n,\sigma',\Gamma,\frac{1}{16},\Big(\frac{1}{128},\frac{1}{2}\Big) \Big)\Big\},\, M_0 \geq M_2\Big(n,\sigma',\Gamma,\frac{1}{16},\Big(\frac{1}{128},\frac{1}{2}\Big) \Big)
\end{equation}
In the remainder of this step, we suppose that
\begin{equation}\label{step four V assumptions}
\mbox{$V$ satisfies \eqref{Gamma bounds}, \eqref{mesoscale bounds}, \eqref{mesoscale delta small s8} and \eqref{mesoscale Rstar larger s4} at mesoscale $s$}\,.
\end{equation}
In proving Theorem \ref{theorem mesoscale criterion}-(i), there are three cases depending on whether $\delta_{V,R,\Lambda}$ changes sign on $[s/8, S_*]$.

\par

\noindent\textit{Case one}: $\delta_{V,R,\Lambda}(r)\geq 0$ for all $ r\in[s/8, S_*]$. If the deficit is non-negative, then in particular
\begin{equation}\label{case one step four}
0\leq \delta_{V,R,\Lambda}(s/8)\leq \e_0
\end{equation}
and $S_*=S_+$, where $S_+$ was defined in \eqref{S+ definition}. By our choice of $\e_0$ and $M_0$ at the beginning of this step and the equivalence of \eqref{case one step four} and \eqref{step one positive deficit assumption}, step one applies and the conclusion \eqref{step one conclusion} is \eqref{mesoscale thesis graphicality}. Thus Theorem \ref{theorem mesoscale criterion}-(i) is proved.

\par

\noindent\textit{Case two}: $\delta_{V,R,\Lambda}(r)\leq 0$ for all $r\in [s/8,S_*]$. Should the deficit be non-positive in this interval, then in particular, \eqref{step three sign assumption} holds in addition to \eqref{Gamma bounds}, \eqref{mesoscale bounds}, \eqref{mesoscale delta small s8} and \eqref{mesoscale Rstar larger s4}. Therefore, by our choice of $\e_0$ and $M_0$, step three applies. The conclusion \eqref{step three conclusion} is \eqref{mesoscale thesis graphicality} (in fact with larger upper radii $S_*/2$), and Theorem \ref{theorem mesoscale criterion}-(i) is proved.

\par

\noindent\textit{Case three}: $\delta_{V,R,\Lambda}$ changes sign in $[s/8,S_*]$. By the monotonicity of $\delta_{V,R,\Lambda}$,
\begin{equation}\label{deficit changes sign}
\delta_{V,R,\Lambda}(s/8) > 0 > \delta_{V,R,\Lambda}(S_*)\,.
\end{equation}
First, by \eqref{deficit changes sign}, \eqref{step one positive deficit assumption} is satisfied, so \eqref{step one conclusion} gives $K_+\in \mathcal{H}$ and $u_+\in \X_{\sigma'}(\Sigma_{K_+},s/32,S_{+}/16
)$ such that
\begin{equation}\label{step one conclusion redux}
\textup{$V$ corresponds to $\Sigma_{K_+}(u_+,s/32,S_{+}/16)$ on $A_{s/32}^{S_{+}/16}$,}
\end{equation}
where
\begin{equation}\label{S+ definition redux}
R_+= \max \Big\{\sup\Big\{\rho\ge\frac{s}8: \de_{V,R,\Lambda}(\rho)\ge 0\Big\}, 4\,s\Big\}\,,\quad S_+= \min \Big\{R_+,\frac{\varepsilon_0}{\Lambda} \Big\}\,.
\end{equation}
If $S_+ = S_*$, then \eqref{step one conclusion redux} is \eqref{mesoscale thesis graphicality} and we are done. So we assume for the rest of this case that $S_+<S_*$, which implies $S_+\neq \e_0/\Lambda$ and thus \begin{equation}\label{R+ le Sstar}
4\,s \leq R_+=S_+<S_*\,.
\end{equation}
Next, we make the following

\par

\noindent\textbf{Claim:} There exists $K_-\in \mathcal{H}$ and $u_-\in \X_{\s'}(\Sigma_{K_-},R_+/2,S_{*}/2
)$ such that
\begin{equation}\label{step three conclusion redux}
\textup{$V$ corresponds to $\Sigma_{K_-}(u_-,R_+/2,S_{*}/2)$ on $A_{R_+/2}^{S_{*}/2}$}\,.
\end{equation}
\par

\noindent\textbf{Proof of the claim:} There are two subcases.

\par

\noindent\textit{Subcase one}: $16\,R_+ < \e_0/4\,\Lambda $ and $ 64\,R_+ < R_* $. We claim the conditions of step three are verified at $s'=16\,R_+$. First, \eqref{Gamma bounds} holds from \eqref{step four V assumptions}, and
\begin{equation}\notag
    \max\{64,M_0 \}\,R< 16\,R_+ < \frac{\e_0}{4\,\Lambda}
\end{equation}
(which is \eqref{mesoscale bounds}) holds due to the assumption of the subcase and $16R_+\geq s > \max\{64,M_0 \}R$. Next, $2\, R_+< R_*/4$ by the assumption of the subcase, which combined with the monotonicity of $\delta_{V,R,\Lambda}$ and \eqref{S+ definition redux} gives $-\varepsilon_0\leq \delta_{V,R,\Lambda}(2\,R_+) \leq 0$. This implies \eqref{step three sign assumption} and \eqref{mesoscale delta small s8} with $s'=16\,R_+$. Lastly, \eqref{mesoscale Rstar larger s4} holds at $s'=16\,R_+$ since $64\,R_+ < R_*$. Thus we apply \eqref{step three conclusion} at $s'=16\,R_+$, finding \eqref{step three conclusion redux}.

\par

\noindent\textit{Subcase two}: One or both of $16\,R_+ \geq \e_0/4\,\Lambda$, $64\,R_+ \geq R_*$ hold. In this case,
\begin{equation}\label{inequality of Sstar and R_+}
    64\,R_+ \geq \min \{\e_0 / \Lambda, R_* \}=S_*\,.
\end{equation}
We wish to apply Lemma \ref{lemma graphicality lambda}-(ii) with $r=S_*$, $\lambda_1=\frac{1}{16}$, $(\eta_1,\eta_2) = (\frac{1}{128},\frac{1}{2})$. By \eqref{step four V assumptions}, \eqref{gamma bound lemma statement} holds for $V$, and by \eqref{step four eps0 M0 lemma},
\eqref{step four V assumptions}, and $S_*\geq 4\,s$,
\begin{equation}\notag
    \max \{M_2,64 \}\,R \leq s \leq \frac{S_*}{4} \leq S_*\leq  \frac{\e_2}{\Lambda}\,,
\end{equation}
which is \eqref{alternative conditions on r lambda}. Finally, we have $R_*\geq S_*/16\geq s/8$, so that by the definition of $R_*$, \eqref{mesoscale delta small s8}, the monotonicity of $\delta_{V,R,\Lambda}$, and \eqref{step four eps0 M0 lemma},
\begin{equation}\notag
    \max\Big\{\Big|\delta_{V,R,\Lambda}\Big(\frac{S_*}{16}\Big)\Big|,|\delta_{V,R,\Lambda}(S_*)| \Big\}\leq \e_0\leq \varepsilon_2\,,
\end{equation}
which is \eqref{alternative hypo to angular flatness}. By the choices \eqref{step four eps0 M0 lemma}, Lemma \ref{lemma graphicality lambda}-(ii) applies and yields the existence of $K_-\in\mathcal{H}$ and $u_-\in \X_{\sigma'}(\Sigma_{K_-},S_*/128,S_{*}/2
)$ such that
\begin{equation}\label{step three conclusion subcase two}
\textup{$V$ corresponds to $\Sigma_{K_-}(u_+,S_*/128,S_{*}/2)$ on $A_{S_*/128}^{S_{*}/2}$,}
\end{equation}
By \eqref{inequality of Sstar and R_+}, $S_*/128 \leq R_+/2$, so \eqref{step three conclusion subcase two} implies \eqref{step three conclusion redux}. The proof of the {\bf claim} is complete.

\par

Returning to the proof of Theorem \ref{theorem mesoscale criterion}-(i) under the assumption \eqref{deficit changes sign}, we recall \eqref{R+ le Sstar} and choose $R'\in (R_+, \min\{2\,R_+,S_* \})$. Again, we want to apply Lemma \ref{lemma graphicality lambda}-(ii), this time with $ r=R'$, $\lambda_1=\frac{1}{16}$, and $(\eta_1,\eta_2) = (1/128,1/2)$. To begin with, $V$ satisfies \eqref{gamma bound lemma statement} as usual from \eqref{step four V assumptions}. Second, \eqref{alternative conditions on r lambda} holds at $R'$ by $s\leq R'\leq S_*$, \eqref{step four V assumptions}, and the choices \eqref{step four eps0 M0 lemma}. By the monotonicity of $\delta_{V,R,\Lambda}$ and $[R'/16,R']\subset [R_+/16, S_*] \subset [s/4, S_*]$, \eqref{alternative hypo to angular flatness} is valid by our choice \eqref{step four eps0 M0 lemma} of $\e_0$. The graphicality result from Lemma \ref{lemma graphicality lambda}-(ii) therefore yields $K\in \mathcal{H}$ and $u\in \X_{\s'}(\Sigma_{K},R'/128, R'/2 )$ such that
\begin{equation}\label{double graph correspondence}
    \mbox{$V$ corresponds to $\S_{K}(u,R'/128,R'/2)$ on $A_{R'/128}^{R'/2}$}\,.
\end{equation}
Now $s/32\leq R'/128 < R_+/64<S_+/16$ by $R' < 2\,R_+$ and \eqref{R+ le Sstar}, and $R_+/2<R'/2 <S_*/2$, so by \eqref{step one conclusion redux} and \eqref{step three conclusion redux}, respectively, we have
\begin{eqnarray}\label{K+ corr}
    &&\mbox{$V$ corresponds to $\S_{K_+}(u_+, R'/128, S_+/16)$ on $A_{R'/128}^{S_+/16}$}
    \\ \label{K- corr}
    &&\mbox{$V$ corresponds to $\S_{K_-}(u_-, R_+/2, R'/2)$ on $A_{R_+/2}^{R'/2}$}\,,
\end{eqnarray}
where $u_+\in \X_{\s'}(\S_{K_+},R'/128, S_+/16)$ and $u_-\in \X_{\s'}(\S_{K_-},R_+/2, R'/2)$. Furthermore, up multiplying $\nu_{K_+}$ or $\nu_{K_-}$ by minus one, we may assume $\nu_{K}\cdot \nu_{K_\pm}\geq 0$. Thus $V$ is represented by multiple spherical graphs on nontrivial annuli. 
By combining \eqref{double graph correspondence}, \eqref{K+ corr} and \eqref{K- corr}, $\nu_{K}\cdot \nu_{K_\pm}\geq 0$ and $\s'=\s'(\e')$, \eqref{prove planes are close} applies and gives
\begin{equation}\label{closeness of nus}
    |\nu_{K}-\nu_{K_+}| < \e'\,,\quad |\nu_{K}-\nu_{K_-}| < \e'\,.
\end{equation}
But $\e'$ was chosen according to \eqref{were gonna tilt} so that Lemma \ref{lemma step one}-(i) is applicable; that is, since $\e'<\e_0$ and $\s'<\s_0$ from that lemma, we may reparametrize \eqref{step one conclusion redux} and \eqref{step three conclusion redux}, respectively, as
\begin{eqnarray}\label{K+ corr 2}
    &&\mbox{$V$ corresponds to $\S_{K}(w_+, s/32, S_+/16)$ on $A_{s/32}^{S_+/16}$}
    \\ \label{K- corr 2}
    &&\mbox{$V$ corresponds to $\S_{K}(w_-, R_+/2, S_*/2)$ on $A_{R_+/2}^{S_*/2}$}\,,
\end{eqnarray}
where
\begin{equation}\notag
w_+ \in \X_{C_0(\s' + \e')}(\S_K,s/32, S_+/16)\,,\quad w_- \in \X_{C_0(\s' + \e')}(\S_K,R_+/2, S_*/2)\,.
\end{equation}
By \eqref{were gonna tilt}, $C_0(\s'+\e')\leq \s$, and by $R'/128<S_+/16<R_+/2<R'/2$, \eqref{K+ corr 2} and \eqref{K- corr 2}, we may extend the $u$ defined in \eqref{double graph correspondence} onto $\S_K \times (s/32,S_*/2) $ using $w_+$ and $w_-$ with $C^1$-norm bounded by $\s$. The resulting extension is such that \eqref{mesoscale thesis graphicality} holds, so the proof of Theorem \ref{theorem mesoscale criterion} is finished.

\end{proof}

\section{Application of quantitative isoperimetry}\label{section existence and quantitative isoperimetry} Here we apply quantitative isoperimetry to prove Theorem \ref{thm main psi}-(i) and parts of Theorem \ref{thm main psi}-(iv).

\begin{theorem}\label{thm existence and uniform min}
If $W\subset\R^{n+1}$ is compact, $v>0$, then ${\rm Min}[\psi_W(v)]\ne\emptyset$.  Moreover, depending on $n$ and $W$ only, there are $v_0$, $C_0$, $\Lambda_0$ positive, $s_0\in(0,1)$, and $R_0(v)$  with $R_0(v)\to 0^+$ and $R_0(v)\,v^{1/(n+1)}\to\infty$ as $v\to\infty$, such that, if $v> v_0$ and $E_v$ is a minimizer of $\psi_W(v)$, then:

\noindent {\bf (i):} $E_v$ is a {\bf $(\Lambda_0/v^{1/(n+1)},s_0\,v^{1/(n+1)})$-perimeter minimizer with free boundary in $\Om$}, that is
  \begin{equation}
    \label{uniform lambda minimality}
      P(E_v; \Om\cap B_r(x))\le P(F;\Om\cap B_r(x))+\frac{\Lambda_0}{v^{1/(n+1)}}\,\big|E_v\Delta F\big|\,,
  \end{equation}
  for every $F\subset\Om=\R^{n+1}\setminus W$ with $E_v\Delta F\cc B_r(x)$ and $r<s_0\,v^{1/(n+1)}$;
\noindent {\bf (ii):} $E_v$ determines $x\in\R^{n+1}$ such that
\begin{equation}\label{isop estimate 2}
|E_v\Delta B^{(v)}(x)| \le C_0\,v^{-1+1/[2(n+1)]}\,;
\end{equation}
if $\Rr(W)>0$, then $E_v$ also determines $u\in C^\infty(\pa B^{(1)})$ such that
\begin{eqnarray}
  \label{x and u of Ev take 2}
 &&(\pa E_v)\setminus B_{R_0\,v^{1/(n+1)}}
 \\\nonumber
 &&=\Big\{y+ v^{1/(n+1)}u\Big(\frac{y-x}{v^{1/(n+1)}}\Big)\,\nu_{B^{(v)}(x)}(y):y\in\pa B^{(v)}(x)\Big\}\setminus B_{R_0\,v^{1/(n+1)}}\,;
\end{eqnarray}
\noindent {\bf (iii):} if $\Rr(W)>0$ and $x$ and $u$ depend on $E_v$ as in \eqref{isop estimate 2} and \eqref{x and u of Ev take 2}, then
\begin{equation}\label{limsupmax goes to zero take 2}
\lim_{v\to\infty}\sup_{E_v\in{\rm Min}[\psi_W(v)]}\!\!\!\max\big\{\big||x|\,v^{-1/(n+1)}-\omega_{n+1}^{-1/(n+1)}\big|\,,
  \|u\|_{C^1(\pa B^{(1)})}\big\}=0\,.
\end{equation}
\end{theorem}

\begin{remark}[Improved convergence]\label{remark improved convergence}
  {\rm We will repeatedly use the following fact (see, e.g. \cite{FigalliMaggiARMA,F2M3,CicaleseL.ardi,CiLeMaIC1}): {\it If $\Om$ is an open set, $\Lambda\ge0$, $s>0$, if $\{F_j\}_j$ are {\bf $(\Lambda,s)$-perimeter minimizers in $\Om$}, i.e. if it holds that
  \begin{equation}
    \label{in comparing}
      P(F_j;B_r(x))\le P(G_j; B_r(x))+\Lambda\,|F_j\Delta G_j|\,,
  \end{equation}
  whenever $G_j\Delta F_j\cc B_r(x)\cc \Om$ and $r<s$, and if $F$ is an open set with smooth boundary in $\Om$ such that $F_j\to F$ in $L^1_{{\rm loc}}(\Om)$ as $j\to\infty$, then for every $\Om'\cc \Om$ there is $j(\Om')$ such that
  \[
  (\pa F_j)\cap\Om'= \Big\{y+u_j(y)\,\nu_F(y):y\in \Om\cap\pa F\Big\}\cap\Om'\,,\qquad\forall j\ge j(\Om')\,,
  \]
  for a sequence $\{u_j\}_j\subset C^1(\Om\cap\pa F)$ with $\|u_j\|_{C^1(\Om\cap\pa F)}\to 0$.}
  Compare the terminology used in \eqref{uniform lambda minimality} and \eqref{in comparing}: when we add ``with free boundary'', the ``localizing balls'' $B_r(x)$ are not required to be compactly contained in $\Om$, and the perimeters are computed in $B_r(x)\cap\Om$.
  }
\end{remark}

\begin{proof}[Proof of Theorem \ref{thm existence and uniform min}] {\bf Step one:} We prove ${\rm Min}[\psi_W(v)]\ne\emptyset$ for all $v>0$. Since $W$ is compact, $B^{(v)}(x)\cc\Om$ for $|x|$ large. Hence there is $\{E_j\}_j$ with
  \begin{equation}
    \label{minimizing sequence}
E_j\subset\Om\,,\,\,\, |E_j|=v\,,\,\,\,P(E_j;\Om)\le \min\Big\{P(B^{(v)}),P(F;\Om)\Big\}+(1/j)\,,
  \end{equation}
for every $F\subset\Om$ with $|F|=v$. Hence, up to extracting subsequences, $E_j\to E$ in $L^1_{\rm loc}(\R^{n+1})$ with $P(E;\Omega) \leq \varliminf_{j\to \infty}P(E_j;\Omega)$, where $E\subset\Om$ and $|E|\le v$. We now make three remarks concerning $E$:

\noindent {\bf (a):} {\it If $\{\Omega_i\}_{i\in I}$ are the connected components of $\Om$, then $\Om\cap\pa^*E=\emptyset$ if and only  $E = \bigcup_{i\in I_0} \Omega_i$} ($I_0\subset I$). Indeed, $\Omega \cap \partial^\ast E=\emptyset$ implies $\cl(\pa^*E)\cap \Omega =\partial E\cap \Omega$, hence $\pa E\subset\pa\Om$ and $E = \bigcup_{i\in I_0} \Omega_i$. The converse is immediate.

\noindent {\bf (b):} {\it If $\Om\cap\pa^*E\ne\emptyset$, then we can construct a system of ``volume--fixing variations'' for $\{E_j\}_j$}. Indeed, if $\Om\cap\pa^*E\ne\emptyset$, then there are $B_{S_0}(x_0)\cc\Om$ with $P(E;\pa B_{S_0}(x_0))=0$ and a vector field $X\in C^\infty_c(B_{S_0}(x_0);\R^{n+1})$ such that $\int_E\Div\,X=1$. By \cite[Theorem 29.14]{maggiBOOK}, there are constants $C_0,c_0>0$, depending on $E$ itself, with the following property: whenever $|(F\Delta E)\cap B_{S_0}(x_0)|<c_0$, then there is a smooth function $\Phi^F:\R^n\times(-c_0,c_0)\to\R^n$ such that, for each $|t|<c_0$, the map $\Phi_t^F=\Phi^F(\cdot,t)$ is a smooth diffeomorphism with $\{\Phi_t^F\ne\id\}\cc B_{S_0}(x_0)$, $|\Phi_t^F(F)|=|F|+t$, and $P(\Phi_t^F(F);B_{S_0}(x_0))\le(1+C_0\,|t|)\,P(F;B_{S_0}(x_0))$. For $j$ large enough, we evidently have $|(E_j\Delta E)\cap B_{S_0}(x_0)|<c_0$, and thus we can construct smooth functions $\Phi^j:\R^n\times(-c_0,c_0)\to\R^n$ such that, for each $|t|<c_0$, the map $\Phi_t^j=\Phi^j(\cdot,t)$ is a smooth diffeomorphism with $\{\Phi_t^j\ne\id\}\cc B_{S_0}(x_0)$, $|\Phi_t^j(E_j)|=|E_j|+t$, and $P(\Phi_t^j(E_j);B_{S_0}(x_0))\le(1+C_0\,|t|)\, P(E_j;B_{S_0}(x_0))$.

  \noindent {\bf (c):} {\it If $\Om\cap\pa^*E\ne\emptyset$, then $E$ is bounded}. Since $|E|\le v<\infty$, it is enough to prove that $\Om\cap\pa^*E$ is bounded. In turn, taking $x_0\in\Om\cap\pa^*E$, and since $W$ is bounded and $|E|<\infty$, the boundedness of $\Om\cap\pa^*E$ descends immediately by the following density estimate: there is $r_1>0$ such that
  \begin{equation}
    \label{lower bound volume of E}
    \begin{split}
      &|E\cap B_r(x)|\ge c(n)\,r^{n+1}
      \\
      &\forall\,\,x\in\Om\cap\pa^*E\,,\,\, r<r_1\,,\,\, B_r(x)\cc \R^{n+1}\setminus \big(I_{r_1}(W)\cup B_{S_0}(x_0)\big)\,.
    \end{split}
  \end{equation}
  To prove \eqref{lower bound volume of E}, let $r_1>0$ be such that $|B_{r_1}|<c_0$, let $x$ and $r$ be as in \eqref{lower bound volume of E}, and set $F_j=(\Phi_t^j(E_j)\cap B_{S_0}(x_0))\cup[E_j\setminus(B_r(x)\cup B_{S_0}(x_0))]$ for $t=|E_j\cap B_r(x)|$ (which is an admissible value of $t$ by $|B_{r_1}|<c_0$). In this way, $|F_j|=|E_j|=v$, and thus we can exploit \eqref{minimizing sequence} with $F=F_j$. A standard argument (see, e.g. \cite[Theorem 21.11]{maggiBOOK}) leads then to \eqref{lower bound volume of E}.

Now, since $\pa\Om\subset W$ is bounded, every connected component of $\Om$ with finite volume is bounded. Thus, by (a), (b) and (c) above, there is $R>0$ such that $W\cup E\,\cc \,B_R$. Since $|E\cap [B_{R+1}\setminus B_R]|=0$, we can pick $T\in(R,R+1)$ such that $\H^n(E_j\cap\pa B_T)\to 0$ and $P(E_j\setminus B_T)=\H^n(E_j\cap\pa B_T)+P(E_j;\Om\setminus B_T)$,  and consider the sets $F_j=(E_j\cap B_T)\cup B_{\rho_j}(y)$ corresponding to $\rho_j=(|E_j\setminus B_T|/\om_{n+1})^{1/(n+1)}$ and to $y\in\R^{n+1}$ which is independent from $j$ and such that $|y|>\rho_j+T$ (notice that $\sup_j\rho_j\le C(n)\,v^{1/(n+1)}$). Since $|F_j|=|E_j|=v$, \eqref{minimizing sequence} with $F=F_j$ and $P(B_{\rho_j})\le P(E_j\setminus B_T)$ give
  \begin{eqnarray*}
  P(E_j;\Om)-(1/j)\!\!&\le& \!\!\!\!P(F_j;\Om)\le P(E_j;\Om\cap B_T)+\H^n(E_j\cap\pa B_T)+P(B_{\rho_j})
  \\
  &\le&\!\!\!\!P(E_j;\Om)+2\,\H^n(E_j\cap\pa B_T)\,,
  \end{eqnarray*}
so that, by the choice of $T$, $\{F_j\}_j$ is a minimizing sequence for $\psi_W(v)$, with $F_j\subset B_{T^*}$ and $T^*$ independent of $j$. We conclude by the Direct Method.

  \noindent{\bf Step two:} We prove \eqref{isop estimate 2}. If $E_v$ a minimizer of $\psi_W(v)$ and $R>0$ is such that $W \cc B_R$, then by $P(E_v;\Om)\le P(B^{(v)})$ we have, for $v>v_0$, and $v_0$ and $C_0$ depending on $n$ and $W$,
\begin{eqnarray}
\label{cookie 1}
P(E_v\setminus B_{R})&\le&P(E_v;\Omega) + n\,\om_n\,R^n\le P(B^{(v)})+C_0
\\\nonumber
&\le&(1+(C_0/v))\,P(B^{(|E_v\setminus B_{R}|)})  + C_0\,,
\end{eqnarray}
where we have used that, if $v>2\,b>0$ and $\a=n/(n+1)$, then
\[
P(B^{(v)})\,P(B^{(v-b)})^{-1}-1=(v/(v-b))^\a-1\le \a\,b/(v-b)\le 2\,\a\,b\,v^{-1}\,.
\]
By combining \eqref{quantitative euclidean isop} and \eqref{cookie 1} we conclude that, for some $x\in\R^{n+1}$,
\[
c(n)\,\Big(\frac{|(E_v\setminus B_{R})\Delta B^{(|E_v\setminus B_{R}|)}(x)|}{|E_v \setminus B_{R}|}\Big)^2
 \le \frac{P(E_v\setminus B_{R})}{P(B^{(|E_v\setminus B_{R}|)})}-1\le \frac{C_0}{v^{n/(n+1)}}\,,
\]
provided $v>v_0$. Hence we deduce \eqref{isop estimate 2} from
\begin{eqnarray*}
  &&|E_v \Delta B^{(v)}(x)|=2\,|E_v \setminus B^{(v)}(x)|\le C_0+2\,\big|(E_v\setminus B_R) \setminus B^{(v)}(x)\big|
  \\
  &&\le C_0+2\,\big|(E_v\setminus B_R) \setminus B^{(|E_v\setminus B_{R}|)}(x)\big|\le C_0+|E_v\setminus B_R|\,C_0\,v^{-n/2\,(n+1)}\,.
\end{eqnarray*}

  \noindent {\bf Step three:} We prove the existence of $v_0$, $\Lambda_0$, and $s_0$ such that every $E_v\in{\rm Min}[\psi_W(v)]$ with $v>v_0$ satisfies \eqref{uniform lambda minimality}. Arguing by contradiction, we assume the existence of $v_j\to\infty$, $E_j\in{\rm Min}[\psi_W(v_j)]$, $F_j\subset\Om$ with $|F_j\Delta E_j|>0$ and $F_j\Delta E_j\cc B_{r_j}(x_j)$ for some $x_j\in\R^{n+1}$ and $r_j=v_j^{1/(n+1)}/j$, such that
  \[
  P(E_j;\Om\cap B_{r_j}(x_j))\ge  P(F_j;\Om\cap B_{r_j}(x_j))+j\,v_j^{-1/(n+1)}\,\big|E_j\Delta F_j\big|\,.
  \]
  Denoting by $E_j^*$, $F_j^*$ and $\Om_j$ the sets obtained by scaling $E_j$, $F_j$ and $\Om$ by a factor $v_j^{-1/(n+1)}$, we find that $F_j^*\Delta E_j^*\cc B_{1/j}(y_j)$ for some $y_j\in\R^{n+1}$, and
  \begin{equation}
    \label{unifor min contra scaled}
      P(E_j^*;\Om_j\cap B_{1/j}(y_j))\ge  P(F_j^*;\Om_j\cap B_{1/j}(y_j))+j\,\big|E_j^*\Delta F_j^*\big|\,.
  \end{equation}
  By \eqref{isop estimate 2} there are $z_j\in\R^{n+1}$ such that $|E_j^*\Delta B^{(1)}(z_j)|\to 0$. We can therefore use the volume-fixing variations of $B^{(1)}$ to find diffeomorphisms $\Phi^j_t:\R^n\to\R^n$ and constants $c(n)$ and $C(n)$ such that, for every $|t|<c(n)$, one has $\{\Phi^j_t\ne\id\}\cc U_j$ for some open ball $U_j$ with $U_j\cc \Om_j\setminus B_{1/j}(y_j)$,
  $|\Phi^j_t(E_j^*)\cap U_j|=|E_j^*\cap U_j|+t$, and $P(\Phi^j_t(E_j^*);U_j)\le (1+C(n)\,|t|)\,P(E_j^*;U_j)$.  Since $F_j^*\Delta E_j^*\cc B_{1/j}(y_j)$ implies $||F_j^*|-|E_j^*||<c(n)$ for $j$ large, if $t=|E_j^*|-|F_j^*|$, then $G_j^*=\Phi^j_t(F_j^*)$ is such that $|G_j^*|=|E_j^*|$, and by $E_j\in{\rm Min}[\psi_W(v_j)]$,
  \begin{eqnarray*}
  &&\!\!\!\!\!\!\!\!\!\!\!\!\!\!\!P(E_j^*;\Om_j)\le P(G_j^*;\Om_j)
  \le P\big(E_j^*;\Om_j\setminus(U_j\cup B_{1/j}(y_j))\big)
  \\
  &&+P(F_j^*;\Om_j\cap B_{1/j}(y_j))+P(E_j^*;U_j)+C(n)\,P(E_j^*;U_j)\,\big|E_j^*\Delta F_j^*\big|\,.
  \end{eqnarray*}
  Taking into account $P(E^*_j;U_j)\le \psi_W(v_j)/v_j^{n/(n+1)}\le C(n)$, we thus find
  \begin{eqnarray*}
  P(E_j^*;\Om_j\cap B_{1/j}(y_j))\le P(F_j^*;\Om_j\cap  B_{1/j}(y_j))+C(n)\,\big|E_j^*\Delta F_j^*\big|\,,
  \end{eqnarray*}
  which, by \eqref{unifor min contra scaled}, gives $j\,\big|E_j^*\Delta F_j^*\big|\le C(n)\,\big|E_j^*\Delta F_j^*\big|$. Since $|E_j^*\Delta F_j^*|>0$, this is a contradiction for $j$ large enough.

\noindent {\bf Step four:} We now prove that, if $\Rr(W)>0$, then
\begin{equation}\label{xv equation}
\lim_{v\to\infty}\,\sup_{E_v\in{\rm Min}[\psi_W(v)]}\,\big||x|\,v^{-1/(n+1)}-\omega_{n+1}^{-1/(n+1)}\big|=0\,,
\end{equation}
where $x$ is related to  $E_v$  by \eqref{isop estimate 2}. In proving \eqref{xv equation} we will use the assumption $\Rr(W)>0$ and the energy upper bound
\begin{equation}
  \label{bello}
  \varlimsup_{v\to \infty} \psi_W(v) - P(B^{(v)})\le-\Rr(W)\,.
\end{equation}
A proof of \eqref{bello} is given in step one of the proof of Theorem \ref{thm main psi}, see section \ref{section resolution for exterior}; in turn, that proof is solely based on the results from section \ref{section isoperimetric residues}, where no part of Theorem \ref{thm existence and uniform min} (not even the existence of minimizers in $\psi_W(v)$) is ever used. This said, when $|W|>0$, and thus $\Ss(W)>0$, one can replace \eqref{bello} in the proof of \eqref{xv equation} by the simpler upper bound
\begin{equation}\label{strict}
\varlimsup_{v\to \infty} \psi_W(v) - P(B^{(v)})\le-\Ss(W)\,,
\end{equation}
where, we recall, $\Ss(W)=\sup\{\H^n(W\cap\Pi):\mbox{$\Pi$ is a hyperplane in $\R^{n+1}$}\}$. To prove \eqref{strict}, given $\Pi$, we construct competitors for $\psi_W(v)$ by intersecting $\Om$ with balls $B^{(v')}(x_v)$ with $v'>v$ and $x_v$ such that
$|B^{(v')}(x_v)\setminus W|=v$ and $\H^n(W\cap\pa B^{(v')}(x_v))\to\H^n(W\cap\Pi)$ as $v\to\infty$. Hence,
$\varlimsup_{v\to\infty}\psi_W(v)-P(B^{(v)})\le-\H^n(W\cap\Pi)$,
thus giving \eqref{strict}. The proof of \eqref{bello} is identical in spirit to that of \eqref{strict}, with the difference that to glue a large ball to $(F,\nu)\in{\rm Max}[\Rr(W)]$ we will need to establish the decay of $\pa F$ towards a hyperplane parallel to $\nu^\perp$ to the high degree of precision expressed in \eqref{asymptotics of F}. Now to prove \eqref{xv equation}: by contradiction, consider $v_j\to\infty$, $E_j\in{\rm Min}[\psi_W(v_j)]$, and $x_j\in\R^{n+1}$ with $\inf_{x\in\R^{n+1}}|E_j\Delta B^{(v_j)}(x)|=|E_j\Delta B^{(v_j)}(x_j)|$, such that
\begin{equation}\label{bad assumption on xv}
 \varliminf_{j\to\infty}\big||x_j|\,v_j^{-1/(n+1)}-\omega_{n+1}^{-1/(n+1)}\big|>0\,,
\end{equation}
and set $\l_j=v_j^{-1/(n+1)}$, $E_j^*=\l_j\,(E_j-x_j)$, $W_j^*=\l_j\,(W-x_j)$, and $\Om_j^*=\l_j\,(\Om-x_j)$. By \eqref{uniform lambda minimality}, each $E_j^*$ is a  $(\Lambda_0,s_0)$-perimeter minimizer with free boundary in $\Om_j^*$. By \eqref{isop estimate 2} and the defining property of $x_j$, $E_j^*\to B^{(1)}$ in $L^1(\R^{n+1})$. Moreover, $\diam(W_j^*)\to 0$ and, by \eqref{bad assumption on xv},
\begin{equation}
  \label{because}
  \varliminf_{j\to\infty}\dist\big(W_j^*,\pa B^{(1)}\big)>0\,.
\end{equation}
Thus there is $z_0\not\in\pa B^{(1)}$ such that, for every $\rho<\dist(z_0,\pa B^{(1)})$, there is $j(\rho)$ such that $\{E_j^*\}_{j\ge j(\rho)}$ is a sequence of $(\Lambda_0,s_0)$-perimeter minimizers in $\R^{n+1}\setminus B_{\rho/2}(z_0)$. By Remark \ref{remark improved convergence}, up to increasing $j(\rho)$, $(\pa E_j^*)\setminus B_{\rho}(z_0)$ is contained in the normal graph over $\pa B^{(1)}$ of $u_j$ with $\|u_j\|_{C^1(\pa B^{(1)})}\to 0$; in particular, by \eqref{because}, $(\pa E_j^*)\setminus B_{\rho}(z_0)$ is disjoint from $W_j^*$. By the constant mean curvature condition satisfied by $\Om\cap\pa E_j^*$, and by Alexandrov's theorem \cite{alexandrov}, $(\pa E_j^*)\setminus B_{\rho}(z_0)$ is a sphere $M_j^*$ for $j\ge j(\rho)$. Let $B_j^*$ be the ball bounded by $M_j^*$. Since $M_j^*\cap W_j^*=\emptyset$, we have either one of the following:

\noindent {\bf Case one:} $W_j^*\subset B_j^*$. We have $\pa[B_j^*\cup E_j^*]\subset M_j^*\cup[(\pa E_j^*)\setminus \cl(B_j^*)]\subset (\pa E_j^*)\setminus W_j^*$, so that, by $|B_j^*\cup E_j^*|\ge |E_j^*|+|W_j^*|\ge 1$, we find $P(E_j^*;\Om_j^*)\ge P(B_j^*\cup E_j^*)\ge P(B^{(1)})$, that is, $\psi_W(v_j)\ge P(B^{(1)})$, against \eqref{bello}.

\noindent {\bf Case two:}  $W_j^*\cap B_j^*=\emptyset$. In this case, $E_j^*=B_j^*\cup G_j^*$, where $G_j^*$ is the union of the connected components of $E_j^*$ whose boundaries have non-empty intersection with $W_j^*$: in other words, we are claiming that $B_j^*$ is the only connected component of $E_j^*$ whose closure is disjoint from $W_j^*$. Indeed, if this were not the case, we could recombine all the connected components of $E_j^*$ with closure disjoint from $W_j^*$ into a single ball of same total volume, centered far away from $W_j^*$, in such a way to strictly decrease $P(E_j^*;\Om_j^*)$, against $E_j\in{\rm Min}[\psi_W(v_j)]$. Let us now set
$G_j=x_j+v_j^{1/(n+1)}\,G_j^*$ and $U_j=x_j+v_j^{1/(n+1)}\,B_j^*$, so that $E_j=G_j\cup U_j$ and $\dist(G_j,U_j)>0$.

If we start sliding $U_j$ from infinity towards $G_j\cup W$ along arbitrary directions, then at least one of the resulting ``contact points'' $z_j$ belongs to $\Om\cap\pa G_j$: if this were not the case, then $G_j$ would be contained in the convex envelope of $W$, so that $|B_j|=|E_j|-|G_j|\ge v_j-C(W)$, and thus, by $\psi_W(v_j)=P(E_j;\Om)\ge P(B_j;W)=P(B_j)$, and by $P(B_j)\ge P(B^{(v_j-C(W))})\ge P(B^{(v_j)})-C(W)\,v_j^{-1/(n+1)}$, against with \eqref{bello} for $j$ large.

By construction, there is a half-space $H_j$ such that $G_j\subset H_j$, $z_j\in(\pa G_j)\cap(\pa H_j)$, and $G_j$ is a perimeter minimizer in $B_r(z_j)$ for some small $r>0$. By the strong maximum principle, see, e.g. \cite[Lemma 2.13]{dephilippismaggiCAP-ARMA}, $G_j$ has $H_j-z_j$ as its unique blowup at $z_j$. By De Giorgi's regularity theorem, see e.g. \cite[Part III]{maggiBOOK}, $G_j$ is an open set with smooth boundary in a neighborhood of $z_j$. Therefore, if we denote by $U_j'$ the translation of $U_j$ constructed in the sliding argument, then, $E_j'=G_j\cup U_j'\in{\rm Min}[\psi_W(v)]$ and, in a neighborhood of $z_j$, $E_j'$ is the union of two disjoint sets with smooth boundary which touch tangentially at $z_j$. In particular, $|E_j'\cap B_r(z_j)|/|B_r|\to 1$ as $r\to 0^+$, against volume density estimates implied by \eqref{uniform lambda minimality}, see, e.g. \cite[Theorem 21.11]{maggiBOOK}.

\noindent {\bf Step five:} We finally show the existence of $v_0$ and $R_0(v)$ with $R_0(v)\to 0^+$ and $R_0(v)\,v^{1/(n+1)}\to \infty$, such that each $E_v\in{\rm Min}[\psi_W(v)]$ with $v>v_0$ determines $x$ and $u\in C^{\infty}(\partial B^{(1)})$ such that \eqref{x and u of Ev take 2} holds and $\sup_{E_v} \|u \|_{C^1(\partial B^{(1)})}\to 0$ as $v\to\infty$. To this end, let us consider $v_j\to\infty$, $E_j\in{\rm Min}[\psi_W(v_j)]$, and define $x_j$, $E_j^*$ and $W_j^*$ as in step four. Thanks to \eqref{xv equation}, there is $z_0\in\pa B^{(1)}$ s.t. $\dist(z_0,W_j^*)\to 0$. In particular, for every $\rho>0$, we can find $j(\rho)\in\N$ such that if $j\ge j(\rho)$, then $E_j^*$ is a $(\Lambda_0,s_0)$-perimeter minimizer in $\R^{n+1}\setminus B_\rho(z_0)$, with $E_j^*\to B^{(1)}$. By Remark \ref{remark improved convergence}, there are $u_j\in C^1(\pa B^{(1)})$ such that
\[
(\pa E_j^*)\setminus B_{2\,\rho}(z_0)=\big\{y+u_j(y)\,\nu_{B^{(1)}}(y):y\in\pa B^{(1)}\big\}\setminus B_{2\,\rho}(z_0)\,,\,\,\forall j\ge j(\rho)\,,
\]
and $\|u_j\|_{C^1(\pa B^{(1)})}\to 0$. By the arbitrariness of $\rho$ and by a contradiction argument, \eqref{x and u of Ev take 2} holds with $R_0(v)\to 0^+$ such that $R_0(v)\,v^{1/(n+1)}\to\infty$ as $v\to\infty$, and with the uniform decay of $\|u \|_{C^1(\partial B^{(1)})}$.
\end{proof}

\section{Properties of isoperimetric residues}\label{section isoperimetric residues} Here we prove Theorem \ref{thm main of residue}. It will be convenient to introduce some notation for cylinders and slabs in $\R^{n+1}$: precisely, given $r>0$, $\nu\in\SS^n$ and $I\subset\R$, and setting $\pp_{\nu^\perp}(x)=x-(x\cdot\nu)\,\nu$ ($x\in\R^{n+1}$), we let
\begin{eqnarray}\nonumber
\textbf{D}_r^\nu&=&\big\{x\in\R^{n+1}:|\pp_{\nu^\perp}x|<r\,,x\cdot\nu=0\big\}\,,
\\\nonumber
\textbf{C}_r^\nu&=&\big\{x\in\R^{n+1}:|\pp_{\nu^\perp}x|<r\big\}\,,
\\\label{cylinders and slabs}
\textbf{C}_{r,I}^\nu&=&\big\{x\in\R^{n+1}:|\pp_{\nu^\perp}x|<r\,,x\cdot\nu\in I\big\}\,,
\\\nonumber
\partial_\ell \textbf{C}_{r,I}^\nu&=&\big\{x\in\R^{n+1}:|\pp_{\nu^\perp}x|=r\,,x\cdot\nu\in I\big\}\,,
\\\nonumber
\textbf{S}_{I}^\nu &=& \big\{x\in\R^{n+1}: x \cdot \nu \in I \big\}\,.
\end{eqnarray}
In each case, given $x\in\R^{n+1}$, we also set $\textbf{D}_r^\nu(x)=x+\textbf{D}_r^\nu$, $\textbf{C}_r^\nu(x)=x+\textbf{C}_r^\nu$, etc. We premise the following proposition, used in the proof of Theorem \ref{thm main of residue} and Theorem \ref{thm main psi}, and based on \cite[Proposition 1 and Proposition 3]{Scho83}.

\begin{proposition}\label{prop schoen}
  Let $n\ge 2$, $\nu\in\SS^n$, and let $f$ be a Lipschitz solution to the minimal surface equation on $\nu^\perp\setminus\cl(\DD_R^\nu)$.
  If $n=2$, assume in addition that $M=\{x+f(x)\,\nu:|x|>R\}$ is stable and has natural area growth, i.e.
  \begin{eqnarray}
    \label{schoen 1}
    \int_M\,|\nabla^M\vphi|^2-|A|^2\,\vphi^2\ge0\,,&&\qquad\forall\vphi\in C^1_c(\R^3\setminus B_R)\,,
    \\
    \label{schoen 2}
    \H^2(M\cap B_r)\le C\,r^2\,,&&\qquad\forall r>R\,.
  \end{eqnarray}
  Then there are $a,b\in\R$ and $c\in\nu^\perp$ such that, for every $|x|>R$,
  \begin{eqnarray}
    \label{schoen conclusion 1}
    &&\big|f(x)-\big(a+b\,|x|^{2-n}+(c\cdot x)\,|x|^{-n}\big)\big|\le C\,|x|^{-n}\,,\,\,(n\ge 3)
    \\
    \label{schoen conclusion 2}
    &&\big|f(x)-\big(a+b\,\log|x|+(c\cdot x)\,|x|^{-2}\big)\big|\le C\,|x|^{-2}\,,\,\, (n=2)
    \\
    \label{schoen derivatives}
    &&\max\Big\{|x|^{n-1}\,|\nabla f(x)|,|x|^n\,|\nabla^2f(x)|:|x|>R\Big\}\le C\,,\,\,(\mbox{every $n$})\,.
  \end{eqnarray}
\end{proposition}

\begin{proof} If $n\ge 3$, the fact that $\nabla f$ is bounded allows one to represent $f$ as the convolution with a singular kernel which, by a classical result of Littman,  Stampacchia, and Weinberger \cite{LSWannSNS}, is comparable to the Green's function of $\R^n$; \eqref{schoen conclusion 1} is then deduced starting from that representation formula. For more details, see \cite[Proposition 3]{Scho83}. In the case $n=2$, by \eqref{schoen 1} and \eqref{schoen 2}, we can exploit a classical ``logarithmic cut-off argument'' to see that $M$ has finite total curvature, i.e. $\int_M |K| \,d\H^2< \infty$, where $K$ is the Gaussian curvature of $M$. As a consequence, see, e.g. \cite[Section 1.2]{PerezRos}, the compactification $\ov{M}$ of $M$ is a Riemann surface with boundary, and $M$ is conformally equivalent to $\ov{M}\setminus\{p_1,...,p_m\}$, where $p_i$ are interior points of $\ov{M}$. One can thus conclude by the argument in \cite[Proposition 1]{Scho83} that $M$ has $m$-many ends satisfying the decay \eqref{schoen conclusion 2}, and then that $m=1$ thanks to the fact that $M=\{x+f(x)\,\nu:|x|>R\}$.
\end{proof}

\begin{proof}[Proof of Theorem \ref{thm main of residue}] {\bf Step one:} Given a hyperplane $\Pi$ in $\R^{n+1}$, if $F$ is a half-space with $\pa F=\Pi$ and $\nu$ is a unit normal to $\Pi$, then ${\rm res}_W(F,\nu)=\H^n(W\cap\Pi)$. Therefore the lower bound in \eqref{R larger than S} follows by
\begin{equation}
  \label{famous lb}
\Rr(W)\ge\Ss(W)=\sup\big\{\H^n(\Pi\cap W):\mbox{$\Pi$ an hyperplane in $\R^{n+1}$}\big\}\,.
\end{equation}

\noindent {\bf Step two:} We notice that, if  $(F,\nu)\in\F$, then by \eqref{def Sigma nu 1}, \eqref{def Sigma nu 2}, and the divergence theorem (see, e.g., \cite[Lemma 22.11]{maggiBOOK}), we can define a Radon measure on the open set $\nu^\perp\setminus\pp_{\nu^\perp}(W)$ by setting
\begin{equation}
  \label{is a radon measure}
  \mu(U)=P\big(F;(\pp_{\nu^\perp})^{-1}(U)\big)-\H^n(U)\,,\qquad U\subset\nu^\perp\setminus\pp_{\nu^\perp}(W)\,.
\end{equation}
In particular, setting $R'=\inf\{\rho: W\subset\CC_\rho^\nu\}$, $\mu(\DD_R^\nu\setminus\pp_{\nu^\perp}(W))\ge0$  gives
\[
P(F;\CC_R^\nu\setminus W)\ge\om_n\,R^n-\H^n(\pp_{\nu^\perp}(W))\,,\qquad\forall R>R'\,,
\]
while the identity
\begin{eqnarray*}
  \om_n\,R^n-P(F;\CC_R^\nu\setminus W)=-\mu(\DD_R^\nu\setminus\DD_{R'}^\nu)+\om_n\,(R')^n-P(F;\CC_{R'}^\nu\setminus W)
\end{eqnarray*}
(which possibly holds as $-\infty=-\infty$ if $P(F;\CC_{R'}^\nu\setminus W)=+\infty$) gives that
\begin{equation}
  \label{perimeter is decreasing}
  R\in(R',\infty)\mapsto \om_n\,R^n-P(F;\CC_R^\nu\setminus W)\,\,\,
  \mbox{is {\it decreasing} on $(R',\infty)$}\,.
\end{equation}
In particular, the limsup defining ${\rm res}_W$ always exists as a limit.

\noindent {\bf Step three:} We prove the existence of $(F,\nu)\in{\rm Max}[\Rr(W)]$ and \eqref{local perimeter minimizer}. We first claim that if $\{(F_j,\nu_j)\}_j$ is a maximizing sequence for $\Rr(W)$, then, in addition to $\pp_{\nu_j^\perp}(\pa F_j)=\nu_j^\perp$, one can modify $(F_j,\nu_j)$, preserving the optimality in the limit $j\to \infty$, so that (setting $X\subset^{\L^{n+1}} Y$ for $|X\setminus Y|=0$)
\begin{eqnarray}
  \label{properties of Fj updated}
&&\pa F_j\subset \textbf{S}_{[A_j,B_j]}^{\nu_j}\,,\,\,\, \textbf{S}_{(-\infty, A_j)}^{\nu_j}\stackrel{\L^{n+1}}{\subset}F_j\,,
\,\,\, \textbf{S}_{(B_j,\infty)}^{\nu_j} \stackrel{\L^{n+1}}{\subset}\R^{n+1}\setminus F_j\,,\hspace{0.5cm}
\\
  \label{E containment}
&&\mbox{where}\,\,[A_j,B_j]=\bigcap\big\{(\a,\b):W\subset \textbf{S}_{(\a,\b)}^{\nu_j}\big\}\,.
\end{eqnarray}
Indeed, since $(F_j,\nu_j)\in\F$, for some $\a_j<\b_j\in\R$ we have
\begin{equation}
  \label{properties of Fj}
  \pa F_j\subset \textbf{S}_{[\a_j,\beta_j]}^{\nu_j}\,,\qquad \pp_{\nu_j^\perp}(\pa F_j)=\nu_j^\perp\,.
\end{equation}
Would it be that either $\textbf{S}_{(-\infty, \a_j)\cup(\b_j,\infty)}^{\nu_j} \subset_{\L^{n+1}} F_j$ or $\textbf{S}_{(-\infty, \a_j)\cup(\b_j,\infty)}^{\nu_j} \subset_{\L^{n+1}} \R^{n+1}\setminus F_j$, then, by the divergence theorem and by $\pp_{\nu_j^\perp}(\pa F_j)=\nu_j^\perp$,
\[
P(F_j;\CC_R^{\nu_j}\cap\Om)\ge 2\,\big(\om_n\,R^n-\H^n(\pp_{\nu_j^\perp}(W))\big)\,,\qquad\forall R>0\,,
\]
and thus ${\rm res}_{W}(F_j,\nu_j)=-\infty$; in particular, $(F_j,\nu_j)\in\F$ being a maximizing sequence, we would have $\Rr(W)=-\infty$, against \eqref{famous lb}. This proves the validity (up to switching $F_j$ with $\R^{n+1}\setminus F_j$), of the inclusions
\begin{equation}
  \label{properties of Fj updated new}
  \textbf{S}_{(-\infty, \a_j)}^{\nu_j} \subset_{\L^{n+1}} F_j\,,
\qquad \textbf{S}_{(\b_j,\infty)}^{\nu_j} \subset_{\L^{n+1}} \R^{n+1}\setminus F_j\,.
\end{equation}
Thanks to \eqref{properties of Fj updated new} (and by exploiting basic set operations on sets of finite perimeter, see, e.g., \cite[Theorem 16.3]{maggiBOOK}),  we see that
\begin{eqnarray}\label{truncate up to the obstacle}
&&\mbox{$F_j^*=\big(F_j\cup \textbf{S}_{(-\infty,A_j-1/j)}^{\nu_j}\big)\cap \textbf{S}_{(-\infty,B_j+1/j)}^{\nu_j}$ satisfies}
\\\nonumber
&&(F_j^*,\nu_j)\in\F\,,\qquad P\big(F_j^*;\CC_R^{\nu_j}\setminus W\big)\le P\big(F_j;\CC_R^{\nu_j}\setminus W\big)\,,\qquad\forall R>0\,;
\end{eqnarray}
in particular, $\{(F_j^*,\nu_j)\}_j$ is also a maximizing sequence for $\Rr(W)$. By standard compactness theorems there are $F$ of locally finite perimeter in $\R^{n+1}$ and $\nu\in\SS^n$ such that $F_j\to F$ in $L^1_{{\rm loc}}(\R^{n+1})$ and $\nu_j\to\nu$. If $A\cc\CC_R^\nu\setminus W$ is open, then, for $j$ large enough, $A\cc \CC_R^{\nu_j}\setminus W$, and thus
\begin{equation}
  \label{liminf on cylinders}
P(F;\CC_R^\nu\setminus W)=\sup_{A\cc\CC_R^\nu\setminus W}\,P(F;A)\le\varliminf_{j\to\infty}P(F_j;\CC_R^{\nu_j}\setminus W)\,.
\end{equation}
By \eqref{perimeter is decreasing}, $R\mapsto \om_n\,R^n-P(F_j;\CC_R^{\nu_j}\setminus W)$ is decreasing on $R>R_j=\inf\{\rho:W\subset \CC_\rho^{\nu_j}\}$. By $\sup_jR_j\le C(W)<\infty$ and \eqref{liminf on cylinders} we have
\[
\om_n\,R^n-P(F;\CC_R^\nu\setminus W)\ge\varlimsup_{j\to\infty}\om_n\,R^n-P(F_j;\CC_R^{\nu_j}\setminus W)\ge\varlimsup_{j\to\infty}{\rm res}_W(F_j,\nu_j)\,,
\]
for every $R>C(W)$; in particular, letting $R\to\infty$,
\begin{equation}
  \label{max inq}
  {\rm res}_W(F,\nu)\ge\varlimsup_{j\to\infty}{\rm res}_W(F_j,\nu_j)=\Rr(W)\,.
\end{equation}
By $F_j\to F$ in $L^1_{{\rm loc}}(\R^{n+1})$, $\pa F=\cl(\pa^*F)$ is contained in the set of accumulation points of sequences $\{x_j\}_j$ with $x_j\in\pa F_j$, so that \eqref{properties of Fj updated} gives
\begin{equation}
  \label{properties of F updated}
  \pa F\subset \textbf{S}_{[A,B]}^{\nu}\,,\qquad \textbf{S}_{(-\infty, A)}^{\nu} \subset_{\L^{n+1}} F\,,
\qquad \textbf{S}_{(B,\infty)}^{\nu} \subset_{\L^{n+1}}\R^{n+1}\setminus F\,,
\end{equation}
if $[A,B]=\bigcap\{(\a,\b):W\subset\textbf{S}^\nu_{(\a,\b)}\}$. Therefore $(F,\nu)\in\F$, and thus, by \eqref{max inq}, $(F,\nu)\in{\rm Max}[\Rr(W)]$. We now show that \eqref{max inq} implies \eqref{local perimeter minimizer}, i.e.
\begin{equation}\label{local perimeter minimizer proof}
P(F;\Omega \cap B) \leq P(G;\Omega \cap B)\,,\qquad\mbox{$\forall F\Delta G\cc B$, $B$ a ball}\,.
\end{equation}
Indeed, should \eqref{local perimeter minimizer proof} fail, we could find $\de>0$ and $G\subset\R^{n+1}$ with $F\Delta G\cc B$ for some ball $B$, such that $P(G;B\setminus W)+\de\le P(F;B\setminus W)$. For $R$ large enough to entail $B\cc \CC_R^\nu$ we would then find
\[
{\rm res}_W(F,\nu)+\de\le\om_n\,R^n-P(F;\CC_R^\nu\setminus W)+\de\le\om_n\,R^n-P(G;\CC_R^\nu\setminus W)\,,
\]
which, letting $R\to\infty$, would violate the maximality of $(F,\nu)$ in $\Rr(W)$.

\noindent {\bf Step four:} We show that if $\Rr(W)>0$ and $(F,\nu)\in{\rm Max}[\Rr(W)]$, then $\partial F\subset\mathbf{S}_{[A,B]}^\nu$ for $A,B$ as in \eqref{properties of F updated}. Otherwise, by the same truncation procedure leading to \eqref{truncate up to the obstacle} and by $(F,\nu)\in{\rm Max}[\Rr(W)]$, we would find
\begin{equation*}
\omega_n R^n - P\big(F^*;\CC_R^{\nu_j}\setminus W\big)\geq \omega_n R^n -P\big(F;\CC_R^{\nu_j}\setminus W\big)\geq \mathcal{R}(W) \qquad\forall R>0\,,
\end{equation*}
so that $(F^\ast,\nu)\in{\rm Max}[\Rr(W)]$ too. Now $P\big(F;\CC_R^{\nu_j}\setminus W\big) - P\big(F^*;\CC_R^{\nu_j}\setminus W\big)  $ is increasing in $R$, and since ${\rm res}_W(F,\nu) = {\rm res}_W(F^\ast,\nu)$, it follows that $P\big(F;\CC_R^{\nu_j}\setminus W\big) = P\big(F^*;\CC_R^{\nu_j}\setminus W\big) $ for large $R$. But this can hold only if $\partial F \cap \Omega$ is an hyperplane disjoint from $W$, in which case $\mathcal{R}(W)={\rm res}_W(F,\nu)=0$.

\noindent {\bf Step five:} Still assuming $\Rr(W)>0$, we complete the proof of statement (ii) by proving \eqref{asymptotics of F}. By \eqref{properties of F updated}, if $(F,\nu)\in{\rm Max}[\Rr(W)]$, then $F/R\to H^-=\{x\in\R^{n+1}:x\cdot\nu<0\}$ in $L^1_{{\rm loc}}(\R^{n+1})$ as $R\to\infty$. By \eqref{local perimeter minimizer proof} and by improved convergence (i.e., Remark \ref{remark improved convergence} -- notice carefully that $\pa F$ is bounded in the direction $\nu$ thanks to step four), we find $R_F>0$ and functions $\{f_R\}_{R>R_F}\subset C^1(\DD_2^\nu\setminus\DD_1^\nu)$ such that
\[
\big(\CC_2^\nu\setminus\CC_1^\nu\big)\cap\pa (F/R)=\big\{x+f_R(x)\,\nu:x\in \DD_2^\nu\setminus\DD_1^\nu\big\}\,,\qquad\forall R>R_F\,.
\]
with $\|f_R\|_{C^1(\DD_2^\nu\setminus\DD_1^\nu)}\to 0$ as $R\to\infty$. Scaling back to $F$ we deduce that
\begin{equation}
  \label{F represented by u}
  (\pa F)\setminus\CC_{R_F}^\nu=\big\{x+f(x)\,\nu:x\in\nu^\perp\setminus\DD_{R_F}^\nu\big\}\,,
\end{equation}
for a (necessarily smooth) solution $f$ to the minimal surfaces equation with
\begin{equation}
  \label{u estimates}
  \|f\|_{C^0(\nu^\perp\setminus\DD_{R_F}^\nu)}\le B-A\,,\qquad
  \lim_{R\to\infty}\|\nabla f\|_{C^0(\DD_{2\,R}^\nu\setminus \DD_{R}^\nu)}=0\,,
\end{equation}
thanks to the fact that $f(x)=R\,f_R(x/R)$ if $x\in\DD_{2\,R}^\nu\setminus \DD_{R}^\nu$. {\bf When $n\ge 3$}, \eqref{asymptotics of F} follows by \eqref{F represented by u} and Proposition \ref{prop schoen}. {\bf When $n=2$}, \eqref{schoen 1} holds by \eqref{local perimeter minimizer proof}. To check \eqref{schoen 2}, we deduce by ${\rm res}_W(F,\nu)\ge0$ the existence of $R'>R_F$ such that $\om_n\,R^n\ge P(F;\CC_R^\nu\setminus W)-1$ if $R>R'$. In particular, setting $M=(\pa F)\setminus B_{R_F}$, for $R>R'$ we have
\[
\H^2(M\cap B_R)\le \H^2(M\cap W)+P(F;\CC_R^\nu\setminus W)\le \om_n\,R^n+1+\H^2(M\cap W)\le C\,R^n\,,
\]
provided $C=\om_n+[(1+\H^2(M\cap W))/(R')^n]$; while if $R\in(R_F,R')$, then $\H^2(M\cap B_R)\le C\,R^n$ with $C=\H^2(M\cap B_{R'})/R_F^n$. This said, we can apply Proposition \ref{prop schoen} to deduce \eqref{schoen conclusion 2}. Since $\pa F$ is bounded in a slab, the logarithmic term in \eqref{schoen conclusion 2} must vanish (i.e. \eqref{schoen conclusion 2} holds with $b=0$), and thus \eqref{asymptotics of F} is proved. {\bf Finally, when $n=1$}, by \eqref{F represented by u} and \eqref{u estimates} there are $a_1,a_2\in\R$, $x_1<x_2$, $x_1,x_2\in\nu^\perp\equiv\R$ such that $f(x)=a_1$ for $x\in\nu^\perp$, $x<x_1$, and $f(x)=a_2$ for $x\in\nu^\perp$, $x>x_2$. Now, setting $M_1=\{x+a_1\,\nu:x\in\nu^\perp,x<x_1\}$ and $M_2=\{x+a_2\,\nu:x\in\nu^\perp,x>x_2\}$, we have that
\[
P(F;\CC_R^\nu\setminus W)=\H^n\big(\CC_R^\nu\cap(\pa F)\setminus(W\cup M_1\cup M_2)\big)+2\,R-|x_2-x_1|\,;
\]
while, if $L$ denotes the line through $x_1+a_1\,\nu$ and $x_2+a_2\,\nu$, then we can find $\nu_L\in\SS^1$ and a set $F_L$ such that $(F_L,\nu_L)\in\F$ with $\pa F_L=\big[\big((\pa F)\setminus(M_1\cup M_2)\big)\cup (L_1\cup L_2)\big]$, where $L_1$ and $L_2$ are the two half-lines obtained by removing from $L$ the segment joining $x_1+a_1\,\nu$ and $x_2+a_2\,\nu$. In this way, $P(F_L;\CC_R^{\nu_L}\setminus W)=\H^n\big(\CC_R^\nu\cap(\pa F)\setminus(W\cup M_1\cup M_2)\big)+2\,R-\big|(x_1+a_1\,\nu)-(x_2+a_2\,\nu)\big|$, so that ${\rm res}_W(F_L,\nu_L)-{\rm res}_W(F,\nu)=\big|(x_1+a_1\,\nu)-(x_2+a_2\,\nu)\big|-|x_2-x_1|>0$, against $(F,\nu)\in{\rm Max}[\Rr(W)]$ if $a_1\ne a_2$. Hence, $a_1=a_2$.

We are left to prove that \eqref{F represented by u} holds with $R_2=R_2(W)$ in place of $R_F$, and the constants $a$, $b$, $c$ and $C_0$ appearing in \eqref{asymptotics of F} can be bounded in terms of $W$ only. To this end, we notice that the argument presented in step one shows that ${\rm Max}[\Rr(W)]$ is pre-compact in $L^1_{\rm loc}(\R^{n+1})$. Using this fact and a contradiction argument based on improved convergence (Remark \ref{remark improved convergence}), we conclude the proof of statement (ii).

\noindent {\bf Step six:} We complete the proof of statement (i) and begin the proof of statement (iii) by showing that, setting for brevity $d=\diam(W)$, it holds
\begin{equation}
\label{RW upper}
\H^n(W\cap\Pi)\le\Rr(W) \leq \sup_{\nu\in\SS^n}\hn(\pp_{\nu^\perp} (W))\leq\om_n\,(d/2)^n\,,
\end{equation}
whenever $\Pi$ is a hyperplane in $\R^{n+1}$. We have already proved the first inequality in step one. To prove the others, we notice that, if $(F,\nu)\in\F$, then $\pp_{\nu^\perp}(\pa F)=\nu^\perp$ and \eqref{perimeter is decreasing}
give, for every $R>R'$,
\begin{eqnarray}
  \nonumber
  &&\!\!\!\!\!\!\!\!\!\!-{\rm res}_W(F,\nu)
  \ge
  P(F;\CC_R^\nu\setminus W)-\om_n\,R^n
  \ge\H^n\big(\pp_{\nu^\perp}(\pa F\setminus W)\cap\DD_R^\nu\big)-\om_n\,R^n
  \\
  \label{projection type inequality}
  &&\!\!\!\!=-\H^n\big(\DD_R^\nu\setminus\pp_{\nu^\perp}(\pa F\setminus W)\big)
  \ge-\H^n(\pp_{\nu^\perp}(W))\ge-\om_n\,(d/2)^n\,,
\end{eqnarray}
where in the last step we have used the isodiametric inequality. Maximizing over $(F,\nu)$ in \eqref{projection type inequality} we complete the proof of \eqref{RW upper}. Moreover, if $W=\cl(B_{d/2})$, then, since $\Ss(\cl(B_{d/2}))=\H^n(\cl(B_{d/2})\cap\Pi)=\om_n\,(d/2)^n$ for any hyperplane $\Pi$ through the origin, we find that $\Rr(\cl(B_{d/2}))=\om_n\,(d/2)^n$; in particular, \eqref{RW upper} implies \eqref{optimal RW}.

\noindent {\bf Step seven:} We continue the proof of statement (iii) by showing \eqref{characterization 2}. Let $\Rr(W)=\om_n\,(d/2)^n$ and let $(F,\nu)\in{\rm Max}[\Rr(W)]$. Since every inequality in \eqref{projection type inequality} holds as an equality, we find in particular that
\begin{eqnarray}\label{dai 1}
  &&\sup_{R>R'}P(F;\CC_R^\nu\setminus W)-\H^n\big(\pp_{\nu^\perp}(\pa F\setminus W)\cap\DD_R^\nu\big)=0\,,
  \\\label{dai 3}
  &&\H^n(\pp_{\nu^\perp}(W))=\om_n\,(d/2)^n\,.
\end{eqnarray}
By \eqref{dai 3} and the discussion of the equality cases for the isodiametric inequality (see, e.g. \cite{maggiponsiglionepratelli}), we see that, for some $x_0\in\nu^\perp$,
\begin{equation}
  \label{third condition}
  \pp_{\nu^\perp}(W)=\cl(\DD_{d/2}^\nu(x_0))\,,\qquad\mbox{so that $W\subset\CC_{d/2}^\nu(x_0)$}\,.
\end{equation}
Condition \eqref{dai 1} implies that \eqref{asymptotics of F} holds with $u\equiv a$ for some $a\in[A,B]=\bigcap\{(\a,\b):W\subset\mathbf{S}^\nu_{(\a,\b)}\}$; in particular, since $(\pa F)\setminus W$ is a minimal surface and $W\subset\CC_{d/2}^\nu(x_0)$, by analytic continuation we find that
\begin{equation}
  \label{two conditions i}
  (\partial F)\setminus\CC_{d/2}^\nu(x_0) =\Pi\setminus\CC_{d/2}^\nu(x_0)\,,\qquad\Pi=\big\{x:x\cdot\nu=a\big\}\,.
\end{equation}
By \eqref{two conditions i}, we have that for $R>R'$,
\[
P(F;\CC_R^\nu\setminus W)-\om_n\,R^n=P(F;\CC_{d/2}^\nu(x_0)\setminus W)-\om_n\,(d/2)^n\,.
\]
Going back to \eqref{projection type inequality}, this implies $P(F;\CC_{d/2}^\nu(x_0)\setminus W)=0$. However, since $(\pa F)\setminus W$ is (distributionally) a minimal surface, $P(F;B_\rho(x)\setminus W)\ge\om_n\,\rho^n$ whenever $x\in (\pa F)\setminus W$ and $\rho<\dist(x,W)$, so that
$P(F;\CC_{d/2}^\nu(x_0)\setminus W)=0$ gives $((\pa F)\setminus W)\cap\CC_{d/2}^\nu(x_0)=\emptyset$. Hence, using also \eqref{two conditions i}, we find $(\partial F)\setminus W=\Pi\setminus\cl\big(B_{d/2}(x)\big)$ for some $x\in\Pi$, that is \eqref{characterization 2}.

\noindent {\bf Step eight:} We finally prove that $\Rr(W)=\om_n\,(d/2)^n$ if and only if there are a hyperplane $\Pi$ and a point $x\in\Pi$ such that
\begin{eqnarray}
  \label{condo 1}
  &&\Pi\cap\pa B_{d/2}(x)\subset W\,,
  \\
  \label{condo 2}
  &&\mbox{$\Om\setminus(\Pi\setminus B_{d/2}(x))$ has two unbounded connected components}\,.
\end{eqnarray}
We first prove that the two conditions are sufficient. Let $\nu$ be a unit normal to $\Pi$ and let $\Pi^+$ and $\Pi^-$ be the two open half-spaces bounded by $\Pi$. The condition $\Pi\cup\pa B_{d/2}(x)\subset W$ implies $W\subset\CC_{d/2}^\nu(x)$, and thus
\[
\Om\setminus\cl\big[ \CC^\nu_{d/2,(-d,d)}(x)\big]=(\Pi^+\cup\Pi^-)\setminus\cl\big[\CC^\nu_{d/2,(-d,d)}(x)\big]\,.
\]
In particular, $\Omega \setminus(\Pi\setminus B_{d/2}(x))$ has a connected component $F$ which contains
\[
\Pi^+\setminus\cl\big[\CC^\nu_{d/2,(-d,d)}(x)\big]\,;
\]
and since $\Om\setminus(\Pi\setminus B_{d/2}(x))$ contains exactly two unbounded connected components, it cannot be that $F$ contains also $\Pi^-\setminus\cl[\CC^\nu_{d/2,(-d,d)}(x)]$, therefore
\begin{equation}
  \label{pi plus pi minus}
\Pi^+\setminus\cl\big[\CC^\nu_{d/2,(-d,d)}(x)\big]\subset F\,,\qquad \Pi^-\setminus\cl\big[\CC^\nu_{d/2,(-d,d)}(x)\big]\subset\R^{n+1}\setminus\cl(F)\,.
\end{equation}
As a consequence $\pa F$ is contained in the slab $\{y:|(y-x)\cdot\nu|<d\}$, and is such that $\pp_{\nu^\perp}(\pa F)=\nu^\perp$, that is, $(F,\nu)\in\F$. Moreover, \eqref{pi plus pi minus} implies
\[
\Pi\setminus \cl(B_{d/2}(x))\subset\Om\cap\pa F\,,
\]
while the fact that $F$ is a connected component of $\Omega \setminus (\Pi\setminus B_{d/2}(x))$ implies $\Om\cap\pa F\subset \Pi\setminus \cl(B_{d/2}(x))$. In conclusion, $\Om\cap\pa F=\Pi\setminus\cl(B_{d/2}(x))$, hence
\begin{equation}
\om_n\,(d/2)^n=
\lim_{r\to \infty}\omega_n r^n - P(F; \textbf{C}_r^{\nu}\setminus W) \leq \Rr(W) \leq \om_n\,(d/2)^n\,,
\end{equation}
and $\Rr(W)=\om_n\,(d/2)^n$, as claimed. We prove that the two conditions are necessary. Let $(F,\nu)\in{\rm Max}[\Rr(W)]$. As proved in step seven, there is a hyperplane $\Pi$ and $x\in\Pi$ such that $\Om\cap\pa F=\Pi\setminus \cl(B_{d/2}(x))$. If $z\in\Pi\cap\pa B_{d/2}(x)$ but $z\in\Om$, then there is $\rho>0$ such that $B_\rho(z)\subset\Om$, and since $\pa F$ is a minimal surface in $\Om$, we would obtain that $\Pi\cap B_\rho(z)\subset\Om\cap\pa F$, against $\Om\cap\pa F=\Pi\setminus\cl(B_{d/2}(x))$. So it must be $\Pi\cap\pa B_{d/2}(x)\subset W$, and the necessity of \eqref{condo 1} is proved. To prove the necessity of \eqref{condo 2}, we notice that since $\Pi^+\setminus\cl[\CC^\nu_{d/2,(-d,d)}(x)]$ and $\Pi^-\setminus\cl[\CC^\nu_{d/2,(-d,d)}(x)]$ are both open, connected, and unbounded subsets of $\Omega \setminus (\Pi\setminus B_{d/2}(x))$, and since the complement in $\Omega \setminus (\Pi\setminus B_{d/2}(x))$ of their union is bounded, it must be that $\Omega \setminus (\Pi\setminus B_{d/2}(x))$ has {\it at most} two unbounded connected components: therefore we just need to exclude that {\it it has only one}. Assuming by contradiction that this is the case, we could then connect any point $x^+\in\Pi^+\setminus\cl[\CC^\nu_{d/2,(-d,d)}(x)]$  to any point $x^-\in\Pi^-\setminus\cl[\CC^\nu_{d/2,(-d,d)}(x)]$ with a continuous path $\g$ entirely contained in $\Omega \setminus (\Pi\setminus  B_{d/2}(x))$. Now, recalling that $\Om\cap\pa F=\Pi\setminus\cl(B_{d/2}(x))$, we can pick $x_0\in\Pi\setminus\cl(B_{d/2}(x))$ and $r>0$ so that
\begin{equation}
  \label{separating F}
B_r(x_0)\cap\Pi^+\subset F\,,\qquad B_r(x_0)\cap\Pi^-\subset\R^{n+1}\setminus\cl(F)\,,
\end{equation}
and $B_r(x_0)\cap\cl[\CC^\nu_{d/2,(-d,d)}(x)]=\emptyset$. We can then pick $x^+\in B_r(x_0)\cap\Pi^+$, $x^-\in B_r(x_0)\cap\Pi^-$, and then connect them by a path $\g$ entirely contained in $\Omega \setminus (\Pi\setminus  B_{d/2}(x))$. By \eqref{separating F}, $\g$ must intersect $\pa F$, and since $\g$ is contained in $\Om$, we see that $\g$ must intersect $\Om\cap\pa F=\Pi\setminus\cl(B_{d/2}(x))$, which of course contradicts the containment of $\g$ in $\Omega \setminus (\Pi\setminus  B_{d/2}(x))$. We have thus proved that $\Omega \setminus (\Pi\setminus  B_{d/2}(x))$ has exactly two unbounded connected components.
\end{proof}

\section{Resolution theorem for exterior isoperimetric sets}\label{section resolution for exterior} The notation set in \eqref{cylinders and slabs} is in use. Given $v_j\to\infty$, we set $\l_j=v_j^{1/(n+1)}$.

\begin{proof}[Proof of Theorem \ref{thm main psi}] We recall that, throughout the proof, $\Rr(W)>0$. Theorem \ref{thm main psi}-(i) and the estimate for $|v^{-1/(n+1)}\,|x|-\om_{n+1}^{-1/(n+1)}|$ in Theorem \ref{thm main psi}-(iv), have already been proved in Theorem \ref{thm existence and uniform min}-(ii, iii).

\noindent {\bf Step one:} We prove that
  \begin{equation}\label{energy upper bound}
  \varlimsup_{v \to \infty} \psi_W(v)- P(B^{(v)})\leq -\Rr(W)\,.
  \end{equation}
  To this end, let $(F,\nu)\in{\rm Max}[\Rr(W)]$, so that by \eqref{main residue graphicality of F} and \eqref{asymptotics of F}, we have
  \begin{equation}
    \label{order of F0}
      F\setminus\CC_{R_2}^\nu=\big\{x+t\,\nu:x\in\nu^\perp\,,|x|>R_2\,,t<f(x)\big\}\,,
  \end{equation}
  for a function $f\in C^1(\nu^\perp)$ satisfying
  \begin{eqnarray}
    \label{order of F}
    &&\big|f(x)-\big(a+b\,|x|^{2-n}+(c\cdot x)\,|x|^{-n}\big)\big|\le C_0\,|x|^{-n}\,,
    \\
    \nonumber
    &&\max\big\{|x|^{n-1}\,|\nabla f(x)|,|x|^{n}\,|\nabla^2 f(x)|\big\}\le C_0\,,\qquad\forall x\in\nu^\perp\,,|x|>R_2\,,
  \end{eqnarray}
  and for some $a,b\in\R$ and $c\in\nu^\perp$ such that $\max\{|a|,|b|,|c|\}\le C(W)<\infty$ (moreover, we can take $b=0$, $c=0$ and $C_0=0$ if $n=1$). We are going to construct competitors for $\psi_W(v)$ with $v$ large by gluing a large sphere $S$ to $\pa F$ along $\pa\CC_r^\nu$ for $r>R_2$. This operation comes at the price of an area error located on the cylinder $\pa\CC_r^\nu$. We can make this error negligible thanks to the fact that \eqref{order of F} determines the distance (inside of $\pa\CC_r^\nu$) of $\pa F$ from a hyperplane (namely, $\pa G_r$ for the half-space $G_r$ defined below) up to ${\rm o}(r^{1-n})$ as $r\to\infty$. Thus, the asymptotic expansion \eqref{asymptotics of F} is just as precise as needed in order to perform this construction, i.e. our construction would not be possible with a less precise information.

We now discuss the construction in detail. Given $r>R_2$, we consider the half-space $G_r\subset\R^{n+1}$ defined by the condition that
\begin{equation}
  \label{def of Gr}
 G_r \cap \partial\CC_r^\nu =\big\{x+t\,\nu:x\in\nu^\perp\,,|x|=r\,, t< a + b\,r^{2-n}+(c\cdot x)\,r^{-n}\big\}\,,
\end{equation}
so that $G_r$ is the ``best half-space approximation'' of $F$ on $\pa\CC_r^\nu$ according to \eqref{order of F}. Denoting by $\hd(X,Y)$ the Hausdorff distance between $X,Y\subset\R^{n+1}$, for every $r>R_2$ and $v>0$ we can define $x_{r,v}\in\R^{n+1}$ in such a way that $v\mapsto x_{r,v}$ is continuous and
\begin{eqnarray}
\label{hausdorff convergence to halfspace}
  \lim_{v\to\infty}\hd(B^{(v)}(x_{r,v})\cap K,G_r\cap K)=0\qquad\forall\,K\cc\R^{n+1}\,.
\end{eqnarray}
Thus, the balls $B^{(v)}(x_{r,v})$ have volume $v$ and are locally converging in Hausdorff distance, as $v\to\infty$, to the optimal half-space $G_r$. Finally, we notice that by \eqref{order of F} we can find $\a<\b$ such that
\begin{align}\label{containment in finite cylinder}
\big((\partial F) \cup (\partial G_r) \cup (G_r \Delta F)\big) \cap \CC_r^\nu \,\,\subset\,\, \CC_{r,(\a+1,\b-1)}^\nu\,,
\end{align}
and then define $F_{r,v}$ by setting
\begin{equation}\label{the competitors}
  F_{r,v}=\big(F\cap\CC_{r,(\a,\b)}^\nu\big)\cup\big(B^{(v)}(x_{r,v})\setminus\cl\big[\CC_{r,(\a,\b)}^\nu\big]\big)\,,
\end{equation}
see
\begin{figure}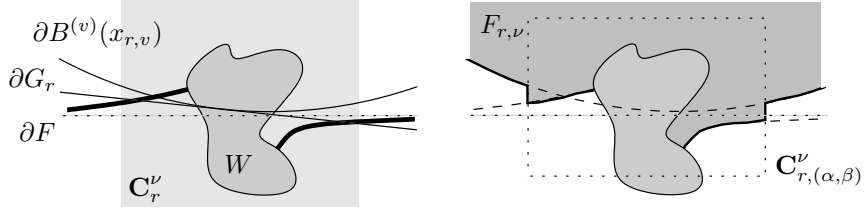\caption{\small{The competitors $F_{r,v}$ constructed in \eqref{the competitors}. A maximizer $F$ in the isoperimetric residue $\Rr(W)$ is joined to a ball of volume $v$, whose center $x_{r,v}$ is determined by looking at best hyperplane $\pa G_r$ approximating $\pa F$ on the ``lateral'' cylinder $\pa\CC_r^\nu$. To ensure the area error made in joining this large sphere to $\pa F$ is negligible, the distance between $\pa F$ and the sphere inside $\pa\CC_r^\nu$ must be ${\rm o}(r^{1-n})$ as $r\to\infty$. The asymptotic expansion \eqref{order of F} gives a hyperplane $\pa G_r$ which is close to $\pa F$ up to ${\rm O}(r^{-n})$, and is thus just as precise as needed to perform the construction.}}\label{fig upperbound}\end{figure}
Figure \ref{fig upperbound}. We claim that, by using $F_{r,v}$ as comparisons for $\psi_W(|F_{r,v}|)$, and then sending first $v\to\infty$ and then $r\to\infty$, one obtains \eqref{energy upper bound}. We first notice that by \eqref{hausdorff convergence to halfspace} and \eqref{containment in finite cylinder} (see, e.g. \cite[Theorem 16.16]{maggiBOOK}), we have
\begin{eqnarray}\nonumber
P(F_{r,v};\Om)&=&P(F;\CC_{r,(\a,\b)}^\nu\setminus W)+P\big(B^{(v)}(x_{r,v});\R^{n+1}\setminus\cl\big[\CC_{r,(\a,\b)}^\nu\big]\big)
\\\label{perimeter Frv}
&&+\H^n\big((F\Delta B^{(v)}(x_{r,v}))\cap\pa_\ell\CC_{r,(\a,\b)}^\nu\big)\,,
\end{eqnarray}
where the last term is the ``gluing error'' generated by the mismatch between the boundaries of $\pa F$ and $\pa B^{(v)}(x_{r,v})$ along $\pa_\ell\CC_{r,(\a,\b)}^\nu$. Now, thanks to \eqref{order of F} we have
$\hd(G_r\cap\pa\CC_r^\nu,F\cap\pa\CC_r^\nu)\le C_0\,r^{-n}$, so that
\begin{equation}
  \label{small error on lateral boundary}
  \H^n\big((F\Delta G_r)\cap\pa\CC_r^\nu\big)\le n\,\om_n\,r^{n-1}\,  \hd(G_r\cap\pa\CC_r^\nu,F\cap\pa\CC_r^\nu)\le C(n,W)/r\,.
\end{equation}
At the same time, by \eqref{hausdorff convergence to halfspace},
\begin{equation}
\label{L1 convergence on cylinder boundaries}
\lim_{v\to\infty}\H^n\big((G_r\Delta B^{(v)}(x_{r,v}))\cap\pa_\ell\CC_{r,(\a,\b)}^\nu\big)=0\,,
\end{equation}
and thus we have the following estimate for the gluing error,
\begin{equation}
\label{gin 1}
  \varlimsup_{v\to\infty}\H^n\big((F\Delta B^{(v)}(x_{r,v}))\cap\pa_\ell\CC_{r,(\a,\b)}^\nu\big)\le\frac{C(n,W)}r\,,\qquad\forall r>R_2\,.
\end{equation}
Again by \eqref{hausdorff convergence to halfspace}, we find
\begin{eqnarray}
\label{ball perimeter converges to halfspace perimeter}
&&\!\!\!\!\!\!\!\!\!\!\!\!\!\!\!\!\!\!\!\!\!\!\lim_{v\to\infty}P\big(B^{(v)}(x_{r,v});\CC_{r,(\a,\b)}^\nu\big)=P\big(G_r;\CC_{r,(\a,\b)}^\nu\big)
\\
\label{Gr perimeter}
&&\!\!\!\!\!\!\!\!\!\!\!\!\!\!\!\!\!\!\!\!\!\!1\le(\om_n\,r^n)^{-1}\,P\big(G_r;\CC_{r,(\a,\b)}^\nu\big)=\fint_{\DD_r^\nu}\sqrt{1+(c/r^n)^2}\le 1+C_0\,r^{-2\,n}\,,
\end{eqnarray}
so that, by \eqref{ball perimeter converges to halfspace perimeter} and by the lower bound in \eqref{Gr perimeter}, for every $r>R_2$,
\begin{equation}\label{gin 2}
\varlimsup_{v\to\infty}P\big(B^{(v)}(x_{r,v});\R^{n+1}\setminus\cl\big[\CC_{r,(\a,\b)}^\nu\big]\big)-P(B^{(v)})
\le -\om_n\,r^n\,.
\end{equation}
Combining \eqref{gin 1} and \eqref{gin 2} with \eqref{perimeter Frv} and the fact that $\CC_{r,(\a,\b)}^\nu\cap\pa F=\CC_r^\nu\cap\pa F$ (see \eqref{containment in finite cylinder}), we find that for every $r>R_2$,
\begin{eqnarray}\nonumber
&&\varlimsup_{v\to\infty}P(F_{r,v};\Om)-P(B^{(v)})\le P(F;\CC_r^\nu\setminus W)-\om_n\,r^n+ C(n,W)/r
\\
\label{perimeter Frv 2}
&&\le-{\rm res}_W(F,\nu)+C(n,W)/r=-\Rr(W)+C(n,W)/r\,.
\end{eqnarray}
where \eqref{perimeter is decreasing} has been used. Now, combining the elementary estimates
\begin{equation}
  \label{volum of Frv}
  \max\big\{\big||F_{r,v}|-v\big|\,,v^{-1/(n+1)}\,|P(B^{(v)})-P(B^{(|F_{r,v}|)})|\big\}\le C(n)\,r^{n+1}
\end{equation}
with \eqref{perimeter Frv 2}, we see that
\begin{eqnarray}\label{perimeter Frv 3}
\varlimsup_{v\to\infty}\psi_W(|F_{r,v}|)-P(B^{(|F_{r,v}|)})\le-\Rr(W)+ C(n,W)/r\,,\,\,\forall r>R_2\,.
\end{eqnarray}
Again by \eqref{volum of Frv} and since $v\mapsto|F_{r,v}|$ is a continuous function, we see that
$\varlimsup_{v\to\infty}\psi_W(|F_{r,v}|)-P(B^{(|F_{r,v}|)})=\varlimsup_{v\to\infty}\psi_W(v)-P(B^{(v)})$. This last identity combined with \eqref{perimeter Frv 3} implies \eqref{energy upper bound} in the limit $r\to\infty$.

\noindent {\bf Step two:} Now let $E_j\in{\rm Min}[\psi_W(v_j)]$ for $v_j \to \infty$. By \eqref{uniform lambda minimality} and a standard argument (see, e.g. \cite[Theorem 21.14]{maggiBOOK}), there is a local perimeter minimizer with free boundary $F$ in $\Om$ such that, up to extracting subsequences,
\begin{eqnarray}\nonumber
  &&\mbox{$E_j\to F$ in $L^1_{{\rm loc}}(\R^{n+1})$, $\H^n\llcorner\pa E_j\weak\H^n\llcorner\pa F$ as Radon measures in $\Om$}\,,
  \\
  &&\hd(K\cap\pa E_j;K\cap\pa F)\to0\qquad\mbox{for every $K\cc\Om$}\,.\label{convergence of Ej to F last proof}
\end{eqnarray}
Notice that it is not immediate to conclude from $E_j\in{\rm Min}[\psi_W(v_j)]$ that (for some $\nu\in\SS^n$) $(F,\nu)\in{\rm Max}[\Rr(W)]$ (or even that $(\nu,F)\in\F$), nor that $P(E_j;\Om)-P(B^{(v_j)})$ is asymptotically bounded from below by $-{\rm res}_W(F,\nu)$. In this step we prove some preliminary properties of $F$, and in particular we exploit the blowdown result for exterior minimal surfaces contained in Theorem \ref{theorem mesoscale criterion}-(ii) to prove that $F$ satisfies \eqref{order of F0} and \eqref{order of F} (see statement (c) below). Then, in step three, we use the decay rates \eqref{order of F} to show that $E_j$ can be ``glued'' to $F$, similarly to the construction of step one, and then derive from the corresponding energy estimates the lower bound matching \eqref{energy upper bound} and the optimality of $F$ in $\Rr(W)$.

\noindent {\bf (a)} {\it $\Om\cap\pa F\cap\pa B_\rho\ne\emptyset$ for every $\rho$ such that $W\cc B_\rho$}: If not there would be $\e>0$ such that $W\cc B_{\rho-\e}$ and $\Om\cap\pa F\cap A_{\rho-\e}^{\rho+\e}=\emptyset$ (recall that $A_r^s=\{x:s>|x|>r\}$). By \eqref{convergence of Ej to F last proof} and the constant mean curvature condition satisfied by $\Om\cap\pa E_j$, we would then find that each $E_j$ (with $j$ large enough) has a connected component of the form $B^{(w_j)}(x_j)$, with $B^{(w_j)}(x_j)\cc \R^{n+1}\setminus B_{\rho+\e}$ and $w_j\ge v_j-C(n)\,(\rho+\e)^{n+1}$. In particular, against $\Rr(W)>0$,
\[
\psi_W(v_j)=P(E_j;\Om)\ge P(B^{(v_j-C\,(\rho+\e)^{n+1})})\ge P(B^{(v_j)})-C\l_j^{-1}(\rho+\e)^{n+1}\,.
\]

\noindent {\bf (b)} {\it Sharp area bound}: We combine the upper energy bound \eqref{energy upper bound} with the perimeter inequality for spherical symmetrization, to prove
\begin{equation}\label{growth bound for F}
   P(F;\Omega \cap B_r) \leq \omega_n r^n - \Rr(W)\,,\qquad\mbox{for every $r$ s.t. $W\cc B_r$}\,.
\end{equation}
(Notice that \eqref{growth bound for F} does not immediately imply the bound for $P(F;\Omega\cap\CC_r^\nu)$ which would be needed to compare $\Rr(W)$ and ${\rm res}_W(F,\nu)$.) To prove \eqref{growth bound for F} we argue by contradiction, and consider the existence of $\de>0$ and $r$ with $W\cc B_r$ such that
$P(F;\Om\cap\ B_r)\ge \om_n\,r^n-\Rr(W)+\de$. In particular, for $j$ large enough, we would then have
\begin{equation}\label{bad one}
P(E_j;\Om\cap B_r)\ge\omega_nr^n - \Rr(W)+\de\,.
\end{equation}
Again for $j$ large, it must be $\H^n(\pa E_j\cap\pa B_r)=0$: indeed, by \eqref{uniform lambda minimality}, $\Om\cap\pa E_j$ has mean curvature of order ${\rm O}(\l_{j}^{-1})$, while of course $\pa B_r$ has constant mean curvature equal to $n/r$. Thanks to $\H^n(\pa E_j\cap\pa B_r)=0$,
\begin{equation}
  \label{bad one 1}
  P(E_j;\Om)=P(E_j;\Om\cap B_r)+P\big(E_j;\R^{n+1}\setminus\cl(B_r)\big)\,.
\end{equation}
If $E_j^s$ denotes the spherical symmetral of $E_j$ such that $E_j^s\cap\pa B_\rho$ is a spherical cap in $\pa B_\rho$, centered at $\rho\,e_{n+1}$, with area equal to $\H^n(E_j\cap\pa B_\rho)$, then we have the perimeter inequality
\begin{equation}
  \label{bad one 2}
  P\big(E_j;\R^{n+1}\setminus\cl(B_r)\big)\ge P\big(E_j^s;\R^{n+1}\setminus\cl(B_r)\big)\,;
\end{equation}
see \cite{cagnettiperuginistoger}. Now, we can find a half-space $J$ orthogonal to $e_{n+1}$ and such that $\H^n(J\cap\pa B_r)=\H^n(E_j\cap\pa B_r)$. In this way, using that $|E_j^s\setminus B_r|=|E_j\setminus B_r|$ (by Fubini's theorem in spherical coordinates), and that $\H^n(B_r\cap\pa J)\le\om_n\,r^n$ (by the fact that $\pa J$ is a hyperplane), we find
\begin{eqnarray*}
P\big(E_j^s;\R^{n+1}\setminus\cl(B_r)\big)&=&P\big((E_j^s\setminus\cl(B_r))\cup (J\cap B_r)\big)-\H^n(B_r\cap\pa J)
\\
&\ge&P\big(B^{(|E_j|-|E_j\cap B_r|+|J\cap B_r|)}\big)-\om_n\,r^n
\\
&\ge&P(B^{(v_j)})-C(n)\,r^{n+1}\,\l_{j}^{-1}-\om_n\,r^n
\end{eqnarray*}
which, with \eqref{bad one}, \eqref{bad one 1} and \eqref{bad one 2}, finally gives
$P(E_j;\Om)-P(B^{(v_j)})> -\Rr(W)+\de-C(n)\,r^{n+1}\,\l_{j}^{-1}$
for $j$ large, against \eqref{energy upper bound}.

\noindent {\bf (c)} {\it Asymptotic behavior of $\pa F$}: We prove that there are $\nu\in\SS^n$, $f\in C^\infty(\nu^\perp)$, $a,b\in\R$, $c\in\nu^\perp$, $R'>\sup\{\rho:W\subset\CC_\rho^\nu\}$ and $C$ positive, with
\begin{eqnarray}
  \label{main residue graphicality of F end}
  &&\partial F \setminus \CC^\nu_{R'}=\big\{x+f(x)\,\nu:x\in\nu^\perp\,,|x|>R'\big\}\,,
\\\nonumber
  &&f(x)=a\,,\hspace{6.6cm} (n=1)
  \\\label{asymptotics of F end}
  &&\big|f(x)-\big(a+b\,|x|^{2-n}+(c\cdot x)\,|x|^{-n}\big)\big|\le C\,|x|^{-n}\,,\,\,\, (n\ge 2)\,,
  \\\nonumber
  &&
  \max\big\{|x|^{n-1}\,|\nabla f(x)|,|x|^n\,|\nabla^2f(x)|\big\}\le C_0\,,\qquad\forall x\in\nu^\perp\,,|x|>R'\,.
\end{eqnarray}
To this end, by a standard argument exploiting the local perimeter minimality of $F$ in $\Om$, given $r_j\to\infty$, then, up to extracting subsequences, $F/r_j\toloc J$ in $L^1_{\rm loc}(\R^{n+1})$, where $J$ is a perimeter minimizer in $\R^{n+1}\setminus\{0\}$, $0\in\pa J$ (thanks to property (a)), $J$ is a cone with vertex at $0$
(thanks to Theorem \ref{theorem 7.17AA exterior lambda} and, in particular to \eqref{conelimit}), and $P(J;B_1)\le\om_n$ (by \eqref{growth bound for F}). {\bf If $n\ge 2$}, then $\pa J$ has vanishing distributional mean curvature in $\R^{n+1}$ (as points are removable singularities for the mean curvature operator when $n\ge 2$), thus $P(J;B_1)\ge\om_n$ by upper semicontinuity of area densities, and, finally, by $P(J;B_1)=\om_n$ and Allard's regularity theorem, $J$ is a half-space. {\bf If $n=1$}, then $\pa J$ is the union of two half-lines $\ell_1$ and $\ell_2$ meeting at $\{0\}$. If $\ell_1$ and $\ell_2$ are not opposite (i.e., if $J$ is not a half-space), then we can find a half-space $J^*$ such that $(J\cap J^*)\Delta J\cc B\cc \R^2\setminus\{0\}$ for some ball $B$, and $P(J\cap J^*;B)<P(J;B)$, thus violating the fact that $J$ is a perimeter minimizer in $\R^{n+1}\setminus\{0\}$.

If $n=1$ it is immediate from the above information that, for some $R'>0$, $F\setminus B_{R'}=J\setminus B_{R'}$; this proves \eqref{main residue graphicality of F end} and \eqref{asymptotics of F end} in the case $n=1$. To prove \eqref{main residue graphicality of F end} and \eqref{asymptotics of F end} when $n\ge 2$, we let $M_0$ and $\e_0$ be as in Theorem \ref{theorem mesoscale criterion}-(ii) with $(n,\Gamma=2\,n\,\om_n,\s=1)$. Since $J$ is a half-space, by using Remark \ref{remark improved convergence} and $F/r_j\toloc J$ on the annulus $A_{1/2}^{2\,L}$, for some $L>\max\{M_0,64\}$ to be chosen later on depending also on $\e_0$, we find that
\begin{equation}\label{between one half and L}
(\pa F)\cap A^{4\,L\,r_j}_{r_j/2}
=\big\{x+r_j\,f_j\big(x/r_j\big)\,\nu:x\in\nu^\perp\big\}\cap A^{4\,L\,r_j}_{r_j/2}\,,\qquad \nu^\perp=\pa J\,,
\end{equation}
for  $f_j\in C^1(\nu^\perp)$ with $\|f_j\|_{C^1(\nu^\perp)}\to 0$. By \eqref{between one half and L},
$V_j=\var\big((\pa F)\setminus B_{r_j},1)\in\V_n(0,r_j,\infty)$, with (for ${\rm o}(1)\to 0$ as $j\to\infty$)
\begin{eqnarray}\label{FF1}
  &&r_j^{-n}\,\int x\cdot\nu^{\rm co}_{V_j}\,d{\rm bd}_{V_j}=-n\,\om_n+{\rm o}(1)
  \\\label{FF2}
  &&r_j^{1-n}\|{\rm bd}_{V_j}\|(\pa B_{r_j})=n\,\om_n+{\rm o}(1)\,,
  \\\label{FF3}
  &&\sup_{r\in(r_j,3\,L\,r_j)}\big|(r^n-r_j^n)^{-1}\,\|V_j\|(B_r\setminus B_{r_j})-\om_n\big|={\rm o}(1)\,.
\end{eqnarray}
By our choice of $\Gamma$, by \eqref{growth bound for F} and \eqref{FF1} we see that, for $j$ large, we have
\begin{equation}
  \label{blowdown hp check 1}
  \|{\rm bd}_{V_j}\|(\pa B_{r_j})\le \Gamma\,r_j^{n-1}\,,\qquad \|V_j\|(B_\rho\setminus B_{r_j})\le\Gamma\rho^n\,,\,\,\forall\rho>r_j\,.
\end{equation}
Moreover, we claim that setting
\[
s_j=2\,L\,r_j
\]
(so that, in particular, $s_j>\max\{M_0,64\}\,r_j$), then
\begin{eqnarray}
\label{blowdown hp check 2}
  |\de_{V_j,r_j,0}(s_j/8)|\le\e_0\,,\qquad \inf_{r>s_j/8}\de_{V_j,r_j,0}(r)\ge-\e_0\,,
\end{eqnarray}
provided $j$ and $L$ are taken large enough depending on $\e_0$. To check the first inequality in \eqref{blowdown hp check 2} we notice that, by \eqref{FF1} and \eqref{FF3},
\begin{eqnarray*}
   \de_{V_j,r_j,0}(s_j/8)\!\!&=&\!\!\om_n-\frac{\|V_j\|(B_{s_j/8}\setminus B_{r_j})}{(s_j/8)^n}+\frac1{n\,(s_j/8)^n}\,\int x\cdot\nu^{\rm co}_{V_j}\,d\,{\rm bd}_{V_j}
   \\
   &=&\!\!\om_n-\big(\om_n+{\rm o}(1)\big)\,\frac{(s_j/8)^n-r_j^n}{(s_j/8)^n}-\frac{\om_n\,r_j^n}{(s_j/8)^n}\,\big(1+{\rm o}(1)\big)
   \\
   &=&\!\!{\rm o}(1)\,(1+(r_j/s_j)^n)={\rm o}(1)\,,
\end{eqnarray*}
so that $|\de_{V_j,r_j,0}(s_j/8)|\le\e_0$ as soon as $j$ is large with respect to $\e_0$. Similarly, if $r>s_j/8=(L\,r_j)/4$, then by \eqref{FF1}, \eqref{FF3}, \eqref{growth bound for F}, and $r_j/r\le 4/L$,
\begin{eqnarray*}
  &&\!\!\!\!\!\!\!\de_{V_j,r_j,0}(r)=\om_n-\frac{\|V_j\|(B_r\setminus B_{2\,r_j})}{r^n}-\frac{\|V_j\|(B_{2\,r_j}\setminus B_{r_j})}{r^n}
  -\frac{\om_n\,r_j^n}{r^n}\,\big(1+{\rm o}(1)\big)
  \\
  &&\ge\om_n-\frac{\om_n\,r^n-\Rr(W)}{r^n}-\big(\om_n+{\rm o}(1)\big)\,\frac{(2\,r_j)^n-r_j^n}{r^n}-\frac{\om_n\,r_j^n}{r^n}\,\big(1+{\rm o}(1)\big)
  \\
  &&\ge r^{-n}\, \Rr(W) -2\,(4/L)^n\,\big(\om_n+{\rm o}(1)\big)-(4/L)^n\,{\rm o}(1)\ge-3\,(4/L)^n\,\om_n\,,
\end{eqnarray*}
provided $j$ is large; hence the second inequality in \eqref{blowdown hp check 2} holds if $L$ is large in terms of $\e_0$.
By \eqref{blowdown hp check 1} and \eqref{blowdown hp check 2}, Theorem \ref{theorem mesoscale criterion}-(ii) can be applied to $(V,R,\Lambda,s)=(V_j,r_j,0,s_j)$ with $j$ large. As a consequence, passing from spherical graphs to cylindrical graphs with the aid of Lemma \ref{lemma D1}, we find that, for some large $j$,
\begin{equation}\label{between one half and L fine}
(\pa F)\setminus B_{s_j/16}
=\big\{x+f(x)\,\nu:x\in\nu^\perp\big\}\setminus B_{s_j/16}\,,
\end{equation}
where $f:\nu^\perp\to\R$ is a smooth function which solves the minimal surfaces equation on $\nu^\perp\setminus B_{s_j/16}$. Since $\pa F$ admits at least one sequential blowdown limit hyperplane (namely, $\nu^\perp=\pa J$), by a theorem of Simon \cite[Theorem 2]{SimonAIHP} we find that $\nabla f$ has a limit as $|x|\to\infty$; in particular, $|\nabla f|$ is bounded. Moreover, by \eqref{between one half and L fine} (or by the fact that $F$ is a local perimeter minimizer in $\Om$), $\pa F$ is a stable minimal surface in $\R^{n+1}\setminus B_{s_j/16}$, which, thanks to \eqref{growth bound for F}, satisfies an area growth bound like \eqref{schoen 2}. We can thus apply Proposition \ref{prop schoen} to deduce the validity of \eqref{asymptotics of F end} when $n\ge 3$, and of $|f(x)-[a+b\,\log\,|x|+(c\cdot x)\,|x|^{-2}]|\le C\,|x|^{-2}$ for all $|x|>R'$ when $n=2$ (with $R'>s_j$). Recalling that $F$ is a local perimeter minimizer with free boundary in $\Om$ (that is, $P(F;\Om\cap B)\le P(F';\Om\cap B)$ whenever $F\Delta F'\cc B\cc\R^3$) it must be that $b=0$, as it can be seen by comparing $F$ with the set $F'$ obtained by changing $F$ inside $\CC_r^\nu$ ($r>>R'$) with the half-space $G_r$ bounded by the plane $\{x+t\,\nu:x\in\nu^\perp,t=a+b\,\log(r)+c\cdot x/r^2\}$ and such that $\H^2((F\Delta G_r)\cap\pa\CC_r^\nu)\le C/r^2$ (we omit the details of this standard comparison argument). Having shown that $b=0$, the proof is complete.

\noindent {\bf (d)} {\it $F\cup W$ defines an element of $\F$}: With $R>R'$ as in \eqref{main residue graphicality of F end} and \eqref{asymptotics of F end}, $V_R=\var((\pa F)\cap(B_R\setminus W))$ is a stationary varifold in $\R^{n+1}\setminus K_R$ for
$K_R=W\cup\big\{x+f(x)\,\nu:x\in\nu^\perp\,,|x|=R\}$, and has bounded support. By the convex hull property \cite[Theorem 19.2]{SimonLN}, we deduce that, for every $R>R'$, $\spt V_R$ is contained in the convex hull of $K_R$, for every $R>R'$. Taking into account that $f(x)\to a$ as $|x|\to\infty$ we conclude that $\Om\cap\pa F$ is contained in the smallest slab $\mathbf{S}_{[\a,\b]}^\nu$ containing both $W$ and $\{x:x\cdot\nu=a\}$. Now set
$F'=F\cup W$. Clearly $F'$ is a set of locally finite perimeter in $\Om$ (since $P(F';\Om')=P(F;\Om')$ for every $\Om'\cc\Om$). Second, $\pa F'$ is contained in $\mathbf{S}_{[\a,\b]}^\nu$ (since $\pa F'\subset [(\pa F)\cap\Om]\cup W$). Third, by \eqref{main residue graphicality of F end} and \eqref{asymptotics of F end},
\begin{eqnarray}
\label{coo1}
&&\big\{x+t\,\nu:x\in\nu^\perp\,,|x|>R'\,,t<\a\big\}\subset F'\,,
\\
\label{coo2}
&&\big\{x+t\,\nu:x\in\nu^\perp\,,|x|>R'\,,t>\b\big\}\subset \R^{n+1}\setminus F'\,,
\\
\label{coo3}
&&\big\{x+t\,\nu:x\in\nu^\perp\,,|x|<R'\,,t\in\R\setminus[\a,\b]\big\}\cap(\pa F')=\emptyset\,.
\end{eqnarray}
By combining \eqref{coo1} and \eqref{coo3} we see that $\{x+t\,\nu:x\in\nu^\perp\,,t<\a\}\subset F'$, and by combining \eqref{coo2} and \eqref{coo3} we see that $\{x+t\,\nu:x\in\nu^\perp\,,t>\b\}\subset \R^{n+1}\setminus F'$: in particular, $\pp_{\nu^\perp}(\pa F')=\nu^\perp$, and thus $(F',\nu)\in\F$.

\noindent {\bf Step three:} We prove that
\begin{equation}\label{lower bound equation}
\varliminf_{v\to \infty} \psi_W(v) - P(B^{(v)}) \geq -\Rr(W)\,.
\end{equation}
For  $v_j\to\infty$ achieving the liminf in \eqref{lower bound equation}, let $E_j\in{\rm Min}[\psi_W(v_j)]$ and let $F$ be a (sub-sequential) limit of $E_j$, so that properties (a), (b), (c) and (d) in step two hold for $F$. In particular, properties \eqref{main residue graphicality of F end} and \eqref{asymptotics of F end} from (c) are entirely analogous to properties \eqref{order of F0} and \eqref{order of F} exploited in step one: therefore, the family of half-spaces $\{G_r\}_{r>R'}$ defined by \eqref{def of Gr} is such that
\begin{eqnarray}
  \label{b1}
  \big((\partial F) \cup (\partial G_r) \cup (G_r \Delta F)\big) \cap \CC_r^\nu \,\,\subset\,\, \CC_{r,(\a+1,\b-1)}^\nu\,,
  \\
  \label{b2}
  \H^n\big((F\Delta G_r)\cap\pa\CC_r^\nu\big)\le r^{-1}\,C(n,W)\,,
  \\
  \label{b3}
  \big|P\big(G_r;\CC_{r,(\a,\b)}^\nu\big)-\om_n\,r^n\big|\le r^{-n}\,C(n,W)\,,
\end{eqnarray}
(compare with \eqref{containment in finite cylinder}, \eqref{small error on lateral boundary}, and \eqref{Gr perimeter} in step one). By \eqref{b3} we find
\begin{equation}
  \label{coo4}
  -{\rm res}_W(F',\nu)
  =\lim_{r\to\infty}P(F;\CC_r^\nu\setminus W)-P(G_r;\CC_{r,(\a,\b)}^\nu)\,.
\end{equation}
In order to relate the residue of $(F',\nu)$ to $\psi_W(v_j)-P(B^{(v_j)})$ we consider the sets
$Z_j=(G_r\cap\CC_{r,(\a,b)}^\nu)\cup(E_j\setminus\CC_{r,(\a,\b)}^\nu)$, which, by isoperimetry, satisfy
\begin{eqnarray}\label{thanks iso}
P(Z_j)\!\ge\! P(B^{(|E_j\setminus\CC_{r,(\a,\b)}^\nu|)})\ge
P(B^{(v_j)})-C(n)\,r^n\,(\b-\a)\,\l_{j}^{-1}\,.
\end{eqnarray}
Since for a.e. $r>R'$ we have
\[
P(Z_j)=P(E_j;\R^{n+1}\!\setminus\!\CC_{r,(\a,\b)}^\nu)+P(G_r;\CC_{r,(\a,b)}^\nu)+\H^n\big((E_j\Delta G_r)\cap\pa \CC_{r,(\a,b)}^\nu\big)\,,
\]
we conclude that
\begin{eqnarray*}
\psi_W(v_j)-P(B^{(v_j)})\!\!\!\!&=&\!\!\!P(E_j;\CC_{r,(\a,\b)}^\nu\!\setminus\! W)\!+P(E_j;\R^{n+1}\!\setminus\!\CC_{r,(\a,\b)}^\nu)\!-\!P(B^{(v_j)})
\\
&=&\!\!\!P(E_j;\CC_{r,(\a,\b)}^\nu\setminus W)+P(Z_j)-P(B^{(v_j)})
\\
&&\!\!\!-P(G_r;\CC_{r,(\a,b)}^\nu)-\H^n\big((E_j\Delta G_r)\cap\pa \CC_{r,(\a,b)}^\nu\big)
\end{eqnarray*}
so that $E_j\to F$ in $L^1_{\rm loc}(\R^{n+1})$ and \eqref{thanks iso} give, for a.e. $r>R'$,
\begin{eqnarray*}
&&\!\!\!\!\!\!\!\!\!\varliminf_{j\to\infty}\psi_W(v_j)-P(B^{(v_j)})\ge P(F;\CC_{r,(\a,\b)}^\nu\!\setminus\! W)-P(G_r;\CC_{r,(\a,b)}^\nu)
\\
&&-\H^n\big((F\Delta G_r)\cap\pa \CC_{r,(\a,\b)}^\nu\big)\ge P(F;\CC_r^\nu\!\setminus\! W)-P(G_r;\CC_r^\nu)-C(n,W)/r\,,
\end{eqnarray*}
thanks to \eqref{b2} and $(F\Delta G_r)\cap\pa\CC_r^\nu=(F\Delta G_r)\cap\pa\CC_{r,(\a,\b)}^\nu$. Letting $r\to\infty$, recalling \eqref{coo4}, and by $(F',\nu)\in\F$, we find $\varliminf_{j\to\infty}\psi_W(v_j)-P(B^{(v_j)})\ge-{\rm res}_W(F',\nu)\ge-\Rr(W)$. This completes the proof of \eqref{lower bound equation}, which in turn, combined with \eqref{energy upper bound}, gives \eqref{main asymptotic expansion}, and also shows that $L^1_{\rm loc}$-subsequential limits $F$ of $E_j\in{\rm Min}[\psi_W(v_j)]$ for $v_j\to\infty$ are such that, for some $\nu\in\SS^n$, $(F\cup W,\nu)\in\F$ and $F'=F\cup W\in{\rm Max}[\Rr(W)]$.

\noindent {\bf Step four:} Moving towards the proof of \eqref{f of Ev}, we prove the validity, uniformly among varifolds associated to maximizers of $\Rr(W)$, of estimates analogous to \eqref{blowdown hp check 1} and \eqref{blowdown hp check 2}.
For a constant $\Gamma>2\,n\,\om_n$ to be determined later on (see \eqref{blowdown hp check 1 Ej part two}, \eqref{choose Gamma 1}, and \eqref{choose Gamma 2} below) in dependence of $n$ and $W$, and for $\s>0$, we let $M_0=M_0(n,2\,\Gamma,\s)$ and $\e_0=\e_0(n,2\,\Gamma,\s)$ be determined by Theorem \ref{theorem mesoscale criterion}. If $(F,\nu)\in{\rm Max}[\Rr(W)]$, then by Theorem \ref{thm main of residue}-(ii) we can find $R_2=R_2(W)>0$, $f\in C^\infty(\nu^\perp)$ such that
\begin{equation}
  \label{basta}
  (\pa F)\setminus\CC_{R_2}^\nu=\big\{x+f(x)\,\nu:x\in\nu^\perp\,,|x|>R_2\big\}\,,
\end{equation}
and such that \eqref{asymptotics of F} holds with  $\max\{|a|,|b|,|c|\}\le C(W)$ and $|\nabla f(x)|\le C_0/|x|^{n-1}$ for $|x|>R_2$. Thus $\|\nabla f\|_{C^0(\nu^\perp\setminus\DD_r^\nu)}\to 0$ as $r\to\infty$ uniformly on $(F,\nu)\in{\rm Max}[\Rr(W)]$, and there is $R_3>\max\{2\,R_2,1\}$ (depending on $W$) such that, if
$V_F=\var((\pa F)\setminus B_{R_3},1)$, then $V_F\in\V_n(0,R_3,\infty)$, and
\begin{equation}
  \label{blowdown hp check 1 four}
  \|{\rm bd}_{V_F}\|(\pa B_{R_3})\le \Gamma\,R_3^{n-1}\,,\qquad \|V_F\|(B_\rho\setminus B_{R_3})\le\Gamma\,\rho^n\qquad\forall\rho>R_3\,,
\end{equation}
(compare with \eqref{blowdown hp check 1}). Then, arguing as in step three-(c), or more simply by exploiting \eqref{basta} and the decay estimates \eqref{asymptotics of F}, we see that there is $L>\max\{M_0,64\}$, depending on $n$, $W$ and $\s$ only, such that, setting
\begin{equation}
  \label{def of sWs}
  s_W(\s)=2\,L\,R_3
\end{equation}
we have for some $c(n)>0$ (compare with \eqref{blowdown hp check 2})
\begin{eqnarray}
\label{blowdown hp check 2 four}
  |\de_{V_F,R_3,0}(s_W(\s)/8)|\le\e_0/2\,,\qquad \inf_{r>s_W(\s)/8}\de_{V_F,R_3,0}(r)\ge-\e_0/2\,.
\end{eqnarray}
\noindent {\bf Step five:} Given $E_j\in{\rm Min}[\psi_W(v_j)]$ for $v_j\to\infty$, we prove the existence of $(F,\nu)\in{\rm Max}[\Rr(W)]$ and $h_j\in C^\infty((\pa F)\setminus B_{R_2})$ such that
\begin{eqnarray}
  \label{f of Ev jj}
  &&(\pa E_j)\cap A^{R_1\,\l_{j}}_{4\,R_2}
  =\Big\{y+h_j(y)\,\nu_F(y):y\in\pa F\Big\}\cap A^{R_1\,\l_{j}}_{4\,R_2}\,,
  \\
  \label{aspetta}
  &&\lim_{j\to\infty}\|h_j\|_{C^1((\pa F)\cap A_{4\,R_2}^M)}=0\,,\qquad\forall M<\infty\,;
\end{eqnarray}
and that if $x_j$ satisfies $|E_j\Delta B^{(v_j)}(x_j)|=\inf_x|E_j\Delta B^{(v_j)}(x)|$, then
\begin{equation}
  \label{aspetta ancora}
  \lim_{j\to\infty}||x_j|^{-1}\,x_j-\nu|=0\,;
\end{equation}
finally, we prove statement (iii) (i.e., $(\pa E_j)\setminus B_{R_2}$ is diffeomorphic to an $n$-dimensional disk). By step three, there is $(F,\nu)\in{\rm Max}[\Rr(W)]$ such that, up to extracting subsequences, \eqref{convergence of Ej to F last proof} holds. By \eqref{convergence of Ej to F last proof} and \eqref{basta}, and with $s_W(\s)$ defined as in step four (see \eqref{def of sWs}) starting from $F$, we can apply Remark \ref{remark improved convergence} to find $f_j\in C^\infty(\nu^\perp)$ such that
\begin{equation}
  \label{basta j}
  (\pa E_j)\cap A^{s_W(\s)}_{2\,R_2}=\big\{x+f_j(x)\,\nu:x\in\nu^\perp\big\}\cap A^{s_W(\s)}_{2\,R_2}\,,
\end{equation}
for $j$ large enough (in terms of $\s$, $n$, $W$, and $F$), and such that $f_j\to f$ in $C^1(\DD_{s_W(\s)}^\nu\setminus\DD_{2\,R_2}^\nu)$.
With $R_3$ as in step four and with the goal of applying Theorem \ref{theorem mesoscale criterion} to the varifolds
$V_j=\var((\pa E_j)\setminus B_{R_3},1)$, we notice that $V_j\in\V_n(\Lambda_j,R_3,\infty)$, for some $\Lambda_j\le \Lambda_0\,\l_{j}^{-1}$ (thanks to \eqref{uniform lambda minimality}). In particular, by \eqref{def of sWs}, $s_W(\s)$ satisfies the ``mesoscale bounds'' (compare with \eqref{mesoscale bounds})
\begin{equation}
  \label{mesoscale bounds end}
  \e_0\,(4\,\Lambda_j)^{-1}>s_W(\s)>\max\{M_0,64\}\,R_3
\end{equation}
provided $j$ is large. Moreover, by $R_3>2\,R_2$ and $s_W(\s)/8>2\,R_2$, by \eqref{basta}, \eqref{basta j} and $f_j\to f$ in $C^1$, we exploit \eqref{blowdown hp check 1 four} and \eqref{blowdown hp check 2 four} to deduce
\begin{eqnarray}
  \label{blowdown hp check 1 Ej part one}
  \|{\rm bd}_{V_j}\|(\pa B_{R_3})&\le&(2\,\Gamma)\,R_3^{n-1}\,,
  \\
  \label{blowdown hp check 2 Ej part one}
  |\de_{V_j,R_3,0}(s_W(\s)/8)|&\le&(2/3)\,\e_0\,.
\end{eqnarray}
We claim that, up to increasing $\Gamma$ (depending on $n$ and $W$), we can entail
\begin{equation}
  \label{blowdown hp check 1 Ej part two}
 \|V_j\|(B_\rho\setminus B_{R_3}) \le\Gamma\,\rho^n\,,\qquad\forall \rho>R_3\,.
\end{equation}
Indeed, by Theorem \ref{thm existence and uniform min}-(i), for some positive $\Lambda_0$ and $s_0$ depending on $W$ only,  $E_j$ is a $(\Lambda_0\,\l_{j}^{-1},s_0\,\l_{j})$-perimeter minimizer with free boundary in $\Om$. Comparing $E_j$ to $E_j\setminus B_r$ by  \eqref{uniform lambda minimality}, for every $r<s_0\,\l_{j}$,
\begin{equation}
  \label{choose Gamma 1}
  P(E_j;\Om\cap B_r)\le C(n)\,\big(r^n+\Lambda_0\,\l_{j}^{-1}\,r^{n+1}\big)\le C(n,W)\,r^n\,;
\end{equation}
since, at the same time, if $r>s_0\,\l_{j}$, then
\begin{equation}
  \label{choose Gamma 2}
  P(E_j;\Om\cap B_r)\le P(E_j;\Om)=\psi_W(v_j)\le P(B^{(v_j)})\le C(n)\,s_0^{-n}\,r^n\,,
\end{equation}
by combining \eqref{choose Gamma 1} and \eqref{choose Gamma 2} we find \eqref{blowdown hp check 1 Ej part two}. With \eqref{blowdown hp check 1 Ej part one} and \eqref{blowdown hp check 1 Ej part two} at hand, we can also show that
\begin{equation}
\label{blowdown hp check 2 Ej part two}
  |\de_{V_j,R_3,\Lambda_j}(s_W(\s)/8)|\le\e_0\,.
\end{equation}
Indeed, by $s_W(\s)=2\,L\,R_3$ and by $\Lambda_j\le \Lambda_0\,\l_{j}^{-1}$,
\begin{eqnarray*}
  &&\big|\de_{V_j,R_3,\Lambda_j}(s_W(\s)/8)-\de_{V_j,R_3,0}(s_W(\s)/8)\big|
  \\
  &&\le(\Lambda_0/\l_{j})\,\int_{R_3}^{s_W(\s)/8}\!\!\!\!\rho^{-n}\,\|V_j\|(B_\rho\setminus B_{R_3})\,d\rho
  \le\frac{\Lambda_0\,R_3\,\Gamma}{\l_{j}}\,\Big(\frac{L}{4}-1\Big)\le\frac{\e_0}3\,,
\end{eqnarray*}
provided $j$ is large enough. To complete checking that Theorem \ref{theorem mesoscale criterion} can be applied to every $V_j$ with $j$ large enough, we now consider the quantities
\[
R_{*j}=\sup\big\{\rho>s_W(\s)/8:\de_{V_j,R_3,\Lambda_j}(\rho)\ge-\e_0\big\}\,,
\]
and prove that, for a constant $\tau_0$ depending on $n$ and $W$ only, we have
\begin{equation}
  \label{Rstar true}
  R_{*j}\ge \tau_0\,\l_{j}\,;
\end{equation}
in particular, provided $j$ is large enough, \eqref{Rstar true} implies immediately
\begin{equation}
  \label{Rstar check}
  R_{*j}\ge 4\,s_W(\s)\,,
\end{equation}
which was the last assumption in Theorem \ref{theorem mesoscale criterion} that needed to be checked. To prove \eqref{Rstar true}, we pick $\tau_0$ such that
\begin{equation}
  \label{geometry of B1}
  \big|\tau_0^{-n}\,\H^n\big(B_{\tau_0}(z)\cap\pa B^{(1)}\big)-\om_n\big|\le \e_0/2\,,\qquad\forall z\in\pa B^{(1)}\,.
\end{equation}
(Of course this condition only requires $\tau_0$ to depend on $n$.) By definition of $x_j$ and by \eqref{limsupmax goes to zero take 2}, and up to extracting a subsequence, we have $x_j\to z_0$ for some $z_0\in\pa B^{(1)}$. In particular, setting $\rho_j=\tau_0\,\l_{j}$, we find
\begin{eqnarray*}
  \rho_j^{-n}\,\|V_j\|(B_{\rho_j}\setminus B_{R_3})\!\!\!&=&\!\!\!\tau_0^{-n}\,P\big((E_j-x_j)/\l_{j}\,;\,B_{\tau_0}(-x_j)\setminus B_{R_3/\rho_j}(-x_j)\big)
  \\
  &&\to\tau_0^{-n}\,\H^n\big(B_{\tau_0}(-z_0)\cap\pa B^{(1)}\big)\le \om_n+(\e_0/2)\,,
\end{eqnarray*}
thus proving that, for $j$ large enough,
\begin{eqnarray*}
  &&\!\!\!\!\!\!\!\!\!\de_{V_j,R_3,\Lambda_j}(\rho_j)\!\!\ge\!\!-\frac{\e_0}2+\frac{1}{n\,\rho_j^n}\,\int\,x\cdot\nu^{\rm co}_{V_j}\,d\,{\rm bd}_{V_j}-\Lambda_j\,\int_{R_3}^{\rho_j}\frac{\|V_j\|(B_{\rho}\setminus B_{R_3})}{\rho^n}\,d\rho
  \\
  &&\ge\!\!-\frac{\e_0}2-\frac{2\,\Gamma\,R_3^n}{n\,\tau_0^n\,\l_{j}}
  -\Lambda_0\,\Gamma\,\frac{(\rho_j-R_3)}{\l_{j}}
  \ge-\frac{\e_0}2-\frac{C_*(n,W)}{\tau_0^n\,\l_{j}}-C_{**}(n,W)\,\tau_0\,,
\end{eqnarray*}
where we have used \eqref{blowdown hp check 1 Ej part one}, $\spt\,{\rm bd}_{V_j}\subset\pa B_{R_3}$, and \eqref{blowdown hp check 1 Ej part two}. Therefore, provided we pick $\tau_0$ depending on $n$ and $W$ so that $C_{**}\,\tau_0\le\e_0/4$, and then we pick $j$ large enough to entail $(C_*(n,W)/\tau_0^n)\l_{j}^{-1}\le\e_0/4$, we conclude that if $r\in(R_3,\rho_j]$, then
$\de_{V_j,R_3,\Lambda_j}(r)\ge  \de_{V_j,R_3,\Lambda_j}(\rho_j)\ge-\e_0$,
where in the first inequality we have used Theorem \ref{theorem 7.17AA exterior lambda}-(i) and the fact that $V_j\in\V_n(\Lambda_j,R_3,\infty)$. In summary, by \eqref{blowdown hp check 1 Ej part one} and \eqref{blowdown hp check 1 Ej part two} (which give \eqref{Gamma bounds}), by \eqref{mesoscale bounds end} (which gives \eqref{mesoscale bounds} with $s=s_W(\s)/8$), and by \eqref{blowdown hp check 2 Ej part two} and \eqref{Rstar check} (which imply, respectively, \eqref{mesoscale delta small s8} and \eqref{mesoscale Rstar larger s4}) we see that Theorem \ref{theorem mesoscale criterion}-(i) can be applied with $V=V_j$ and $s=s_W(\s)/8$ provided $j$ is large in terms of $\s$, $n$, $W$ and the limit $F$ of the $E_j$'s. Thus, setting
\begin{equation}\label{def of Starj}
S_{*j}=\min\big\{R_{*j},\e_0/\Lambda_j\big\}\,,
\end{equation}
and noticing that by \eqref{Rstar true} and $\Lambda_j\le \Lambda_0\,\l_{j}^{-1}$ we have
\begin{equation}
  \label{Starj big}
  S_{*j}\ge 16\,R_1\,\l_{j}\,,
\end{equation}
(for $R_1$ depending on $n$ and $W$ only) we conclude that, for $j$ large, there are $K_j\in\H$ and $u_j\in\X_\s(\S_{K_j},\s_W(\s)/32,R_1\,\l_{j})$, such that
\begin{equation}
  \label{par Ej}
  (\pa E_j)\cap A_{s_W(\s)/32}^{R_1\,\l_{j}}=\S_{K_j}\big(u_j,s_W(\s)/32,R_1\,\l_{j}\big)\,.
\end{equation}
%
%
%
Similarly, by \eqref{blowdown hp check 1 four} and \eqref{blowdown hp check 2 four}, thanks to Theorem \ref{theorem mesoscale criterion}-(ii) we have
\begin{equation}
  \label{par F}
  (\pa F)\cap \big(\R^{n+1}\setminus B_{s_W(\s)/32}\big)=\Sigma_{\nu^\perp}\big(u,s_W(\s)/32,\infty\big)\,,
\end{equation}
for $u\in\X_{\s'}(\S_{\nu^\perp},s_W(\s)/32,\infty)$ for every $\s'>\s$. Now, by $E_j\to F$ in $L^1_{\rm loc}(\R^{n+1})$, \eqref{par Ej} and \eqref{par F} can hold only if $|\nu_{K_j}-\nu|\le \zeta(\s)$ for a function $\zeta$, depending on $n$ and $W$ only, such that $\zeta(\s)\to 0$ as $\s\to 0^+$. In particular (denoting by $\s_0^*$, $\e_0^*$ and $C_0^*$ the dimension dependent constants originally introduced in Lemma \ref{lemma step one} as $\s_0$, $\e_0$ and $C_0$) we can find $\s_1=\s_1(n,W)\le\s_0^*$ such that if $\s<\s_1$, then $\e_0^*\ge \zeta(\s)\ge |\nu_{K_j}-\nu|$, and correspondingly, Lemma \ref{lemma step one}-(i) can be used to infer the existence of
$u_j^*\in\X_{C_0\,(\s+\zeta(\s))}(\S_{\nu^\perp},s_W(\s)/32,2\,R_1\,\l_{j})$ such that, for $j$ large,
\begin{eqnarray}
\nonumber
  \S_{\nu^\perp}\big(u_j^*,s_W(\s)/32,2\,R_1\,\l_{j}\big)&=&\S_{K_j}\big(u_j,s_W(\s)/32,2\,R_1\,\l_{j}\big)
  \\
  &=&(\pa E_j)\cap A_{s_W(\s)/32}^{2\,R_1\,\l_{j}}\,.
    \label{star 2 jj}
\end{eqnarray}
By \eqref{basta j} and Lemma \ref{lemma D1}, \eqref{star 2 jj} implies cylindrical graphicality: more precisely, provided $\s_1$ is small enough, there are $g_j\in C^1(\nu^\perp)$ such that
\begin{eqnarray}
  \label{what about gj 1}
  &&\sup_{x\in\nu^\perp}\{|g_j(x)|\,|x|^{-1},|\nabla g_j(x)|\}\le C\,\big(\s+\zeta(\s)\big)\,,
  \\
  \label{what about gj 2}
  &&(\pa E_j)\cap A^{R_1\,\l_{j}}_{2\,R_2}
  =\big\{x+g_j(x)\,\nu:x\in\nu^\perp\big\}\cap A^{R_1\,\l_{j}}_{2\,R_2}\,.
\end{eqnarray}
At the same time, by \eqref{basta}, \eqref{asymptotics of F}, and up to further increasing $R_2$ and decreasing $\s_1$, we can exploit Lemma \ref{lemma D2} in the appendix to find  $h_j\in C^1(G(f))$, $G(f)=\{x+f(x)\,\nu:x\in\nu^\perp\}$, such that
\[
\big\{x+g_j(x)\,\nu:x\in\nu^\perp\big\}\setminus B_{4\,R_2}
=\big\{z+h_j(z)\,\nu_F(z):z\in G(f)\big\}\setminus B_{4\,R_2}
\]
which, combined with \eqref{basta} and \eqref{what about gj 2} shows that
\begin{equation}\nonumber
  (\pa E_j)\cap A^{R_1\,\l_{j}}_{4\,R_2}
  =\big\{z+h_j(z)\,\nu_F(z):z\in\pa F\big\}\cap A^{R_1\,\l_{j}}_{4\,R_2}
\end{equation}
that is \eqref{f of Ev jj}. By $E_j\to F$ in $L^1_{\rm loc}(\R^{n+1})$, we find $h_j\to 0$ in $L^1((\pa F)\cap A_{4\,R_2}^M)$ for every $M<\infty$, so that, by elliptic regularity, \eqref{aspetta} follows. We now recall that, by Theorem \ref{thm existence and uniform min}-(ii), $(\pa E_j)\setminus B_{R_0(v_j)\,\l_{j}}$ coincides with
\begin{eqnarray}\nonumber
&&\!\!\!\!\!\!\!\!\big\{y+ \l_{j}w_j\big((y-x_j)/\l_{j}\big)\,\nu_{B^{(v_j)}(x_j)}(y):y\in\pa B^{(v_j)}(x_j)\big\}\setminus B_{R_0(v_j)\,\l_{j}}
\\\label{x and u of Ev take 2 jjj}
&&\!\!\!\!\!\!\!\!\mbox{for $\|w_j\|_{C^1(\pa B^{(1)})}\to 0$ and $R_0(v_j)\to 0$}\,.
\end{eqnarray}
The overlapping of \eqref{what about gj 2} and \eqref{x and u of Ev take 2 jjj} (i.e., the fact that $R_0(v_j)<R_1$ if $j$ is large enough) implies statement (iii). Finally, combining \eqref{what about gj 1} and \eqref{what about gj 2} with \eqref{x and u of Ev take 2 jjj} and $\|w_j\|_{C^1(\pa B^{(1)})}\to 0$ we deduce the validity of \eqref{aspetta ancora}. More precisely, rescaling by $\l_j$ in \eqref{what about gj 1} and \eqref{what about gj 2} and setting $E_j^*=E_j/\l_j$, we find $g_j^*\in C^1(\nu^\perp)$ such that, for every $j\ge j_0(\s)$ and  $\s<\s_1$,
\begin{eqnarray}
  \label{what about gj 1 star}
  &&\sup_{x\in\nu^\perp}\{|g_j^*(x)||x|^{-1},|\nabla g_j^*(x)|\}\le C\,\big(\s+\zeta(\s)\big)\,,
  \\
  \label{what about gj 2 star}
  &&(\pa E_j^*)\cap A^{R_1}_{2\,R_2/\l_j}=\big\{x+g_j^*(x)\,\nu:x\in\nu^\perp\big\}\cap A^{R_1}_{2\,R_2/\l_j}\,,
\end{eqnarray}
while rescaling by $\l_j$ in \eqref{x and u of Ev take 2 jjj} and setting $z_j=x_j/\l_j$ we find
\begin{eqnarray}
  \label{x and u of Ev take 2 jjjjj}
 &&\!\!\!\!\!\!\!\!\!\!\!\!\!\!\!\!\!\!\!\!(\pa E_j^*)\setminus B_{R_0(v_j)}
 \!\!=\!\big\{z_j+z+w_j(z)\,\nu_{B^{(1)}}(z)::y\in\pa B^{(1)}(z_j)\big\}\!\setminus \!B_{R_0(v_j)}
\end{eqnarray}
where $||z_j|-\om_{n+1}^{1/(n+1)}|\to 0$ thanks to \eqref{limsupmax goes to zero take 2}. Up to subsequences, $z_j\to z_0$, where $|z_0|=\om_{n+1}^{1/(n+1)}$. Should $z_0\ne|z_0|\,\nu$, then picking $\s$ small enough in terms of $|\nu-(z_0/|z_0|)|>0$ and picking $j$ large enough, we would then be able to exploit \eqref{what about gj 1 star} to get a contradiction with $\|w_j\|_{C^1(\pa B^{(1)})}\to 0$.

\noindent {\bf Conclusion}: Theorem \ref{thm existence and uniform min} implies  Theorem \ref{thm main psi}-(i), and \eqref{main asymptotic expansion} was proved in step three. Should Theorem \ref{thm main psi}-(ii), (iii), or (iv) fail, then we could find a sequence $\{(E_j,v_j)\}_j$ contradicting the conclusions of either step five or Theorem \ref{thm existence and uniform min}. We have thus completed the proof of Theorem \ref{thm main psi}.
\end{proof}

\appendix

\section{Proof of Theorem \ref{theorem 7.15AA main estimate lambda}}\label{appendix proof of expansion mean curvature} We assume $H\in\H$, $\Lambda\ge0$, $\eta_0>\eta>0$,  $(r_1,r_2)$ and $(r_3,r_4)$ are $(\eta,\eta_0)$-related as in \eqref{r1r2r3r4}, and $u\in\X_{\s}(\Sigma_H,r_1,r_2)$ is such that $\S_H(u,r_1,r_2)$ has mean curvature bounded by $\Lambda$ in $A_{r_1}^{r_2}$. We want to find $\s_0$ and $C_0$, depending on $n$, $\eta_0$, and $\eta$ only, such that, if $\max\{1,\Lambda\,r_2\}\,\s\le\s_0$, then
\begin{equation}
  \label{lambdaAA 7.14 proof}
  \Big|\H^n(\Sigma_H(u,r_3,r_4))-\H^n(\Sigma_H(0,r_3,r_4))\Big|\le C_0\,\int_{\Sigma_H\times(r_1,r_2)}\!\!\!\!\!\!\!\!\!\!\!\!\!\!\!\!\!r^{n-1}\,\big(u^2+\Lambda\,r\,|u|\big);
\end{equation}
and such that, if there is $r\in(r_1,r_2)$ s.t. $E_{\Sigma_H}^0[u_r]=0$ on $\Sigma_H$, then
\begin{equation}
\label{lambdaAA 7.15(5) proof}
 \int_{\Sigma_H\times(r_3,r_4)} \!\!\!\!\!\!\!\!\!r^{n-1}\,u^2 \le C(n)\,\Lambda\,r_2\,(r_2^n-r_1^n)+ C_0\,\int_{\Sigma_H\times(r_1,r_2)}\!\!\!\!\!\!\!\!\!r^{n-1}\,(r\,\pa_ru)^2 \,.
\end{equation}
We make three preliminary considerations: {\bf (i):} By \cite[4.5(8)]{allardalmgrenRADIAL}
\begin{eqnarray}\label{second order area radial}
&&\Big|\H^n(\S_H(u,r_1,r_2))-\H^n(\S_H(0,r_1,r_2))
\\
\nonumber
&&\hspace{3cm}-\frac12\int_{\S_H\times(r_1,r_2)}\!\!\!\!\!\!\!\!\!r^{n-1}\,\big(|\nabla^{\Sigma_H} u|^2+(r\,\pa_ru)^2-(n-1)\,u^2\big)\Big|
\\
\nonumber
&&\hspace{3cm}\le C(n)\,\s\,\int_{\S_H\times(r_1,r_2)}\!\!\!\!\!\!\!\!\!r^{n-1}\big(u^2+|\nabla^{\Sigma_H} u|^2+(r\,\pa_ru)^2\big)\,.
\end{eqnarray}
Similarly, by the last displayed formula on \cite[Page 236]{allardalmgrenRADIAL} and by \cite[Lemma 4.9(1)]{allardalmgrenRADIAL}, if $\vphi=\psi^2\,w$, $w\in C^1(\S_H\times(r_1,r_2))$ and $\psi\in C^1(r_1,r_2)$, then
\begin{eqnarray}\label{second order first variation radial}
  &&\Big|\frac{d}{dt}\bigg|_{t=0}\!\!\!\!\!\H^n(\S_H(u+t\,\vphi,r_1,r_2))
  \\\nonumber
  &&\hspace{0.5cm}-\int_{\Sigma_H\times(r_1,r_2)}\!\!\!\!\!\!\!\!\!\!r^{n-1}\,\Big\{\nabla^{\Sigma_H}u\cdot\nabla^{\S_H}\vphi+(r\,\pa_ru)\,(r\,\pa_r \vphi)-(n-1)\,u\,\vphi\Big\}\Big|
  \\\nonumber
  &\le& C(n)\,\s\,\int_{\Sigma_H\times(r_1,r_2)}\!\!\!\!\!\!\!\!\!\!r^{n-1}\,\psi^2\,\big\{|\nabla^{\Sigma_H}u|^2+|\nabla^{\Sigma_H}w|^2
  +(r\,\pa_ru)^2+(r\,\pa_rw)^2
  \\\nonumber
  &&\hspace{5.5cm}+\,u^2+w^2+(r\,\psi')^2\,w^2\big\}\,,
\end{eqnarray}
which is the second order expansion of the first variation of the area at $\S_H(u,r_1,r_2)$ along outer variations in spherical coordinates of the form $\vphi=\psi^2\,w$, $\psi=\psi(r)$. {\bf (ii):} For the sake of brevity, given $\zeta:(r_1,r_2)\to\R$ a radial function, $u,v:\Sigma_H\times(r_1,r_2)\to\R$, $X,Y:\Sigma_H\times(r_1,r_2)\to\R^m$, we set
\begin{eqnarray*}
Q_\zeta(u,v)=\int_{\Sigma_H\times(r_1,r_2)}\!\!\!\!\!\!\!\!\!\!\!\!r^{n-1}\,\zeta(r)^2\,u\,v\,,\quad Q_\zeta(X,Y)=\int_{\Sigma_H\times(r_1,r_2)}\!\!\!\!\!\!\!\!\!\!\!\!r^{n-1}\,\zeta(r)^2\,X\cdot Y\,,
\end{eqnarray*}
and $Q_\zeta(u)=Q_\zeta(u,u)$, $Q_\zeta(X)=Q_\zeta(X,X)$. {\bf (iii):} The following two estimates (whose elementary proof is contained in \cite[Lemma 7.13]{allardalmgrenRADIAL}) hold: whenever $v\in C^1(\Sigma_H\times (r_1,r_2))$, we have
\begin{eqnarray}\label{elementary 1}
  \int_{\S_H\times(r_1,r_2)}\!\!\!\!\!\!\!\!\!\!r^{n-1}\,v^2\le C(n,\eta,\eta_0)\,\Big\{  \int_{\S_H\times(r_1,r_2)}\!\!\!\!\!\!\!\!\!\!r^{n-1}\,(r\,\pa_rv)^2
  +  \int_{\S_H\times(r_3,r_4)}\!\!\!\!\!\!\!\!\!\!\!\!r^{n-1}\,v^2\Big\},
\end{eqnarray}
and, provided there is $r\in[r_1,r_2]$ such that $v_r=0$ on $\S_H$, we have
\begin{equation}\label{elementary 2}
    \int_{\S_H\times(r_1,r_2)}r^{n-1}\,v^2\le C(n,\eta_0)\,  \int_{\S_H\times(r_1,r_2)}r^{n-1}\,(r\,\pa_rv)^2\,.
\end{equation}
We are now ready for the proof. Compared to \cite[Chapter 4]{allardalmgrenRADIAL}, the main difference is that we replace \cite[Lemma 4.10]{allardalmgrenRADIAL} with \eqref{going back to psi w plug in}.

\noindent {\bf Step one:} We prove that there is $h:\Sigma_H\times(r_1,r_2)\to[-\Lambda,\Lambda]$ such that for every $w\in C^1(\S_H\times(r_1,r_2))$ and $\psi\in C^1(r_1,r_2)$ we have
\begin{eqnarray}\nonumber
&&\Big|T_\psi(u,w)-\int_{\Sigma_H\times(r_1,r_2)}\!\!\!\!\!\!\!\!\!\!\!r^n\,\psi^2\,w\,h\Big|
\le C(n)\s_0\big(Q_\psi(u)+Q_\psi(w)+Q_\psi(\nabla^{\Sigma_H} u)
\\\label{going back to psi w plug in}
&&\hspace{2cm}+Q_\psi(\nabla^{\Sigma_H} w)+Q_{r\,\psi}(\pa_ru)+Q_{r\,\psi}(\pa_rw)+Q_{r\,\psi'}(w)\big)\,.
\end{eqnarray}
where $T_\psi(u,w)=Q_\psi(\nabla^{\Sigma_H} u,\nabla^{\Sigma_H} w)+Q_{r}\big(\pa_ru,\pa_r[\psi^2w]\big)-(n-1)\,Q_\psi(u,w)$.
We start rewriting \eqref{second order first variation radial} as
\begin{eqnarray}\label{going back to psi w}
&&\Big|T_\psi(u,w)-\frac{d}{dt}\Big|_{t=0}\,\H^n(\S(u+t\,\psi^2\,w,r_1,r_2))\Big|
\\\nonumber
  &&\le C(n)\,\s\,\big(Q_\psi(u)+Q_\psi(w)+Q_\psi(\nabla^{\Sigma_H} u)+Q_\psi(\nabla^{\Sigma_H} w)
\\\nonumber
&&\hspace{5.5cm}+Q_{r\,\psi}(\pa_ru)+Q_{r\,\psi}(\pa_rw)+Q_{r\,\psi'}(w)\big)\,.
\end{eqnarray}
If $F_{u+t\,\vphi}:{\Sigma_H}\times(r_1,r_2)\to{\Sigma_H}(u+t\,\vphi,r_1,r_2)$, $\vphi=\psi^2\,w$, is given by
\[
F_{u+t\,\vphi}(\om,r)=r\,\frac{\om+(u(\om,r)+t\vphi(\om,r))\,\nu_H}{\sqrt{1+(u(\om,r)+t\vphi(\om,r))^2}}\,,
\]
then $\{\Phi_t=F_{u+t\,\vphi}\circ (F_u)^{-1}\}_{t\in[0,1]}$ are diffeomorphisms on ${\Sigma_H}(u,r_1,r_2)$, with $\Phi_t(\S_H(u,r_1,r_2))=\Sigma_H(u+t\,\vphi,r_1,r_2)$ and $\dot{\Phi}_0=(d/dt)_{t=0}\Phi_t$. Since $\S_H(u,r_1,r_2)$ has mean curvature bounded by $\Lambda$ in $A_{r_1}^{r_2}$, for some bounded function $h:\Sigma_H\times(r_1,r_2)\to[-\Lambda,\Lambda]$ we have
\begin{eqnarray*}
  \frac{d}{dt}\Big|_{t=0}\,\H^n(\S(u+t\vphi,r_1,r_2))
  &=&\Lambda\,\int_{\S_H(u,r_1,r_2)}\!\!\!\!\!\!\!\!\!\!\!h(F_u^{-1})\,\dot{\Phi}_0\cdot\nu_{\S_H(u,r_1,r_2)}
  \\
  &=&\Lambda\,\int_{\Sigma_H\times(r_1,r_2)}\!\!\!\!\!\!\!\!\!\!\!h\,\dot{\Phi}_0\big(F_u\big)\cdot\star\big(\pa_r F_u\wedge\,\wedge_{i=1}^{n-1}\pa_i F_u\big)\,,
\end{eqnarray*}
where $\pa_i=\nabla_{\tau_i}$ for a local orthonormal frame $\{\tau_i\}_{i=1}^{n-1}$ in $\Sigma_H$, and where $\star$ is the Hodge star-operator (so that $\star\,(v_1\wedge v_2...\wedge v_n)$ is a normal vector to the hyperplane spanned by the $v_i$s, with length equal to the $n$-dimensional volume of the parallelogram defined by the $v_i$s, and whose orientation depends on the ordering of the $v_i$s themselves). We can compute the initial velocity $\dot{\Phi}_0$ of $\{\Phi_t\}_{t\in[0,1]}$ by noticing that
$\Phi_t\big(F_u(\om,r)\big)=r\,(1+(u+t\,\vphi)^2)^{-1/2}\,(\om+(u+t\,\vphi)\,\nu_H)$, so that,
\begin{eqnarray*}
  \dot\Phi_0(F_u)&=&\frac{d}{dt}\Big|_{t=0}\,r\,\frac{\om+(u+t\,\vphi)\,\nu_H}{\sqrt{1+(u+t\,\vphi)^2}}
  =r\,\frac{-u\,\vphi\,\om+\vphi\,\nu_H}{(1+u^2)^{3/2}}
  \\
  &=&r\,\big(-u\,\vphi\,\om+\vphi\,\nu_H)+r\,\s\,{\rm O}\big(\psi^2\,(u^2+w^2)\big)\,.
\end{eqnarray*}
At the same time
\begin{eqnarray*}
  \pa_rF_u\!\!&=&\!\!\!\frac{\om+u\,\nu_H}{\sqrt{1+u^2}}\!+r\,\pa_r\Big(\frac{\om+u\,\nu_H}{\sqrt{1+u^2}}\Big)
  \!\!=\!\frac{\om+u\,\nu_H}{\sqrt{1+u^2}}\!-\frac{r\,u\,\pa_ru\,\om}{(1+u^2)^{3/2}}
  +\!\frac{r\,\pa_r u\,\nu_H}{(1+u^2)^{3/2}}
  \\
  &=&\!\!\!\big(1-(u^2/2)-u\,r\,\pa_ru\big)\,\om+\big(u+r\,\pa_ru\big)\,\nu_H+\s\,{\rm O}(u^2+(r\pa_r u)^2)
  \\
  &=&\!\!\!A\,\om+B\,\nu_H+\s\,{\rm O}(u^2+(r\pa_r u)^2)\,,
\\
  \frac{\pa_iF_u}r\!\!&=&\!\!\!\pa_i\Big(\frac{\om+u\,\nu_H}{\sqrt{1+u^2}}\Big)=\frac{\tau_i}{\sqrt{1+u^2}}-\frac{u\,\pa_iu}{(1+u^2)^{3/2}}\,\om
  +\frac{\pa_iu}{(1+u^2)^{3/2}}\,\nu_H
  \\
  &=&\big(1-(u^2/2)\big)\,\tau_i-u\,\pa_iu\,\om+\pa_iu\,\nu_H+\s\,{\rm O}(u^2+(\pa_iu)^2)
  \\
  &=& C\,\tau_i+ E_i\,\om+F_i\,\nu_H+\s\,{\rm O}(u^2+(\pa_iu)^2)
\end{eqnarray*}
so that, with $\Xi=\wedge_{i=1}^{n-1}\tau_i$, $\hat\tau_i=\wedge_{j\ne i}\tau_j$, and $P(u)^2=u^2+|\nabla^{\Sigma_H} u|^2+(r\pa_r u)^2$,
\begin{eqnarray*}
  &&\!\!\!\!\!\!\!\!\frac{\pa_r F_u\wedge\wedge_{i=1}^{n-1}\pa_iF_u}{r^{n-1}}
  \!=\!\big(A\,\om+B\,\nu_H\big)\!\wedge\!\wedge_{i=1}^{n-1}\big(C\,\tau_i+E_i\,\om+F_i\,\nu_H\big)
  \!+\!\s{\rm O}\big(P(u)^2\big)
  \\
  &&=A\,C^{n-1}\,\om\wedge\Xi+B\,C^{n-1}\,\nu_H\wedge\Xi\,
  +G_i\,\big(\om\wedge\nu_H\wedge\hat\tau_i\big)+\s{\rm O}\big(P(u)^2\big),
\end{eqnarray*}
for a coefficient $G_i$ which we do not need to compute. Indeed, $\star(\om\wedge\nu_H\wedge\hat\tau_i)$, being parallel to $\tau_i$, is orthogonal to $\om$ and $\nu_H$, so that
\begin{eqnarray*}
  &&r^{-n}\,\dot{\Phi}_0\big(F_u(r,\om)\big)\cdot\star\big(\pa_r F_u\wedge\,\wedge_{i=1}^{n-1}\pa_i F_u\big)
  \\
  &=&\!\!\!\!\big[\big(-u\,\vphi\,\om+\vphi\,\nu_H)+\s\,{\rm O}\big(\psi^2(u^2+v^2)\big)\big]\cdot
  \\
  &&\hspace{4cm}\cdot\big[A\,C^{n-1}\,\nu_H-B\,C^{n-1}\,\om
  +\s\,{\rm O}\big(P(u)^2\big)\big]
  \\
  &=&\!\!\!\!C^{n-1}\!\big[\big(1-\frac{u^2}2-u\,r\,\pa_ru\big)\,\vphi+(u+r\,\pa_ru)\,u\,\vphi\big]
  \!\!+\s\,{\rm O}\big(\psi^2\,(w^2+P(u)^2)\big)
  \\
  &=&\!\!\!\!\vphi+\s\,{\rm O}\big(\psi^2\,(w^2+P(u)^2)\big)
\end{eqnarray*}
In particular, since $|h|\le\Lambda$,
\begin{eqnarray*}
  &&\!\!\!\!\!\!\!\!\frac{d}{dt}\Big|_{t=0}\,\H^n(\S(u+t\vphi,r_1,r_2))
  =\int_{\S\times(r_1,r_2)}\!\!\!\!\!\!\!\!h\,\,\dot{\Phi}_0\big(F_u\big)\cdot\star\big(\pa_r F_u\wedge\,\wedge_{i=1}^{n-1}\pa_i F_u\big)
  \\
  &=&\!\!\!\!\!\!\int_{{\Sigma_H}\times(r_1,r_2)}\!\!\!\!\!\!\!\!\!\!\!\!\!\!\!r^n\,\psi^2w\,h+
  \s\Lambda r_2{\rm O}\big(Q_\psi(u)+Q_\psi(w)+Q_\psi(\nabla^{\Sigma_H} u)+Q_{r\,\psi}(\pa_ru)\big)\,.
\end{eqnarray*}
Plugging this estimate into \eqref{going back to psi w}, and  by $\max\{1,\Lambda\,r_2\}\,\s\le\s_0$, we find \eqref{going back to psi w plug in}.

\noindent {\bf Step two:} We prove that
\begin{eqnarray}\label{lemma 411(1)}
  Q_\psi(\nabla^{\Sigma_H} u)+Q_{r\,\psi}(\pa_ru)\le Q_\psi(|u|,\Lambda\,r)+C(n)\,\big(Q_{\psi}(u)+Q_{r\,\psi'}(u)\big)\,.
\end{eqnarray}
By $Q_r(\psi\,\pa_ru,\psi'u)\le Q_{r\,\psi}(\pa_ru)/4+C\,Q_{r\,\psi'}(u)$ and by \eqref{going back to psi w plug in}$_{w=u}$ we find
\begin{eqnarray*}\label{proof of 411(1)}
&&Q_\psi(\nabla^{\Sigma_H} u)+Q_{r\,\psi}(\pa_ru)\le
Q_\psi(|u|,\Lambda\,r)
+C(n)\,\big(Q_\psi(u)+Q_{r\,\psi'}(u)\big)
  \\\nonumber
  &&+C(n)\,\s_0\,\big(Q_\psi(u)+Q_{r\,\psi'}(u)+Q_\psi(\nabla^{\Sigma_H} u)+Q_{r\,\psi}(\pa_ru)\big)\,.
\end{eqnarray*}
which implies \eqref{lemma 411(1)} provided $\s_0$ is small enough.

\noindent {\bf Step three:} We prove that, if $w:\S_H\times(r_1,r_2)\to\R$ is slice-wise orthogonal to $u-w$, in the sense that $\int_{\Sigma_H}w_r\,(u_r-w_r)=0$, $\int_{\Sigma_H}\pa_rw_r\,(\pa_ru_r-\pa_rw_r)=0$, and $\int_{\Sigma_H}\nabla^{\Sigma_H} w_r\cdot(\nabla^{\Sigma_H} u_r-\nabla^{\Sigma_H} w_r)=0$ for every $r\in(r_1,r_2)$, then
\begin{equation}
  \label{lemma 411(2)}
  |T_\psi(u,w)|\le Q_\psi(|w|,\Lambda\,r)+C(n)\,\s_0\,\big(Q_\psi(u)+Q_{r\,\psi'}(u)+Q_\psi(|u|,\Lambda\,r)\big)\,.
\end{equation}
Indeed, by slice-wise orthogonality, we find that $Q_\zeta(w)\le Q_\zeta(u)$, $Q_\zeta(\pa_rw)\le Q_\zeta(\pa_ru)$ and $Q_\zeta(\nabla^{\Sigma_H} w)\le Q_\zeta(\nabla^{\Sigma_H} u)$ whenever $\zeta:(r_1,r_2)\to\R$ is radial. Therefore \eqref{going back to psi w plug in} gives
$|T_\psi(u,w)|\le Q_\psi(|w|,\Lambda\,r)+C(n)\,\s_0\,R_\psi(u)$, with $R_\psi(u)=Q_\psi(u)+Q_\psi(\nabla^{\Sigma_H} u)+Q_{r\,\psi}(\pa_ru)+Q_{r\,\psi'}(u)$. Combining this with \eqref{lemma 411(1)} we get \eqref{lemma 411(2)}.

\noindent {\bf Step four:} We prove \eqref{lambdaAA 7.14 proof}. Let now $\psi$ be a cut-off function between $(r_3,r_4)$ and $(r_1,r_2)$, so that with $Z_\psi(u)=Q_\psi(\nabla^{\S_H}u)+Q_\psi(u)+Q_{r\,\psi}(\pa_ru)$,
\begin{eqnarray*}
  \Big|\int_{\S_H\times(r_3,r_4)}r^{n-1}\big\{|\nabla^{\Sigma_H}u|^2-(n-1)\,u^2+(r\,\pa_ru)^2\big\}\Big|
  \le Z_\psi(u)\,.
\end{eqnarray*}
If $A(u)=\H^n(\S_H(u,r_3,r_4))-\H^n(\S_H(0,r_3,r_4))$, then by \eqref{second order area radial} with $(r_3,r_4)$ in place of $(r_1,r_2)$, we find
\begin{eqnarray*}
&&|A(u)|\le Z_\psi(u)+C(n)\,\s\,Z_\psi(u)\le Q_\psi(|u|,\Lambda\,r)+C(n)\,\big(Q_{\psi}(u)+Q_{r\,\psi'}(u)\big)
\\
&&+C(n)\,\s\,\big\{Q_\psi(u)+Q_\psi(|u|,\Lambda\,r)+C(n)\,\big(Q_{\psi}(u)+Q_{r\,\psi'}(u)\big)\big\}\,,
\end{eqnarray*}
where in the last inequality we have used  \eqref{lemma 411(1)}. We deduce
\begin{equation}\label{zephyr}
  |A(u)|\le C(n)\,\big(Q_\psi(|u|,\Lambda\,r)+Q_{\psi}(u)+Q_{r\,\psi'}(u)\big)\,,
\end{equation}
and \eqref{lambdaAA 7.14 proof} follows (with $C_0=C_0(n,\eta_0,\eta)$ by the properties of $\psi$).

\noindent {\bf Step five:} We finally prove that, if $E^0_{\Sigma_H}[u_{r_*}]=0$ for some $r_*\in(r_1,r_2)$, then \eqref{lambdaAA 7.15(5) proof} holds, that is
\begin{equation}
  \label{main estimate}
  \int_{\Sigma_H\times(r_3,r_4)}\!\!\!\!\!\!\!\!\!\!\!\!\!\!r^{n-1}\,u^2\le C(n)\,\Lambda\,r_2\,(r_2^n-r_1^n)+
 C(n,\eta_0,\eta)\,\int_{\S\times(r_1,r_2)}\!\!\!\!\!\!\!\!\!\!\!\!\!\!r^{n-1}\,(r\,\pa_ru)^2\,.
\end{equation}
Define $u^+,u^-,u^0:\Sigma_H\times(r_1,r_2)\to\R$ by setting, for $r\in(r_1,r_2)$,
$(u^+)_r=E^+_{\Sigma_H}[u_r]$, $(u^-)_r=E^-_{\Sigma_H}[u_r]$ and $(u^0)_r=E^0_{\Sigma_H}[u_r]$,
where $E^\pm_{\Sigma_H}$ denote the $L^2(\Sigma_H)$-orthogonal projections on the spaces of positive/negative eigenvectors of the Jacobi operator of $\Sigma_H$, and where $E^0_{\Sigma_H}$ is the $L^2(\Sigma_H)$-orthogonal projection onto the space of the Jacobi fields of $\Sigma_H$. Since $(u^0)_{r_*}=0$, we can directly apply \eqref{elementary 2} with $v=u^0$ and deduce that
\begin{equation}
  \label{deh 0}
  \int_{\Sigma_H\times(r_1,r_2)}r^{n-1}\,(u^0)^2\le C(n,\eta_0)\,\int_{\Sigma_H\times(r_1,r_2)}r^{n-1}\,(r\,\pa_ru^0)^2\,.
\end{equation}
By the orthogonality relations between $u^0_r$, $u^+_r$ and $u^-_r$ we have that
\begin{eqnarray}\label{deh 1}
&&\!\!\!\!\!\!\!\!\!\!\!\!\!\!\!\!\int_{\Sigma_H\times(r_3,r_4)}\!\!\!\!\!\!\!r^{n-1}\,u^2=\int_{\Sigma_H\times(a,b)}\!\!\!\!\!\!\!r^{n-1}\,\Big((u^0)^2+(u^+)^2+(u^-)^2\Big)
    \\\label{deh 2}
  &&\!\!\!\!\!\!\!\!\!\!\!\!\!\!\!\!\int_{\Sigma_H\times(r_1,r_2)}\!\!\!\!\!\!\!\!\!\!r^{n+1}\,(\pa_ru)^2=
  \int_{\Sigma_H\times(r_1,r_2)}\!\!\!\!\!\!\!\!\!\!r^{n+1}\,\big((\pa_ru^0)^2+(\pa_ru^+)^2+(\pa_ru^-)^2\big).
\end{eqnarray}
By the spectral theorem, for every $r\in(r_1,r_2)$ we have
$
C_1(n)^{-1}\,\int_{\Sigma_H}(u^-)_r^2\le\,\int_{\Sigma_H}\,(n-1)\,(u^-)_r^2-|\nabla^{\Sigma_H} (u^-)_r|^2\,,
$
which, multiplied by $r^{n-1}\,\psi^2$, gives
\begin{eqnarray*}
C_1(n)^{-1}\,Q_\psi(u^-)&\le& (n-1)\,Q_\psi(u^-)-Q_\psi(\nabla^{\Sigma_H}u^-)
\\
&=&(n-1)\,Q_\psi(u^-,u)-Q_\psi(\nabla^{\Sigma_H}u^-,\nabla^{\Sigma_H}u)
\\
&=&-T_\psi(u^-,u)+Q_{r}(\pa_ru,\pa_r(\psi^2\,u^-))\,,
\end{eqnarray*}
where in the second to last identity we have used that $w=u^-$ is slice-wise orthogonal to $w-u$; in particular, by \eqref{lemma 411(2)} with $w=u^-$, we find
\begin{eqnarray}\nonumber
&& C_1(n)^{-1}\,Q_\psi(u^-)\le
 Q_\psi(|u^-|,\Lambda\,r)+Q_{r}(\pa_ru,\pa_r(\psi^2\,u^-))\\\label{cookies}
 &&+C(n)\,\s_0\,\big(Q_\psi(u)+Q_{r\,\psi'}(u)+Q_\psi(|u|,\Lambda\,r)\big)
 \,.
\end{eqnarray}
Again by slice-wise orthogonality of $w=u^-$ to $w-u$, we have
\begin{eqnarray*}
&&Q_{r}(\pa_ru,\pa_r(\psi^2\,u^-))=Q_{r}(\pa_ru^-,\pa_r(\psi^2\,u^-))=
  Q_{r\,\psi}(\pa_ru^-)
  \\
  &&+2\,Q_r(\psi'\,\pa_ru^-,\psi\,u^-)\le Q_{r\,\psi}(\pa_ru^-)+\frac{Q_\psi(u^-)}{2\,C_1(n)}+C(n)\,Q_{r\,\psi'}(\pa_r u^-)\,,
\end{eqnarray*}
which combined into \eqref{cookies} gives
\begin{eqnarray}\nonumber
&&(2\,C_1(n))^{-1}\,Q_\psi(u^-)\le
 Q_\psi(|u^-|,\Lambda\,r)+Q_{r\,\psi}(\pa_ru^-)+C(n)\,Q_{r\,\psi'}(\pa_r u^-)
 \\\nonumber
 &&\hspace{3.5cm}+C(n)\,\s_0\,\big(Q_\psi(u)+Q_{r\,\psi'}(u)+Q_\psi(|u|,\Lambda,r)\big)\,.
\end{eqnarray}
Using H\"older inequality again we have
\begin{eqnarray*}
&&Q_\psi(|u^-|,\Lambda\,r)\le\frac{Q_\psi(u^-)}{4\,C_1(n)}+C(n)\,\Lambda\,r_2\,(r_2^n-r_1^n)\,,
\\
&&Q_\psi(|u|,\Lambda\,r)\le 2\,Q_\psi(u)+C(n)\,\Lambda\,r_2\,(r_2^n-r_1^n)\,,
\\\nonumber
\mbox{so that}\!\!\!&&\frac{Q_\psi(u^-)}{4\,C_1(n)}\le Q_{r\,\psi}(\pa_ru^-)+C(n)\big(Q_{r\,\psi'}(\pa_r u^-)+\Lambda\,r_2\,(r_2^n-r_1^n)\big)
 \\\nonumber
 &&\hspace{2cm}+C(n)\,\s_0\,\big(Q_\psi(u)+Q_{r\,\psi'}(u)+\Lambda\,r_2\,(r_2^n-r_1^n)\big)
\end{eqnarray*}
Taking $\psi$ to be a cut-off function between $(r_3,r_4)$ and $(r_1,r_2)$, we find
\begin{eqnarray}\label{deh}
&&\int_{\S\times(r_3,r_4)}\!\!\!\!\!\!r^{n-1}\,(u^-)^2\le C(n)\,\Lambda\,r_2\,(r_2^n-r_1^n)
\\\nonumber
&&+C(n,\eta_0,\eta)\,\Big\{
\int_{\S\times(r_1,r_2)}\!\!\!\!\!\!r^{n-1}\,(r\,\pa_ru^-)^2
+\s_0\,\int_{\S\times(r_1,r_2)}\!\!\!\!\!\!r^{n-1}\,u^2\Big\}\,.
\end{eqnarray}
By combining \eqref{deh 0}, \eqref{deh}, and the analogous estimate to \eqref{deh} for $u^+$ with \eqref{deh 1} and \eqref{deh 2} we find that \eqref{deh} holds with $u$ in place of $u^-$; this latter estimate,
thanks to \eqref{elementary 1},
finally gives \eqref{main estimate}.

\section{Spherical and cylindrical graphs}\label{appendix spherical cylindrical} We state here for the reader's convenience two technical lemmas concerning spherical and cylindrical graphs. They are both used in the last step of the proof of Theorem \ref{thm main psi}. The elementary proofs are omitted.

\begin{lemma}[Spherical graphs as cylindrical graphs]\label{lemma D1}
  There are dimension independent positive constants $C$ and $\eta_0$ with the following property. If $n\ge 1$, $H\in\H$ and $u\in\X_\eta(\S_H,r_1,r_2)$ with $\eta<\eta_0$, then we have
  \[
  \DD_{(1-C\,\eta^2)\,r_2}^{\nu_H}\setminus \DD_{r_1}^{\nu_H}\subset\pp_H\big(\S_H(u,r_1,r_2)\big)\subset
  \DD_{r_2}^{\nu_H}\setminus \DD_{(1-C\,\eta^2)\,r_1}^{\nu_H}\,,
  \]
  and there is $g\in C^1(H)$ such that $\sup\big\{|x|^{-1}\,|g(x)|+|\nabla g(x)|:x\in H\big\}\le C\,\eta$ and
  $\S_H(u,r_1,r_2)=\big\{x+g(x)\,\nu_H:x\in\pp_H\big(\S_H(u,r_1,r_2)\big)\big\}$.
  Moreover, if $(\rho_1,\rho_2)\subset ((1+C\,\eta)\,r_1,(1-C\,\eta^2)\,r_2)$, then
  $\S_H(u,\rho_1,\rho_2)=\big\{x+g(x)\,\nu_H:x\in H\big\}\cap A_{\rho_1}^{\rho_2}$.
\end{lemma}

\begin{lemma}\label{lemma D2}
  There is $\eta\in(0,1)$ with the following property. If $H\in\H$, $R>1$, $f\in C^2(H)$, and $g\in C^1(H)$ are such that
  \begin{equation}
    \label{l growth f}
  \max\big\{|f(x)|\,,\,\,|x|\,|\nabla f(x)|\,,\,\,|x|\,|\nabla^2f(x)|:x\in H\,,|x|>R\big\}< \eta\,,
  \end{equation}
  \begin{equation}
    \label{l growth g}
      \max\big\{|x|^{-1}\,|g(x)|,|\nabla g(x)|:x\in H\big\}< \eta\,,
  \end{equation}
  then there is $h\in C^1(G_H(f))$ such that
  \begin{eqnarray}\label{hhh 1}
  &&G_H(g)\setminus B_{4\,R}=\big\{z+h(z)\,\nu_{f}(z):z\in G_H(f)\big\}\setminus B_{4\,R}\,,
  \end{eqnarray}
  where $G_H(f)=\{x+f(x)\,\nu_H:x\in H\}$ and, for $z=x+f(x)\,\nu_H$, we have set $\nu_f(z)=(1+|\nabla f(x)|^2)^{-1/2}\,(-\nabla f(x)+\nu_H)$ .
\end{lemma}

\section{Obstacles with zero isoperimetric residue}\label{appendix R zero}

\begin{proposition}\label{prop RW zero}
  If $W$ is compact and $\Rr(W)=0$, then $\psi_W(v)-P(B^{(v)})\to 0$ as $v\to\infty$ and $W$ is purely $\H^n$-unrectifiable, in the sense that $W$ cannot contain an $\H^n$-rectifiable set of $\H^n$-positive measure. In a partial converse, if $W$ is purely $\H^n$-unrectifiable and $\H^n(W)<\infty$, then $\Rr(W)=0$.
\end{proposition}

\begin{proof} {\bf Step one:} Let $\Rr(W)=0$. Comparing with balls, $\varlimsup_{v\to \infty} \psi_W(v) - P(B^{(v)}) \leq 0 = \mathcal{R}(W)$. To prove the matching lower bound, we argue by contradiction and consider $E_j\in{\rm Min}[\psi_W(v_j)]$ with $v_j \to \infty$ such that
\begin{equation}\label{liminf assumption appendix}
    \varliminf_{v\to \infty} \psi_W(v) - P(B^{(v)})=\lim_{j\to \infty}P(E_j;\Omega) - P(B^{(v_j)})< 0\,.
\end{equation}
With \eqref{liminf assumption appendix} replacing $\Rr(W)>0$, one can repeat {\it verbatim} step two-(a) of the proof of Theorem \ref{thm main of residue}; we thus derive the asymptotic expansion for $F$ as in step two-(c), which is then the key fact used in step three to derive that
$\lim_{j\to \infty} P(E_j;\Omega) - P(B^{(v_j)}) \geq - \mathrm{res}_W(F\cup W,\nu)\geq-\mathcal{R}(W)$;
the latter inequality is of course in contradiction with \eqref{liminf assumption appendix} if $\Rr(W)=0$. Next, arguing again by contradiction, we assume the existence of an $\H^n$-rectifiable set $S$ with $\H^n(W\cap S)>0$. By \cite[Lemma 11.1]{SimonLN}, without loss of generality, $S$ is a $C^1$-embedded hypersurface in $\R^{n+1}$. Let $x$ be a point of tangential differentiability for $W\cap S$, so that $\H^n(W\cap S\cap B_\rho(x))=\om_n\,\rho^n+{\rm o}_x(\rho^n)$ as $\rho\to 0^+$. Since $S$ is a $C^1$-embedded hypersurface, there is $\nu\in\SS^n$ such that for every $\e>0$ there is $\rho_*=\rho_*(x,\e)>0$ with
$S\cap\CC_{\rho_*,\rho_*}^\nu(x)=\{y+g(y)\,\nu:y\in\DD_{\rho_*}^\nu(x)\}$, where $g\in C^1(x+\nu^\perp)$ with $g(x)=0$ and $\Lip(g)\le \e$. Denoting $G(g)=\{y+g(y)\,\nu:y\in (x+\nu^\perp)\}$, and up to decrease $\rho_*$, we can entail
\begin{eqnarray}\label{science}
\H^n\big(G(g)\cap W\cap \CC_{\rho_*}^\nu(x)\big)\ge\H^n(W\cap S\cap B_{\rho_*}(x))\ge (1-\e)\,\om_n\,\rho_*^n\,.
\end{eqnarray}
Since $|g|\le\e\,\rho_*$ on $\pa \DD_{\rho_*}^\nu(x)$, we can define $f:(x+\nu^\perp)\to\R$ so that $f=g$ on $\DD_{\rho_*}^\nu(x)$, $f=0$ on $(x+\nu^\perp)\setminus\DD_{2\,\rho_*}^\nu(x)$, and $\Lip(f)\le\e$. Denoting by $F$ the epigraph of $f$, we have that $(F,\nu)\in\F$ and we compute, for $R$ large enough to entail $\CC_{2\,\rho_*}^\nu(x)\cup W\cc\CC_R^\nu$,
\begin{eqnarray*}
\om_n\,R^n-P(F;\CC_R^\nu\setminus W)\!\!&\ge&\!\!\om_n\,(2{\rho_*})^n-P(F;\CC_{2\,{\rho_*}}^\nu(x)\setminus W)
  \\
  \!\!&=&\!\!\int_{\DD_{2\,{\rho_*}}^\nu(x)}\!\!\!1-\sqrt{1+|\nabla f|^2}+P(F;\CC_{2\,{\rho_*}}^\nu(x)\cap W)
  \\
  \!\!&\ge&\!\!-\om_n\,(2\,{\rho_*})^n\,\e^2+(1-\e)\,\om_n\,{\rho_*}^n
\end{eqnarray*}
where we have used $f=0$ on $\nu^\perp\setminus\DD_{2\,{\rho_*}}^\nu(x)$, \eqref{science} and $\sqrt{1+\e^2}\le 1+\e^2$. Up to taking $\e<\e(n)$, we thus find ${\rm res}_W(F,\nu)>0$, and thus deduce $\Rr(W)>0$.

\noindent {\bf Step two:} Let $W$ be purely $\H^n$-unrectifiable with $\H^n(W)<\infty$, and let $(F,\nu)\in{\rm Max}[\Rr(W)]$. Since $F$ is a local perimeter minimizer in $\Om$, $F$ is open in $\Omega$ with $\Om\cap\pa F=\cl(\pa^*F)$ ($\pa^*F=$ the reduced boundary of $F$ as a set of locally finite perimeter in $\Omega$). Now, $\om_n\,R^n-P(F;\CC_R^\nu\setminus W)$ is decreasing towards $\Rr(W)\ge\Ss(W)\ge0$, therefore $P(F;\CC_R^\nu\setminus W)<\infty$ for every $R$. In particular, $\H^n\llcorner(\Om\cap\pa F)$ is a Radon measure on $\R^{n+1}$. Now, $\pa F\subset (\Om\cap\pa F)\cup W$, so that $\H^n(W)<\infty$ implies that $\H^n\llcorner\pa F$ is a Radon measure on $\R^{n+1}$ and, since $F$ is open, that $F$ is a set of finite perimeter in $\R^{n+1}$ by \cite[Theorem 4.5.11]{FedererBOOK}. The pure $\H^n$-unrectifiability of $W$ gives $P(F;\CC_R^\nu\setminus W)=P(F;\CC_R^\nu)$, where $P(F;\CC_R^\nu)\ge\om_n\,R^n$  by \eqref{def Sigma nu 1} and \eqref{def Sigma nu 2}, and thus $\Rr(W)={\rm res}_W(F,\nu)\le 0$. This proves $\Rr(W)=0$.
\end{proof}

\medskip

{\small \noindent {\bf Data availibility:} Data sharing not applicable to this article as no datasets were
generated or analyzed during the current study.

\noindent {\bf Conflict of interest:} The authors have no conflict of interest to declare that are
relevant to the content of this article.}

\medskip

\bibliographystyle{alpha}
\bibliography{references}

\newcommand{\etalchar}[1]{$^{#1}$}
\begin{thebibliography}{DMMN18}

\bibitem[AA81]{allardalmgrenRADIAL}
W.~K. Allard and F.~J. Almgren, Jr.
\newblock On the radial behavior of minimal surfaces and the uniqueness of
  their tangent cones.
\newblock {\em Ann. of Math. (2)}, 113(2):215--265, 1981.

\bibitem[Ale62]{alexandrov}
A.~D. Alexandrov.
\newblock A characteristic property of spheres.
\newblock {\em Ann. Mat. Pura Appl. (4)}, 58:303--315, 1962.

\bibitem[All72]{Allard}
W.~K. Allard.
\newblock On the first variation of a varifold.
\newblock {\em Ann. Math.}, 95:417--491, 1972.

\bibitem[CGR07]{choeghomiritore}
J.~Choe, M.~Ghomi, and M.~Ritor\'{e}.
\newblock The relative isoperimetric inequality outside convex domains in
  {${\bf R}^n$}.
\newblock {\em Calc. Var. Partial Differential Equations}, 29(4):421--429,
  2007.

\bibitem[Cha01]{ChavelBOOK}
I.~Chavel.
\newblock {\em Isoperimetric inequalities}, volume 145 of {\em Cambridge Tracts
  in Mathematics}.
\newblock Cambridge University Press, Cambridge, 2001.
\newblock Differential geometric and analytic perspectives.

\bibitem[CL12]{CicaleseL.ardi}
M.~Cicalese and G.~P. L.ardi.
\newblock A selection principle for the sharp quantitative isoperimetric
  inequality.
\newblock {\em Arch. Rat. Mech. Anal.}, 206(2):617--643, 2012.

\bibitem[CLM16]{CiLeMaIC1}
M.~Cicalese, G.~P. L.ardi, and F.~Maggi.
\newblock Improved convergence theorems for bubble clusters {I}. {T}he planar
  case.
\newblock {\em Indiana Univ. Math. J.}, 65(6):1979--2050, 2016.

\bibitem[CM17]{ciraolomaggi2017}
G.~Ciraolo and F.~Maggi.
\newblock On the shape of compact hypersurfaces with almost-constant mean
  curvature.
\newblock {\em Comm. Pure Appl. Math.}, 70(4):665--716, 2017.

\bibitem[CPS20]{cagnettiperuginistoger}
F.~Cagnetti, M.~Perugini, and D.~St\"{o}ger.
\newblock Rigidity for perimeter inequality under spherical symmetrisation.
\newblock {\em Calc. Var. Partial Differential Equations}, 59(4):Paper No 139,
  53, 2020.

\bibitem[DM19]{delgadinomaggi}
M.~G. Delgadino and F.~Maggi.
\newblock Alexandrov's theorem revisited.
\newblock {\em Anal. PDE}, 12(6):1613--1642, 2019.

\bibitem[DMMN18]{delgadinomaggimihailaneumayer}
M.~G. Delgadino, F.~Maggi, C.~Mihaila, and R.~Neumayer.
\newblock Bubbling with {$L^2$}-almost constant mean curvature and an
  {A}lexandrov-type theorem for crystals.
\newblock {\em Arch. Ration. Mech. Anal.}, 230(3):1131--1177, 2018.

\bibitem[DPM15]{dephilippismaggiCAP-ARMA}
G.~De~Philippis and F.~Maggi.
\newblock Regularity of free boundaries in anisotropic capillarity problems and
  the validity of {Y}oung's law.
\newblock {\em Arch. Ration. Mech. Anal.}, 216(2):473--568, 2015.

\bibitem[DPM17]{dephilippismaggiCAP-CRELLE}
G.~De~Philippis and F.~Maggi.
\newblock Dimensional estimates for singular sets in geometric variational
  problems with free boundaries.
\newblock {\em J. Reine Angew. Math.}, 725:217--234, 2017.

\bibitem[EM13]{eichmair_metzgerINV}
M.~Eichmair and J.~Metzger.
\newblock Unique isoperimetric foliations of asymptotically flat manifolds in
  all dimensions.
\newblock {\em Invent. Math.}, 194(3):591--630, 2013.

\bibitem[ESV19]{EngelSpolaVelichGandT}
M.~Engelstein, L.~Spolaor, and B.~Velichkov.
\newblock ({L}og-)epiperimetric inequality and regularity over smooth cones for
  almost area-minimizing currents.
\newblock {\em Geom. Topol.}, 23(1):513--540, 2019.

\bibitem[Fed69]{FedererBOOK}
H.~Federer.
\newblock {\em Geometric measure theory}, volume 153 of {\em Die Grundlehren
  der mathematischen Wissenschaften}.
\newblock Springer-Verlag New York Inc., New York, 1969.

\bibitem[FFM{\etalchar{+}}15]{F2M3}
A.~Figalli, N.~Fusco, F.~Maggi, V.~Millot, and M.~Morini.
\newblock Isoperimetry and stability properties of balls with respect to
  nonlocal energies.
\newblock {\em Comm. Math. Phys.}, 336(1):441--507, 2015.

\bibitem[Fin86]{FinnBOOK}
R.~Finn.
\newblock {\em Equilibrium Capillary Surfaces}, volume 284 of {\em Die
  Grundlehren der mathematischen Wissenschaften}.
\newblock Springer-Verlag New York Inc., New York, 1986.

\bibitem[FM11]{FigalliMaggiARMA}
A.~Figalli and F.~Maggi.
\newblock On the shape of liquid drops and crystals in the small mass regime.
\newblock {\em Arch. Rat. Mech. Anal.}, 201:143--207, 2011.

\bibitem[FM21]{fuscomoriniCONVEX}
N.~Fusco and M.~Morini.
\newblock Total positive curvature and the equality case in the relative
  isoperimetric inequality outside convex domains.
\newblock 2021.

\bibitem[FMP08]{fuscomaggipratelli}
N.~Fusco, F.~Maggi, and A.~Pratelli.
\newblock The sharp quantitative isoperimetric inequality.
\newblock {\em Ann. Math.}, 168:941--980, 2008.

\bibitem[FMP10]{FigalliMaggiPratelliINVENTIONES}
A~Figalli, F.~Maggi, and A.~Pratelli.
\newblock A mass transportation approach to quantitative isoperimetric
  inequalities.
\newblock {\em Inv. Math.}, 182(1):167--211, 2010.

\bibitem[GJ86]{gruterjost}
M.~Gr{\"u}ter and J.~Jost.
\newblock Allard type regularity results for varifolds with free boundaries.
\newblock {\em Ann. Scuola Norm. Sup. Pisa Cl. Sci. (4)}, 13(1):129--169, 1986.

\bibitem[Gr{\"u}87]{gruter}
M.~Gr{\"u}ter.
\newblock Boundary regularity for solutions of a partitioning problem.
\newblock {\em Arch. Rat. Mech. Anal.}, 97:261--270, 1987.

\bibitem[LSW63]{LSWannSNS}
W.~Littman, G.~Stampacchia, and H.~F. Weinberger.
\newblock Regular points for elliptic equations with discontinuous
  coefficients.
\newblock {\em Ann. Scuola Norm. Sup. Pisa Cl. Sci. (3)}, 17:43--77, 1963.

\bibitem[Mag12]{maggiBOOK}
F.~Maggi.
\newblock {\em Sets of finite perimeter and geometric variational problems: an
  introduction to Geometric Measure Theory}, volume 135 of {\em Cambridge
  Studies in Advanced Mathematics}.
\newblock Cambridge University Press, 2012.

\bibitem[Maz11]{mazyaBOOKSobolevPDE}
V.~Maz'ya.
\newblock {\em Sobolev spaces with applications to elliptic partial
  differential equations}, volume 342 of {\em Grundlehren der mathematischen
  Wissenschaften}.
\newblock Springer, Heidelberg, 2011.

\bibitem[MPP14]{maggiponsiglionepratelli}
F.~Maggi, Marcello Ponsiglione, and Aldo Pratelli.
\newblock Quantitative stability in the isodiametric inequality via the
  isoperimetric inequality.
\newblock {\em Trans. Amer. Math. Soc.}, 366(3):1141--1160, 2014.

\bibitem[NV20]{nabervaltortaJEMS}
A.~Naber and D.~Valtorta.
\newblock The singular structure and regularity of stationary varifolds.
\newblock {\em J. Eur. Math. Soc. (JEMS)}, 22(10):3305--3382, 2020.

\bibitem[PR02]{PerezRos}
J.~P\'{e}rez and A.~Ros.
\newblock Properly embedded minimal surfaces with finite total curvature.
\newblock In {\em The global theory of minimal surfaces in flat spaces
  ({M}artina {F}ranca, 1999)}, volume 1775 of {\em Lecture Notes in Math.},
  pages 15--66. Springer, Berlin, 2002.

\bibitem[Sch83]{Scho83}
R.~M. Schoen.
\newblock Uniqueness, symmetry, and embeddedness of minimal surfaces.
\newblock {\em J. Differential Geom.}, 18(4):791--809 (1984), 1983.

\bibitem[Sim83a]{Simon83}
L.~Simon.
\newblock Asymptotics for a class of nonlinear evolution equations, with
  applications to geometric problems.
\newblock {\em Ann. of Math. (2)}, 118(3):525--571, 1983.

\bibitem[Sim83b]{SimonLN}
L.~Simon.
\newblock {\em Lectures on geometric measure theory}, volume~3 of {\em
  Proceedings of the Centre for Mathematical Analysis}.
\newblock Australian National University, Centre for Mathematical Analysis,
  Canberra, 1983.

\bibitem[Sim85]{SimonMontecatini}
L.~Simon.
\newblock Isolated singularities of extrema of geometric variational problems.
\newblock In {\em Harmonic mappings and minimal immersions ({M}ontecatini,
  1984)}, volume 1161 of {\em Lecture Notes in Math.}, pages 206--277.
  Springer, Berlin, 1985.

\bibitem[Sim87]{SimonAIHP}
L.~Simon.
\newblock Asymptotic behaviour of minimal graphs over exterior domains.
\newblock {\em Ann. Inst. H. Poincar\'e Anal. Non Lin\'eaire}, 4(no.
  3):231--242, 1987.

\bibitem[Sim96]{simonETH}
L.~Simon.
\newblock {\em Theorems on regularity and singularity of energy minimizing
  maps}.
\newblock Lectures in Mathematics ETH Z\"{u}rich. Birkh\"{a}user Verlag, Basel,
  1996.
\newblock Based on lecture notes by Norbert Hungerb\"{u}hler.

\end{thebibliography}
\end{document}